\documentclass{amsart}
\usepackage{amssymb}
\usepackage{amscd}
\usepackage{epsfig}
\usepackage{graphicx}

\usepackage{pgf,tikz}
\usepackage{braids}
\usepackage{caption}

\usetikzlibrary{arrows}
\usetikzlibrary{decorations.pathreplacing}

\title[Galois action on knots II]
{Galois action on knots II: \\
Proalgebraic string links and knots}

\author{Hidekazu Furusho}

\address{Graduate School of Mathematics, Nagoya University, 
Chikusa-ku, Furo-cho, Nagoya, 464-8602,  Japan}

\email{furusho@math.nagoya-u.ac.jp}

\subjclass[2010]{Primary 16W25; Secondary 11M32, 20F36, 57M25}
\keywords{}

\date{August 21, 2017}

\newtheorem{thm}{Theorem}[section]
\newtheorem{lem}[thm]{Lemma}
\newtheorem{cor}[thm]{Corollary}
\newtheorem{prop}[thm]{Proposition}  

\theoremstyle{remark}

\theoremstyle{definition}
\newtheorem{defn}[thm]{Definition}
\newtheorem{rem}[thm]{Remark}
\newtheorem{nota}[thm]{Notation}     

\newtheorem{eg}[thm]{Example}       
\newtheorem{conj}[thm]{Conjecture}    
   
\newtheorem{prob}[thm]{Problem}

\numberwithin{equation}{section}
\numberwithin{figure}{section}

\newcommand{\creation}{\rotatebox[origin=c]{180}{$\curvearrowleft$}}
\newcommand{\opcreation}{\rotatebox[origin=c]{180}{$\curvearrowright$}}
\newcommand{\annihilation}{\curvearrowright}
\newcommand{\opannihilation}{\curvearrowleft}

\def\orientedcircle{\unitlength.2ex
 \begin{minipage}{8\unitlength}
   \begin{tikzpicture}
     \draw[->] (0,0.125) arc (90:450:0.125);
   \end{tikzpicture}
 \end{minipage}
}

\def\diaCrossP{\unitlength.2ex
  \begin{minipage}{15\unitlength}
    \begin{picture}(15,15)
      \put(0,0){\vector(1,1){15}}
      \qbezier(15,0)(15,0)(10,5)
      \qbezier(5,10)(0,15)(0,15)
      \put(0,15){\vector(-1,1){0}}
    \end{picture}
  \end{minipage}
}
\def\diaCrossN{\unitlength.2ex
  \begin{minipage}{15\unitlength}
    \begin{picture}(15,15)
      \put(15,0){\vector(-1,1){15}}
      \qbezier(0,0)(0,0)(5,5)
      \qbezier(10,10)(15,15)(15,15)
      \put(15,15){\vector(1,1){0}}
    \end{picture}
  \end{minipage}
}

\def\diaProjected{\unitlength.1em
  \begin{minipage}{14\unitlength}
    \begin{picture}(14,14)
      \put(0,0){\vector(1,1){14}}
      \put(14,0){\vector(-1,1){14}}
    \end{picture}
  \end{minipage}
}
\begin{document}
\bibliographystyle{amsalpha+}
\maketitle

\begin{abstract}
We discuss an action of the Grothendieck-Teichm\"{u}ller proalgebraic group 
on the linear span of proalgebraic tangles, oriented tangles completed by a filtration of Vassiliev.
The action yields a motivic structure on tangles.
We derive distinguished properties of the action particularly on
proalgebraic string links and on proalgebraic knots which can not be observed in the action on 
proalgebraic braids.
By exploiting the properties,
we  explicitly calculate
the inverse image of the trivial (the chordless) chord diagram under the Kontsevich isomorphism.
\end{abstract}


\setcounter{section}{-1}
\section{Introduction}
This paper is a continuation of our previous paper \cite{F12},
where the action of the absolute Galois group on profinite knots
was constructed by an action of 
the Grothendieck-Teichm\"{u}ller {\it profinite} group $\widehat{GT}$ there.
While in this paper, the action of the motivic Galois group,
which is the Galois group of
the tannakian category of 
mixed Tate motives over ${\rm Spec}\;\mathbb Z$,
on the linear span of proalgebraic tangles
is deduced
from an action of the Grothendieck-Teichm\"{u}ller {\it proalgebraic} group $GT(\mathbb K)$ ($\mathbb K$: a field of characteristic $0$) there.

{\it Proalgebraic tangles} (Definition \ref{definition of proalgebraic knots})
means the $\mathbb K$-linear span of oriented tangles completed by a filtration \`{a} la Vassiliev.
{\it Proalgebraic $n$-string links} (resp. {\it proalgebraic knots})
are proalgebraic analogues of 
$n$-string links (a good example can be found in 
Figure \ref{example of string link}) (resp. knots) and
they are subspaces of proalgebraic tangles spanned by them.
In \S \ref{Proalgebraic tangles and chord diagrams}, we give 
a $GT(\mathbb K)$-action on proalgebraic tangles
by following a method indicated in \cite{KT98}.
The action is interpreted as an extension of the $GT(\mathbb K)$-action 
on the proalgebraic braids (reviewed in \S \ref{Proalgebraic braids and infinitesimal braids}) into the one on the proalgebraic tangles.
In \S \ref{Main results} we derive distinguished properties of
the $GT(\mathbb K)$-action 
on proalgebraic tangles 
which can not be observed in $GT(\mathbb K)$-action on proalgebraic braids.
Particularly 

{\bf Theorem A (Theorem \ref{inner automorphism theorem on tangles}
and Proposition \ref{Gm-action on knots}).}
{\it Let $GT_1(\mathbb K)$ be the unipotent part of $GT(\mathbb K)$. Then

(1) The $GT_1(\mathbb K)$-action on proalgebraic string links is given by an inner conjugation.

(2) The $GT_1(\mathbb K)$-action on proalgebraic knots is trivial,
which yields a non-trivial decomposition \eqref{grading decomposition of knots}
of each oriented knot.
}

They are  derived by Twistor Lemma (Lemma \ref{twistor lemma} and \ref{twistor lemma on GT}),
which can be seen as reformulations of \cite{AT12} Theorem 7.5,
\cite{AET} Theorem 2,
\cite{LM} Theorem 8 and \cite{LS} Theorem 2.1
in our setting,
though all of which originate from \cite{Dr} Theorem ${\mathrm A}^\prime$. 
The action on proalgebraic $n$-string links in Theorem A.(1)
is faithful for $n\geqslant 3$
and is not faithful (actually it is trivial) for $n=1$ (see Remark \ref{faithful-trivial}).
However the case for $n=2$ remains unsolved (Problem \ref{problem n=2}).
Since $GT(\mathbb K)$ contains the motivic Galois group,
our $GT(\mathbb K)$-action on proalgebraic tangles yields 
the Galois action there,
which yields a structure 
of 
mixed Tate motives of ${\rm Spec}\;\mathbb Z$
there.
Particularly Theorem A.(2) says that proalgebraic knots 
can be decomposed into infinite summations of Tate motives.

By exploiting the properties shown in the above theorem,
we explicitly determine the inverse image of the unit $e$, 
the chordless chord diagram on the oriented circle, 
under Kontsevich isomorphism $I$ in \eqref{Kontsevich's isomorphism}:

{\bf Theorem B (Theorem \ref{inverse image theorem}).}
{\it
Let $c_0$ be the proalgebraic knot which is the infinite summation of (topological) knots given in Figure \ref{c0-intro}.
Put 
\begin{equation}\label{gamma0-intro}
\gamma_0:=\orientedcircle-{c_0}+c_0\sharp c_0-c_0\sharp c_0\sharp c_0
+c_0\sharp c_0\sharp c_0\sharp c_0-\cdots
\end{equation}
where $\orientedcircle$ is the trivial knot and $\sharp$ is the product called the connected sum.
Then Kontsevich (knot) invariant of $\gamma_0$ becomes trivial, i.e.
$I(\gamma_0)=e$. Namely
$$
I^{-1}(e)=\gamma_0.
$$
}
We recall that the image $I(\orientedcircle)$ of the unit $\orientedcircle$,
the trivial knot, 
under the Kontsevich isomorphism $I$ 
was calculated in \cite{BLT}.
Hence the above theorem could be regarded as a calculation 
in an opposite direction to their calculation.

The contents of the paper go as follows:
\S \ref{Proalgebraic braids and infinitesimal braids}
is a review on Drinfeld's tools of Grothendieck-Teichm\"{u}ller groups
and their actions on braids,
which serves for a good understanding of the actions of the groups on
proalgebraic tangles given in \S \ref{Proalgebraic tangles and chord diagrams}.
Main results will be shown in \S \ref{Main results}.

{\bf Convention.}
In this paper, $\mathbb K$ means a commutative field with characteristic $0$.
(Actually we may more generally assume that it is a commutative ring containing the rational number field $\mathbb Q$.)
The symbols $\mathbb C$, $\mathbb R$, $\mathbb Q_p$, $\mathbb Z_p$, $\mathbb Z$
and $\widehat{\mathbb Z}$ stand for
the complex number field, the real number field, the $p$-adic number field ($p$: a prime),
the $p$-adic integer ring, the integer ring and its profinite completion respectively.
\section{Proalgebraic braids and infinitesimal braids}
\label{Proalgebraic braids and infinitesimal braids}
This is an expository section for non-experts on Drinfeld's works on the action of
the proalgebraic Grothendieck-Teichm\"{u}ller groups
on proalgebraic braids and also on infinitesimal braids
and also a short review on a relationship of the groups with
the motivic Galois group. 

\subsection{The $GT$-action}\label{GT-action on proalgebraic braids}
We recall  in Definition \ref{definition of proalgebraic GT} explicitly  the definition 
of the  Grothendieck-Teichm\"{u}ller group $GT(\mathbb K)$, 
a proalgebraic group introduced by  Drinfeld \cite{Dr},
and explain its action on the proalgebraic braids  $\widehat{{\mathbb K}[B_n]}$
for $n\geqslant 2$ in Proposition \ref{proalg-GT-action theorem on braids}.

\begin{nota}\label{various braids preparation}
(1)
Let $B_n$  be the {\it Artin braid group} with $n$-strings ($n\geqslant 2$)
with  standard generators $\sigma_i$
($1\leqslant i \leqslant n-1$) and defining relations
$\sigma_i\sigma_{i+1}\sigma_i=\sigma_{i+1}\sigma_i\sigma_{i+1}$ and
$\sigma_i\sigma_j=\sigma_j\sigma_i$ for $|i-j|>1$.
The generator $\sigma_i$ in $B_n$ is 
depicted as in Figure \ref{sigma}. 
And for $b$ and $b'\in B_n$, 
we draw the product $b\cdot b'\in B_n$ as in 
Figure \ref{product picture} 
(the order of product $b\cdot b'$ is chosen to combine the bottom endpoints of $b$
with the top endpoints of $b'$).
\begin{figure}[h]
\begin{tabular}{c}
  \begin{minipage}{0.5\hsize}
      \begin{center}
          \begin{tikzpicture}
                  \draw[-] (-0.2,0) --(-0.2, 0.5) ;
                    \draw[-] (0,0) --(0, 0.5) ;
                    \draw[dotted] (0.1,0.3) --(0.6, 0.3) ;
                    \draw[-] (0.7,0) --(0.7, 0.5) ;
                     \draw[decorate,decoration={brace,mirror}] (-0.3,0) -- (0.8,0) node[midway,below]{$i-1$};
                   \draw [-] (1.3,0)--(0.9,0.5);
               \draw[color=white, line width=5pt]  (0.9,0)--(1.3,0.5);
                   \draw [-] (0.9,0)--(1.3,0.5);

                   \draw[-] (1.5,0)--(1.5,0.5) ;
                    \draw[-] (1.7,0) --(1.7, .5)  ;

                    \draw[-] (2.6,0) --(2.6, .5)  ;
                   \draw[dotted] (1.8,0.3) --(2.5, 0.3)  ;
                    \draw[decorate,decoration={brace,mirror}] (1.4,0) -- (2.7,0) node[midway,below]{$n-i-1$};
        \end{tikzpicture}
            \caption{$\sigma_i$}
            \label{sigma}
     \end{center}
 \end{minipage}

 \begin{minipage}{0.5\hsize}
      \begin{center}
        \begin{tikzpicture}
                    \draw[-] (4.0,0.2) --(4.0, 0.5) (4.0,1.0) --(4.0, 1.5) (4.0,2.0) --(4.0, 2.3);
                    \draw[-] (4.1,0.2) --(4.1, 0.5) (4.1,1.0) --(4.1, 1.5) (4.1,2.0) --(4.1, 2.3);
                    \draw[dotted] (4.2,0.3) --(4.9, 0.3) (4.2,1.1)--(4.9,1.1)  (4.2,1.35)--(4.9,1.35) (4.2,2.2) --(4.9, 2.2); 
                    \draw[-] (5.0,0.2) --(5.0, 0.5) (5.0,1.0) --(5.0, 1.5) (5.0,2.0) --(5.0, 2.3) ;
                    \draw (3.9,0.5) rectangle (5.1,1);
                    \draw (4.5,0.7) node{$b'$};
                    \draw (3.9,1.5) rectangle (5.1,2.0);
                    \draw (4.5,1.7) node{$b$};
                    \draw[decorate,decoration={brace,mirror}] (3.9,0.2) -- (5.1,0.2) node[midway,below]{$n$};
                    \draw[decorate,decoration={brace}]  (3.9,2.3) -- (5.1,2.3) node[midway,above]{$n$};
                   \draw[color=white, line width=3pt] (3.8,1.25)--(5.2,1.25);

           \end{tikzpicture}
               \caption{$b\cdot b'$}
               \label{product picture}
      \end{center}
  \end{minipage}
\end{tabular}
\end{figure}

We denote the  pure part of $B_n$ by $P_n$, i.e. the kernel of the natural homomorphism $P_n\to\frak S_n$,
and call it by the {\it pure braid group}.
For $1\leqslant i<j\leqslant n$,
special elements
$$x_{i,j}=x_{j,i}=(\sigma_{j-1}\cdots\sigma_{i+1})\sigma_i^2(\sigma_{j-1}\cdots\sigma_{i+1})^{-1}$$
form a generating set of $P_n$.
For $1\leqslant a\leqslant a+\alpha<b\leqslant b+\beta\leqslant n$,
we define
\begin{align*}
x_{a\cdots a+\alpha,b\cdots b+\beta}:=
&(x_{a,b}x_{a,b+1}\cdots x_{a,b+\beta})\cdot
(x_{a+1,b}x_{a+1,b+1}\cdots x_{a+1,b+\beta}) \\
&\cdots(x_{a+\alpha,b}x_{a+\alpha,b+1}\cdots x_{a+\alpha,b+\beta})
\in P_n.
\end{align*}
They are drawn in Figure \ref{pure generator 1} and \ref{pure generator 2}.
\begin{figure}[h]
\begin{tabular}{c}
  \begin{minipage}{0.5\hsize}
      \begin{center}
           \begin{tikzpicture}
                 \draw[-] (-0.5,0)--(-0.5,1) (-0.2,0)--(-0.2,1) (0.2,0)--(0.2,1) (0.3,0)--(0.3,1) (0.6,0)--(0.6,1) (1.2,0)--(1.2,1) (1.5,0)--(1.5,1);
                 \draw[dotted] (-0.5,0.5)--(-0.2,0.5) (0.3,0.5)--(0.6,0.5) (1.2,0.5)--(1.5,0.5);
                 \draw[-] (0.8,0)--(0.8,0.5);
                 \draw[color=white, line width=7pt](0,0) ..controls(0.1,0.5) and (0.9,0)        ..(1,0.5);
                 \draw[-] (0,0) ..controls(0.1,0.5) and (0.9,0)        ..(1,0.5);
                 \draw[color=white, line width=7pt](1,0.5) ..controls(0.9,1.0) and (0.1,0.5)       ..(0,1);
                 \draw[-] (1,0.5) ..controls(0.9,1.0) and (0.1,0.5)       ..(0,1);
                 \draw[color=white, line width=5pt] (0.8,0.5)--(0.8,1);
                 \draw[-] (0.8,0.5)--(0.8,1);
\draw[decorate,decoration={brace,mirror}] (-0.6,-0.1) -- (-0.1,-0.1) node[midway,below]{$i-1$};
\draw[decorate,decoration={brace}] (-0.6,1.1) -- (0.6,1.1) node[midway,above]{$j-1$};
            \end{tikzpicture}
               \caption{$x_{ij}$}
               \label{pure generator 1}     \end{center}
 \end{minipage}

 \begin{minipage}{0.5\hsize}
     \begin{center}
          \begin{tikzpicture}
                 \draw[-] (-0.5,-0.1)--(-0.5,1.1) (-0.2,-0.1)--(-0.2,1.1)  (0.6,-0.1)--(0.6,1.1)  (0.7,-0.1)--(0.7,1.1) (0.8,-0.1)--(0.8,1.1) (1.9,-0.1)--(1.9,1.1) (2.3,-0.1)--(2.3,1.1);
                \draw[dotted] (-0.5,0.5)--(-0.2,0.5) (0.6,0.95)--(0.8,0.95) (1.9,0.5)--(2.3,0.5);
                 \draw[-] (1.0,0)--(1.0,0.5) (1.1,0)--(1.1,0.5) (1.3,0)--(1.3,0.5) ;
                 \draw[color=white, line width=10pt](0.15,0) ..controls(0.25,0.15) and (1.55,0.25)        ..(1.65,0.5);
                 \draw[-] (0,0) ..controls(0.1, 0.5) and (1.4, 0.25)        ..(1.5,0.5);
                 \draw[-] (0.3,0) ..controls(0.4,0.25) and (1.7,0.25)        ..(1.8,0.5);
                 \draw[color=white, line width=7pt] (1,0.5) ..controls(0.9,1.0) and (0.1,0.5)       ..(0,1);
                 \draw[-] (1.5,0.5) ..controls(1.4,0.75) and (0.1,0.5)       ..(0,1);
                 \draw[-] (1.8,0.5) ..controls(1.7,0.75) and (0.4,0.75)       ..(0.3,1);
                 \draw[color=white, line width=12pt] (1.15,0.5)--(1.15,1);
                 \draw[-] (1,0.5)--(1,1.1) (1.1,0.5)--(1.1,1.1) (1.3,0.5)--(1.3,1.1) ;
                \draw[dotted] (0.1,0.1)--(0.3,0.1) (0.1,0.9)--(0.3,0.9) (1,0.8)--(1.3,0.8) ;
\draw[decorate,decoration={brace,mirror}] (0,-0.1) -- (0.3,-0.1) node[midway,below]{\tiny{$\alpha+1$}};
\draw[decorate,decoration={brace,mirror}] (1,-0.1) -- (1.3,-0.1) node[midway,below]{\tiny{$\beta+1$}};
\draw[decorate,decoration={brace}] (-0.5,1.15) -- (-0.2,1.15) node[midway,above]{\tiny{$a-1$}};
\draw[decorate,decoration={brace}] (0.55,1.3) -- (0.8,1.3) node[midway,above]{\tiny{$b-a-\alpha-1$}};
         \end{tikzpicture}
               \caption{$x_{a\cdots a+\alpha,b\cdots b+\beta}$}
               \label{pure generator 2}
      \end{center}
  \end{minipage}
\end{tabular}
\end{figure}

We mean
$\widehat{{\mathbb K}[B_n]}$ 
by the completion of the group algebra ${\mathbb K}[B_n]$
with respect to the two-sided ideal $I$ generated by $\sigma_i-\sigma_i^{-1}$
for $1\leqslant i \leqslant n-1$;
$$\widehat{{\mathbb K}[B_n]}:=\varprojlim_N{\mathbb K}[B_n]/I^N$$
(cf. \cite{F12}).
By abuse of notation, we denote the induced filtration on
$\widehat{{\mathbb K}[B_n]}$
by the same symbol $\{I^n\}_{n\geqslant 0}$.
It is checked that $\widehat{{\mathbb K}[B_n]}$ is a filtered Hopf algebra.
We call its group-like part by  the {\it proalgebraic braid group} and
denote it by $B_n({\mathbb K})$.
It naturally admits a structure of a proalgebraic group over $\mathbb K$.
We note that it is  Hain's \cite{H} relative  completion of $B_n$
with respect to the natural projection $B_n\to\frak S_n$.

(2)
Similarly we denote its pure part by  $\widehat{{\mathbb K}[P_n]}$.
Namely
$$\widehat{{\mathbb K}[P_n]}:=\varprojlim_N{\mathbb K}[P_n]/I_0^N$$
with $I_0=I\cap {\mathbb K}[P_n]$.
It is also  a filtered Hopf algebra.
The  {\it proalgebraic pure braid group} $P_n({\mathbb K})$ means its group-like part.
We note that it  is a unipotent (Malcev) completion of $P_n$
because $I_0$ forms an augmentation ideal of $ {\mathbb K}[P_n]$.

(3)
Let ${F}_2({\mathbb K})$ be the  prounipotent algebraic group over $\mathbb K$,
the unipotent completion of
the free  group $F_2$ of rank $2$ with two variables $x$ and $y$, that is,
the group-like part of the Hopf algebra $\widehat{{\mathbb K}[F_2]}$
completed by the augmentation ideal.
Similarly to the convention in \cite{F12},
for any $f\in{F}_2({\mathbb K})$ and any homomorphism 
$\tau:{F}_2(\mathbb K)\to G(\mathbb K)$ of proalgebraic groups 
sending $x\mapsto \alpha$ and $y\mapsto \beta$,
the symbol $f(\alpha,\beta)$ stands for the image $\tau(f)$.
Particularly for the (actually injective) homomorphism ${F}_2(\mathbb K)\to {P}_n(\mathbb K)$
of proalgebraic groups
sending $x\mapsto x_{a\cdots a+\alpha,b\cdots b+\beta}$ and 
$y\mapsto x_{b\cdots b+\beta,c\cdots c+\gamma}$
($1\leqslant a\leqslant a+\alpha<b\leqslant b+\beta<c\leqslant c+\gamma
\leqslant n$),
the image of $f\in {F}_2(\mathbb K)$ is denoted by 
$f_{a\cdots a+\alpha,b\cdots b+\beta,c\cdots c+\gamma}$.
\end{nota}

The proalgebraic 
group ${GT}(\mathbb K)$
is defined by  Drinfeld \cite{Dr} to be a subgroup of the automorphism
(proalgebraic) group of ${F}_2(\mathbb K)$.

\begin{defn}[\cite{Dr}]\label{definition of proalgebraic GT}
The {\it proalgebraic Grothendieck-Teichm\"{u}ller group}
${GT}(\mathbb K)$ is the proalgebraic group over $\mathbb K$ whose set of $\mathbb K$-values points
forms a  subgroup of  $\mathrm{Aut}{F}_2(\mathbb K)$ and is
defined by
\begin{equation*}
{GT}(\mathbb K):=\Bigl\{
\sigma\in \mathrm{Aut}{F}_2(\mathbb K)\Bigm |
{\begin{array}{l}
\sigma(x)=x^\lambda, \sigma(y)=f^{-1}y^\lambda f
\text{ for some } (\lambda,f)\in {\mathbb K}^\times\times {F}_2(\mathbb K)\\
\text{satisfying the three relations below. }
\end{array}}
\Bigr\}
\end{equation*}
\end{defn}
\begin{equation}\label{proalg-2-cycle}
f(x,y)f(y,x)=1
\quad\text{ in } {F}_2(\mathbb K),
\end{equation}
\begin{equation}\label{proalg-hexagon equation}
f(z,x)z^m f(y,z)y^m f(x,y)x^m=1
\text{ in } {F}_2(\mathbb K)
\text{ with } z=(xy)^{-1},\ m=\frac{\lambda-1}{2},
\end{equation}
\begin{equation}\label{proalg-pentagon equation}
f_{1,2,34}f_{12,3,4}=f_{2,3,4}f_{1,23,4}f_{1,2,3}
\quad\text{ in } {P}_4(\mathbb K).
\end{equation}

The powers $x^\lambda$, $y^\lambda$, $x^m$, $y^m$, $z^m$
appearing in the above all make sense
because  $F_2(\mathbb K)$ is the prounipotent completion of $F_2$.
For $f_{1,2,34}$ etc., see Notation \ref{various braids preparation}.

We remark that each $\sigma\in {GT}(\mathbb K)$ determines a pair $(\lambda, f)$ uniquely
because the pentagon equation \eqref{proalg-pentagon equation} implies that
$f$ belongs to the  commutator of ${F}_2(\mathbb K)$.
By abuse of notation,
we occasionally express the pair $(\lambda,f)$ 
to represents  $\sigma$ and denote as
$\sigma=(\lambda,f)\in {GT}(\mathbb K)$. 

The above set-theoretically defined  ${GT}(\mathbb K)$ forms indeed a proalgebraic group
whose product is induced from that of $\mathrm{Aut}{F}_2(\mathbb K)$
and is given by
\footnote{
For our purpose to make \eqref{GT to Aut B_n} not anti-homomorphic but homomorphic, 
we reverse the order of the product given in  the original paper \cite{Dr}.}
\begin{equation}\label{product of proalg-GT}
(\lambda_2,f_2)\circ (\lambda_1,f_1)
:=\Bigl(\lambda_2\lambda_1, f_1(f_2x^{\lambda_2}f_2^{-1},y^{\lambda_2})\cdot f_2\Bigr)
=\Bigl(\lambda_2\lambda_1,f_2\cdot f_1(x^{\lambda_2},f_2^{-1}y^{\lambda_2}f_2)\Bigr).
\end{equation}
The first equality is the definition and the second equality can be easily verified.
We denote the subgroup of $GT(\mathbb K)$ with $\lambda =1$ by $GT_1(\mathbb K)$;
$$
{GT}_1(\mathbb K):=\left\{\sigma=(\lambda,f)\in GT(\mathbb K) \ | \
\lambda=1\right\}.
$$
It is easily seen that it forms a proalgebraic unipotent subgroup of  $GT(\mathbb K)$.

\begin{rem}
In some literatures, \eqref{proalg-2-cycle}, \eqref{proalg-hexagon equation} and \eqref{proalg-pentagon equation}
are called {\it 2-cycle, 3-cycle and 5-cycle relation} respectively.
The author often calls  \eqref{proalg-2-cycle} and \eqref{proalg-hexagon equation} 
by {\it two hexagon equations}
and \eqref{proalg-pentagon equation} by  {\it one pentagon equation}
because they reflect the three axioms, two hexagon and one pentagon axioms,
of braided monoidal (tensor) categories \cite{JS}.
We remind that \eqref{proalg-pentagon equation} represents
$$
f(x_{12},x_{23}x_{24})f(x_{13}x_{23},x_{34})=f(x_{23},x_{34})f(x_{12}x_{13},x_{24}x_{34})f(x_{12},x_{23})
\quad\text{ in } {P}_4(\mathbb K).
$$
In  several literatures such as \cite{Ih, F06}, the equation
\eqref{proalg-pentagon equation} is replaced by a different (more symmetric) formulation:
$$
f(x^*_{12},x^*_{23})f(x^*_{34},x^*_{45})f(x^*_{51},x^*_{12})f(x^*_{23},x^*_{34})f(x^*_{45},x^*_{51})=1
\quad\text{ in } {P}^*_5(\mathbb K)
$$
where $P_5^*$ is the pure sphere braid group with $5$-strings.
\end{rem}

The author actually showed that  the pentagon equation implies two hexagon equations:

\begin{prop}[\cite{F10}]
Let $\mathbb K$ be an algebraically closed field of characteristic $0$.
For each $f\in F_2(\mathbb K)$
satisfying  \eqref{proalg-pentagon equation},
there always exists (actually unique up to signature) $\lambda\in \mathbb K$ such that  
the pair $(\lambda, f)$ satisfies the two hexagon equations
\eqref{proalg-2-cycle} and \eqref{proalg-hexagon equation}.
\end{prop}

The following Drinfeld's $GT(\mathbb K)$-action on $\widehat{{\mathbb K}[B_n]}$ 
plays a fundamental role in our paper here.

\begin{prop}[\cite{Dr}]\label{proalg-GT-action theorem on braids}
Let $n\geqslant 2$.
There is a continuous $GT(\mathbb K)$-action on 
the filtered Hopf algebra $\widehat{{\mathbb K}[B_n]}$
\begin{equation}\label{GT to Aut B_n}
\rho_n:GT(\mathbb K)\to \mathrm{Aut}\ \widehat{{\mathbb K}[B_n]}
\end{equation}
which is induced by, for each $\sigma=(\lambda,f)\in GT(\mathbb K)$,
\begin{equation*}
\rho_n(\sigma):
\begin{cases}
\sigma_1 \quad \mapsto \quad&\sigma_1^\lambda, \\
\sigma_i  \quad \mapsto  \quad f_{1\cdots i-1,i,i+1}^{-1}&\sigma_i^\lambda f_{1\cdots i-1,i,i+1}
\qquad (2\leqslant i\leqslant n-1). \\
\end{cases}
\end{equation*}
\end{prop}

Here $\sigma_i^\lambda:=\sigma_i\cdot (\sigma_i^2)^{\frac{\lambda-1}{2}}\in \widehat{{\mathbb K}[B_n]}$
and $f_{1\cdots i-1,i,i+1}=f(x_{1i}x_{2i}\cdots x_{i-1,i},x_{i,i+1})\in P_n(\mathbb K)$ (see Notation \ref{various braids preparation}).
.It is well defined because $\sigma_i^2$ belongs to the unipotent completion $P_n(\mathbb K)$
and ${\frac{\lambda-1}{2}}$-th power of $\sigma_i^2$ makes sense in $P_n(\mathbb K)$.
We denote $\rho_n(\sigma)(b)$ simply by $\sigma(b)$ 
when there is no confusion.

We note that $\rho_n$ is injective when $n\geqslant 3$.

\begin{rem}
For each prime $l$,
there are natural homomorphisms
$\widehat{B}_n\to B_n(\mathbb Q_l)$
and $\widehat{P}_n\to P_n(\mathbb Q_l)$.
Hence we have
$$
\widehat{B}_n\to\widehat{{\mathbb Q}_l[B_n]}.
$$
By natural homomorphisms $\widehat{\mathbb Z}\to {\mathbb Q}_l$,
$\widehat{F}_2\to F_2(\mathbb Q_l)$ and
$\widehat{P}_4\to P_4(\mathbb Q_l)$, we obtain a continuous group homomorphism
\begin{equation}\label{GT to GTQl}
\widehat{GT}\to GT(\mathbb Q_l)
\end{equation}
from the profinite Grothendieck-Teichm\"{u}ller group $\widehat{GT}$ (cf. \cite{F12}).
By direct calculations, it can be verified that the above two homomorphisms are
consistent with the $ GT(\mathbb Q_l)$-action on $\widehat{{\mathbb Q}_l[B_n]}$
in Proposition \ref{proalg-GT-action theorem on braids}
and the $\widehat{GT}$-action  on $\widehat{B}_n$
(given in \cite{Dr} and see also \cite{F12}).
\end{rem}

In \S \ref{Main results}, we will extend the action on  proalgebraic braids to the one on 
proalgebraic tangles
and will show that actually the above action on proalgebraic pure braid groups
is realized as an inner automorphism
of proalgebraic string links (Corollary \ref{inner action on braids}).

\subsection{The $GRT$-action}\label{GRT-action on infinitesimal braids}
We recall  explicitly
Drinfeld's definitions of the {\it graded} Grothendieck-Teichm\"{u}ller group $GRT(\mathbb K)$
in Definition \ref{definition of proalgebraic GRT}. 
We discuss  their actions on the algebra $\widehat{U\frak b_n}$ of 
infinitesimal braids for $n\geqslant 2$
in Proposition \ref{proalg-GRT-action theorem on braids}.

\begin{nota}
(1)
Let $\frak p_n$  be the {\it infinitesimal pure braid Lie algebra} with $n$-strings ($n\geqslant 2$)
with  standard generators $t_{ij}$
($1\leqslant i,j \leqslant n$) and defining relations
$t_{ii}=0$, $t_{ij}=t_{ji}$, $[t_{ij},t_{ik}+t_{jk}]=0$ and $[t_{ij},t_{kl}]=0$
when $i$,$j$,$k$,$l$ all differ.
For $1\leqslant a\leqslant a+\alpha<b\leqslant b+\beta\leqslant n$,
we define
\begin{align*}
t_{a\cdots a+\alpha,b\cdots b+\beta}:=
\sum_{0\leqslant i\leqslant\alpha, 0\leqslant j\leqslant \beta}t_{a+i,b+j}
\in {\frak p}_n.
\end{align*}
We denote  by $U\frak p_n$ its enveloping algebra and by $\widehat{U\frak p_n}$ 
its completion with respect to its augmentation ideal.

(2)
We denote $\frak S_n$ to be the symmetric group acting on $\{1,2,\ldots,n\}$ 
($n\geqslant 1$)
and $\tau_{i,i+1}$ in $\frak S_n$ to be the transpose of $i$ and $i+1$
($1\leqslant i\leqslant n-1$).
The group $\frak S_n$ acts $\widehat{U\frak p_n}$
by $\tau\cdot t_{ij}=t_{\tau^{-1}(i),\tau^{-1}(j)}$
for $\tau\in \frak S_n$ and $1\leqslant i,j \leqslant n$.
We may consider the crossed product (cf. \cite{Maj} etc.)
$$\widehat{U\frak b_n}:={\mathbb K}[\frak S_n]*\widehat{U\frak p_n}.$$
It is ${\mathbb K}[\frak S_n]\otimes_{\mathbb K}\widehat{U\frak p_n}$
as vector space
with the product structure given by 
$$
(\tau_1\otimes t_{i_1j_1})\cdot (\tau_2\otimes t_{i_2j_2}):=
\tau_1\tau_2\otimes (t_{\tau_2^{-1}(i_1)\tau_2^{-1}(j_1)}\cdot t_{i_2j_2})
$$
for $\tau_1,\tau_2\in\frak S_n$.
By abuse of notation, occasionally $\tau$ indicate $\tau\otimes 1$ for $\tau\in\mathbb K[\frak S_n]$
and $t$ indicate $1\otimes t$ for $t\in\widehat{U\frak p_n}$ in this paper.
Hence we have
\begin{equation}\label{tau and t}
\tau\cdot t_{ij}=t_{\tau(i)\tau(j)}\cdot  \tau \quad
(=\tau\otimes t_{ij})
\end{equation}
for  $\tau\in\frak S_n$.
We occasionally depict  $t_{ij}\in\widehat{U\frak p_n}$
as the diagram with $n$ vertical lines and a dotted horizontal line
(called a {\it chord})
connecting $i$-th and $j$-th lines 
and $\tau\in\frak S_n$ as the diagram connecting each $i$-th point  on the bottom
with $\tau(i)$-th point on the top by an interval.
The order of the product $b\cdot b'$ is chosen to combine the bottom endpoints of $b$ with
the top endpoints of $b'$. 
By putting 
$\deg t_{ij}=1\ (1\leqslant i,j \leqslant n)\text{ and }\deg \tau=0 \ (\tau\in \frak S_n),  $
we can show that both $\widehat{U\frak p_n}$
and $\widehat{U\frak b_n}$ carry structures of graded Hopf algebras.


(3)
Let ${\frak f}_2$ be the free Lie algebra over $\mathbb K$ with two variables $A$ and $B$
and $\widehat{U{\frak f}_2}$ be its completed Hopf algebra. 
Again similarly,
for any $g\in\widehat{U{\frak f}_2}$ and any algebra homomorphism 
$\tau:\widehat{U{\frak f}_2}\to {S}$
sending $A\mapsto v$ and $B\mapsto w$,
the symbol $g(v,w)$ stands for the image $\tau(g)$.
Particularly for the (actually injective) homomorphism 
$\widehat{U{\frak f}_2}\to\widehat{U{\frak p}_n}$
sending $A\mapsto t_{a\cdots a+\alpha,b\cdots b+\beta}$ and 
$B\mapsto t_{b\cdots b+\beta,c\cdots c+\gamma}$
($1\leqslant a\leqslant a+\alpha<b\leqslant b+\beta<c\leqslant c+\gamma
\leqslant n$),
the image of $g\in \widehat{U{\frak f}_2}$ is denoted by 
$g_{a\cdots a+\alpha,b\cdots b+\beta,c\cdots c+\gamma}\in 
\widehat{U{\frak p}_n}$.
\end{nota}

We note that $\widehat{U\frak p_n}$ is not algebraically generated by $t_{i,i+1}$ 
($1\leqslant i\leqslant n-1$)
while the following holds for $\widehat{U\frak b_n}$.

\begin{lem}
The algebra $\widehat{U\frak b_n}$ is algebraically generated by $t_{i,i+1}$ 
and $\tau_{i,i+1}$ for $1\leqslant i\leqslant n-1$.
\end{lem}

\begin{proof}
The elements $\tau_{i,i+1}$ generate $\frak S_n$ and, by \eqref{tau and t},
any $t_{kl}$ is obtained from  $t_{i,i+1}$ 
and $\tau_{i,i+1}$  ($1\leqslant i\leqslant n-1$),
which yields our claim.
\end{proof}

The proalgebraic 
group ${GRT}_1(\mathbb K)$
is defined by  Drinfeld \cite{Dr} to be a proalgebraic subgroup of the automorphism
(proalgebraic) group of $\exp{\frak f}_2$.

\begin{defn}[\cite{Dr}]\label{definition of proalgebraic GRT}
The {\it proalgebraic graded Grothendieck-Teichm\"{u}ller group}
${GRT}(\mathbb K)$ is
the subgroup of  $\mathrm{Aut} \exp {\frak f}_2$
defined by
\begin{equation*}
{GRT}(\mathbb K):=\Bigl\{
\sigma\in \mathrm{Aut} \exp {\frak f}_2\Bigm |
{\begin{array}{l}
\sigma(e^{A})=e^{A/c}, \sigma(e^B)=g^{-1}e^{B/c} g
\text{ for some }  c\in {\mathbb K}^\times \\
\text{and }
g\in \exp {\frak f}_2
\text{ satisfying two hexagon equations}\\
\text{\eqref{GRT-2-cycle}-\eqref{GRT-hexagon equation} 
and  one pentagon equation \eqref{GRT-pentagon equation} below.}
\end{array}}
\Bigr\}
\end{equation*}
\end{defn}
\begin{equation}\label{GRT-2-cycle}
g(A,B)g(B,A)=1
\quad\text{ in } \exp{\frak f}_2,
\end{equation}
\begin{equation}\label{GRT-hexagon equation}
g(C,A)
g(B,C)
g(A,B)=1\quad\text{ in } \exp{\frak f}_2
\text{  with }  C=-A-B,
\end{equation}
\begin{equation}\label{GRT-pentagon equation}
g_{1,2,34}g_{12,3,4}=g_{2,3,4}g_{1,23,4}g_{1,2,3}
\quad\text{ in } \exp{\frak p}_4.
\end{equation}

Similarly to Definition \ref{definition of proalgebraic GT},
we  remark that each $\sigma\in {GRT}(\mathbb K)$ determines a pair $(c, g)$ uniquely.
By abuse of notation,
we occasionally express the pair $(c,g)$ 
to represents  $\sigma\in {GRT}(\mathbb K)$ and denote as
$\sigma=(c,g)\in {GRT}(\mathbb K)$. 

The above set-theoretically defined  ${GRT}(\mathbb K)$ forms indeed a proalgebraic group
whose product is induced from that of $\mathrm{Aut} \exp{\frak f}_2(\mathbb K)$
and is given by
\footnote{
We remark again that, for our purpose, we reverse the order of the product given in \cite{Dr}.}
\begin{equation}\label{product of proalg-GRT}
(c_2,g_2)\circ (c_1,g_1)
=\left(c_2c_1,\ g_1\left(g_2\frac{A}{c_2}g_2^{-1},\ \frac{B}{c_2}\right)\cdot g_2\right)
=\left(c_2c_1,\ g_2\cdot g_1\left(\frac{A}{c_2},\ g_2^{-1}\frac{B}{c_2} g_2\right)\right).
\end{equation}
The first equality is the definition and the second equality can be easily verified.
Notice the simple equality
$
(1,g)\circ (c,1)=(c,g).
$
We denote the subgroup of $GRT(\mathbb K)$ with $c =1$ by $GRT(\mathbb K)_1$,
i.e.
$
{GRT}_1(\mathbb K):=\left\{\sigma=(c,g)\in GRT(\mathbb K) \ | \
c=1\right\}.
$
It is easily seen that it forms a proalgebraic unipotent subgroup of  $GRT(\mathbb K)$.

\begin{rem}
The symbol $GRT$ stands for \lq graded Grothendieck-Teichm\"{u}ller group.'
Indeed its grading on $GRT_1(\mathbb K)$
is equipped by the action of ${\mathbb G}_m(\mathbb K)$ 
$(={\mathbb K}^\times)$ given by 
\begin{equation}\label{Gm-action}
\left(1, g(A,B)\right)\mapsto\left(1,\ g(\frac{A}{c},\frac{B}{c})\right) 
\end{equation}
for $c\in {\mathbb K}^\times$ and $g\in GRT_1(\mathbb K)$,
which is reformulated by $(1,g)\mapsto (c,1)\circ (1,g)$.
Thus we have, by the action,
$$
GRT(\mathbb K)={\mathbb K}^\times\ltimes GRT_1(\mathbb K).
$$
\end{rem}

Two specific elements of $GRT(\mathbb K)$ are known.

\begin{eg}
(1)
The {\it $p$-adic Drinfeld associator} $\varPhi^p_\mathrm{KZ}(A,B)$, 
a $p$-adic analogue of the Drinfeld (KZ-)associator (cf. Example \ref{two associators})  
is a non-commutative formal power
series whose coefficients are $p$-adic multiple zeta values \cite{F04} . 
It was constructed
as a regularized holonomy of the $p$-adic KZ-equation
and was shown in \cite{F07} by the results of \cite{U} that it belongs to $GRT_1(\mathbb K)$ 
with ${\mathbb K}={\mathbb Q}_p$. 

(2)
The {\it $p$-adic Deligne associator} $\varPhi^p_\mathrm{De}(A,B)$,
a variant of the above $\varPhi^p_\mathrm{KZ}(A,B)$ (cf. \cite{F07})
is shown  in \cite{U} to be in $GRT_1(\mathbb K)$ with ${\mathbb K}={\mathbb Q}_p$.
It was in \cite{F07} shown that each of its coefficients is given by a  certain polynomial combination
of $p$-adic multiple zeta values.
\end{eg}

The following  $GRT(\mathbb K)$-action on 
$\widehat{U\frak b_n}$ 
was explicitly presented neither in Drinfeld's paper\cite{Dr} 
nor  Bar-Natan's paper \cite{Bar98},
where  they showed $GRT_1(\mathbb K)$-action there.

\begin{prop}\label{proalg-GRT-action theorem on braids}
Let $n\geqslant 2$.
There is a continuous $GRT(\mathbb K)$-action on the graded Hopf algebra 
$\widehat{U\frak b_n}$
\begin{equation}\label{GRT to Aut Ub_n}
\rho_n:GRT(\mathbb K)\to \mathrm{Aut}\ \widehat{U\frak b_n}
\end{equation}
which is induced by, for each $\sigma=(c,g) \in GRT(\mathbb K)$,
\begin{equation*}
\rho_n(\sigma):
\begin{cases}
t_{1,2} \quad \mapsto \quad&\frac{t_{1,2}}{c}, \\
t_{i,i+1}  \quad \mapsto  \quad g_{1\cdots i-1,i,i+1}^{-1}&\frac{t_{i,i+1}}{c} \
g_{1\cdots i-1,i,i+1}
\qquad (2\leqslant i\leqslant n-1), \\
\tau_{1,2} \quad \mapsto \quad&\tau_{1,2}, \\
\tau_{i,i+1}  \quad \mapsto  \quad g_{1\cdots i-1,i,i+1}^{-1}&\tau_{i,i+1} \
g_{1\cdots i-1,i,i+1}
\qquad (2\leqslant i\leqslant n-1). \\
\end{cases}
\end{equation*}
\end{prop}

We recall that $\tau_{i,i+1}$ means the transpose of $i$ and $i+1$ in ${\frak S}_n$.
We again note that $\rho_n$ is injective when $n\geqslant 3$.

In \S \ref{Main results} we will extend the above action on infinitesimal braids to the one on 
chord diagrams
and will show that actually the above action on infinitesimal braids
is realized as an inner automorphism
of chord diagrams.

The associated Lie algebra $\frak{grt}_1$ of $GRT_1$,
which was independently introduced by Ihara \cite{Ih}
and called the stable derivation algebra,
is equipped grading by
the ${\mathbb G}_m$-action \eqref{Gm-action}.

\begin{conj}[\cite{De,Dr,Ih}]
The graded Lie algebra $\frak{grt}_1$ is freely generated by one element in each degree
$3,5,7,\dots$.
\end{conj}

\begin{rem}\label{free Lie subalgebra}
By \cite{Br12}, we know that $\frak{grt}_1$ contains such a free Lie subalgebra
(see Remark \ref{rem-incl-Gal-M} below).
\end{rem}

\subsection{Associators}\label{Associators}
We recall   Drinfeld's definition of the associator set $M(\mathbb K)$ 
in Definition \ref{definition of associator set} and 
its $\left(GRT(\mathbb K),GT(\mathbb K)\right)$-bitorsor structure
in Proposition \ref{GRT-GT-bitorsor}.
Then we will explain how associators give isomorphisms between
$\widehat{{\mathbb K}[B_n]}$ and $\widehat{U\frak b_n}$
in Proposition \ref{M to Isom}.

\begin{defn}[\cite{Dr}]\label{definition of associator set}
The {\it associator set}  $M(\mathbb K)$ is the proalgebraic variety whose set of 
$\mathbb K$-valued points is given by 
\begin{equation*}
{M}(\mathbb K):=\Bigl\{
p=(\mu, \varphi)\in {\mathbb K}\times\exp{\frak f}_2
\Bigm |
\mu\in {\mathbb K}^\times
\text{ and $(\mu,\varphi)$ satisfies \eqref{GRT-2-cycle}, \eqref{GRT-pentagon equation}
and \eqref{assoc-hexagon equation}}.
\Bigr\}
\end{equation*}
\begin{equation}\label{assoc-hexagon equation}
\exp\{\frac{\mu A}{2}\}\varphi(C,A)\exp\{\frac{\mu C}{2}\}\varphi(B,C)\exp\{\frac{\mu B}{2}\}\varphi(A,B)=1
\quad\text{ in } \exp{\frak f}_2
\end{equation}
with  $C=-A-B$.
For each fixed $\mu_0\in\mathbb K$, define the proalgebraic variety $M_{\mu_0}(\mathbb K)$ by
$$
M_{\mu_0}(\mathbb K):=\Bigl\{\varphi\in \exp{\frak f}_2\Bigm |
\varphi \text{ satisfies 
\eqref{GRT-2-cycle}, \eqref{GRT-pentagon equation}
and \eqref{assoc-hexagon equation} with $\mu=\mu_0$}
\Bigr\}.
$$
\end{defn}

Hence we have
$M_0(\mathbb K)=GRT_1(\mathbb K).$
Three examples of associators are known:

\begin{eg}\label{two associators}
(1) 
The {\it KZ-associator}, also known as the Drinfeld associator,
$\varPhi_\mathrm{KZ}(A,B)$  is a non-commutative formal power
series whose coefficients are multiple zeta values. 
It was constructed by Drinfeld \cite{Dr} 
as a regularized holonomy of the KZ-equation
and was shown by him that it belongs to $M_{\mu}(\mathbb K)$ 
with ${\mathbb K}={\mathbb C}$ and $\mu=\pm 2\pi\sqrt{-1}$.
It is known to be expressed as follows:
\begin{align}\label{LM-formula}
\varPhi_\mathrm{KZ}(A,B)=1+
\underset{k_m>1}
{\underset{m,k_1,\dots,k_m\in{\mathbb N}}{\sum}}
(-1)^m & \zeta(k_1,\cdots,k_m)
A^{k_m-1}B\cdots A^{k_1-1}B \\ 
&+\text{(regularized terms)}.\notag
\end{align}
Here $\zeta(k_1,\cdots,k_m)$ is the
{\it multiple zeta value} (MZV in short),
the real number defined by the following power series
\begin{equation*}\label{MZV}
\zeta(k_1,\cdots,k_m)
:=\sum_{0<n_1<\cdots<n_m}\frac{1}
{n_1^{k_1}\cdots n_m^{k_m}}
\end{equation*}
for $m$, $k_1$,\dots, $k_m\in {\mathbb N} (={\mathbb Z}_{>0})$
with $k_m>1$ (its convergent condition).
All of the coefficients of $\varPhi_\mathrm{KZ}$ (including its regularized terms)
are explicitly calculated in terms of
MZV's  in \cite{F03} Proposition 3.2.3
by Le-Murakami's method in \cite{LMb}.

(2) 
The {\it Deligne associator} $\varPhi_\mathrm{De}(A,B)$  (\cite{Br13})
(denoted by $\varPhi^{-}_\mathrm{KZ}(A,B)$ in \cite{F07})
is a non-commutative formal power
series in $M_{\mu}(\mathbb K)$ with  ${\mathbb K}={\mathbb R}$ and $\mu=1$
which is located  in the \lq middle' of 
$\varPhi_\mathrm{KZ}\left(\frac{1}{2\pi\sqrt{-1}}A,\frac{1}{2\pi\sqrt{-1}}B\right)$ and
$\varPhi_\mathrm{KZ}\left(\frac{-1}{2\pi\sqrt{-1}}A,\frac{-1}{2\pi\sqrt{-1}}B\right)$.
Its explicit relationship with the above $\varPhi_\mathrm{KZ}(A,B)$
is given in \cite{F07} Lemma 2.25.

(3)
The {\it AT-associator} $\varPhi_\mathrm{AT}(A,B)$  is another associator.
It was
constructed by  Alekseev and Torossian \cite{AT10} as a  holonomy of  AT-connection,
a certain non-holomorphic flat connection on  a certain configuration space.
\v{S}evera and Willwacher \cite{SW} showed that it belongs to $M_{\mu}(\mathbb K)$
with ${\mathbb K}={\mathbb R}$ and $\mu=1$.
Rossi and Willwacher showed that $\varPhi_\mathrm{AT}\neq\varPhi_\mathrm{De}$ in \cite{RW}. 
\end{eg}

We remark again that the author  also in this setting showed that 
the pentagon equation implies two hexagon equations.

\begin{prop}[\cite{F12}]
Let $\mathbb K$ be an algebraically closed field of characteristic $0$.
For each $\varphi\in \exp{\frak f}_2$
satisfying  \eqref{GRT-pentagon equation},
there always exists (actually unique up to signature) $\mu\in {\mathbb K}$ such that  
the pair $p=(\mu, \varphi)$ satisfies the two hexagon equations
\eqref{GRT-2-cycle} and \eqref{assoc-hexagon equation}.
\end{prop}

It was shown by Drinfeld that $GRT(\mathbb K)$ acts freely and transitively on $M(\mathbb K)$
from the left,
$GT(\mathbb K)$ acts freely and transitively on $M(\mathbb K)$ from the right,
and these two actions are commutative:

\begin{prop}[\cite{Dr}]\label{GRT-GT-bitorsor}
The associator set $M(\mathbb K)$ forms a 
$\left(GRT(\mathbb K),GT(\mathbb K)\right)$-bitorsor by 
the left $GRT(\mathbb K)$-action given by 
$$
(c,g)\circ(\mu,\varphi):=\left(\frac{\mu}{c}, \ \varphi\left(g\frac{A}{c}g^{-1},\ \frac{B}{c}\right) \cdot g \right)
=\left(\frac{\mu}{c}, \ g\cdot \varphi\left(\frac{A}{c},\ g^{-1}\frac{B}{c}g\right)  \right)
$$
for $(c,g)\in GRT(\mathbb K)$ and  $(\mu,\varphi)\in M(\mathbb K)$ and 
the right $GT(\mathbb K)$-action given by 
$$
(\mu,\varphi)\circ (\lambda,f):=\left(\lambda\mu, \ f(\varphi e^{\mu A}\varphi^{-1},e^{\mu B})\cdot\varphi\right)
=\left(\lambda\mu, \ \varphi\cdot f( e^{\mu A},\varphi^{-1}e^{\mu B}\varphi)\right)
$$
for $(\mu,\varphi)\in M(\mathbb K)$ and $(\lambda,f)\in GT(\mathbb K)$.
\end{prop}

We must note again that we reverse the order of the product given in the paper \cite{Dr}
for our purpose.
Drinfeld \cite{Dr} showed that associators give an isomorphism between
$\widehat{{\mathbb K}[B_n]}$ and $\widehat{U\frak b_n}$.

\begin{prop}[\cite{Dr}]\label{M to Isom}
Let $n\geqslant 2$.
The $\left(GRT(\mathbb K),GT(\mathbb K)\right)$-bitorsor $M(\mathbb K)$
is mapped to the $\left(\mathrm{Aut}(\widehat{U\frak b_n}),\
\mathrm{Aut}(\widehat{{\mathbb K}[B_n]})\right)$-bitorsor
$\mathrm{Isom}( \widehat{{\mathbb K}[B_n]}, \
\widehat{U\frak b_n})$ by the map
\begin{equation}\label{M to Isom KB UB}
\rho_n: M(\mathbb K)\to 
\mathrm{Isom}\left( \widehat{{\mathbb K}[B_n]}, \
\widehat{U\frak b_n}\right)
\end{equation}
induced by, for each $p=(\mu,\varphi)$,
\begin{equation*}
\rho_n(p):
\begin{cases}
\sigma_1 \quad \mapsto  &\tau_{1,2}\cdot \exp\left\{\frac{\mu t_{12}}{2}\right\}, \\
\sigma_i  \quad \mapsto  \quad
\varphi_{1\cdots i-1,i,i+1}^{-1}\cdot &\tau_{i,i+1}\cdot\exp\left\{\frac{\mu t_{i,i+1}}{2}\right\}
\cdot\varphi_{1\cdots i-1,i,i+1}
\ (2\leqslant i\leqslant n-1).
\end{cases}
\end{equation*}
It is a morphism as bitorsors, i.e. it
is compatible with  \eqref{GT to Aut B_n} and \eqref{GRT to Aut Ub_n}.
\end{prop}

We again note that $\rho_n$ is injective when $n\geqslant 3$.

\subsection{The motivic Galois group}
\label{Formulation of motivic Galois group}
We briefly review the formulations of the motivic Galois groups and
their torsor
(consult also \cite{A} as a nice exposition).
We also review their relationship with the torsor of
the Grothendieck-Teichm\"{u}ller groups discussed in our previous subsections.

The triangulated category $DM(\mathbb Q)_{\mathbb Q}$
of {\it mixed motives} over $\mathbb Q$
(a part of an idea of mixed motives is explained in \cite{De} \S 1)
was
constructed by Hanamura, Levine and Voevodsky.
{\it Tate motives} $\mathbb Q(n)$ ($n\in\mathbb Z$) are
(Tate) objects of the category.
Let $DMT(\mathbb Q)_{\mathbb Q}$ be the triangulated sub-category of 
$DM(\mathbb Q)_{\mathbb Q}$ generated by Tate  motives $\mathbb Q(n)$ ($n\in\mathbb Z$).
By the work of Levine a neutral tannakian $\mathbb Q$-category 
$MT(\mathbb Q)=MT(\mathbb Q)_{\mathbb Q}$ of {\it mixed Tate motives over $\mathbb Q$} 
is extracted by taking the heart with respect to a $t$-structure of
$DMT(\mathbb Q)_{\mathbb Q}$.
Deligne and Goncharov \cite{DG} introduced the full subcategory
$MT(\mathbb Z)=MT(\mathbb Z)_\mathbb Q$
of {\it 
 mixed Tate motives over ${\rm Spec}\;\mathbb Z$} inside of $MT(\mathbb Q)_{\mathbb Q}$.
%
The category $MT(\mathbb Z)$ 
forms a neutral tannakian $\mathbb Q$-category 
and association of each object $M\in MT(\mathbb Z)$
with the underlying $\mathbb Q$-linear space of its 
Betti and de Rham realizations 
give the fiber functor $\omega_{\rm Be}$ and $\omega_{\rm DR}$ respectively.

\begin{defn}
For $*\in\{\text{Be, DR}\}$, the {\it motivic Galois group}
${\rm Gal}^{\mathcal M}_*(\mathbb Z)$
is defined to be the corresponding tannakian fundamental group of ${MT}(\mathbb Z)$,
that is, the pro-$\mathbb Q$-algebraic group defined by
$\underline{\rm Aut}^\otimes({MT}(\mathbb Z):\omega_*)$.
\end{defn}

For $*,*'\in\{\text{Be, DR}\}$,
we denote  the corresponding tannakian fundamental torsor
$\underline{\rm Isom}^\otimes({MT}(\mathbb Z):\omega_*,\omega_{*'})$ 
by ${\rm Gal}^{\mathcal M}_{*,*'}(\mathbb Z)$.
This is a $({\rm Gal}^{\mathcal M}_{*'}(\mathbb Z),{\rm Gal}^{\mathcal M}_{*}(\mathbb Z))$-bitorsor.
We note that ${\rm Gal}^{\mathcal M}_{*,*}(\mathbb Z)={\rm Gal}^{\mathcal M}_{*}(\mathbb Z)$.
%
By the fundamental theorem of tannakian category theory,  
each fiber functor $\omega_*$ induces an equivalence of categories
\begin{equation}\label{tannakian equivalence}
MT(\mathbb Z)\simeq \mathrm{Rep}\; \mathrm{Gal}^\mathcal{M}_*(\mathbb Z)
\end{equation}
where the right hand side stands for the category of finite dimensional $\mathbb Q$-vector spaces
equipped with ${\rm Gal}^\mathcal{M}_*(\mathbb Z)$-action.

\begin{rem}\label{free-remark}
For $*,*'\in\{\text{Be, DR}\}$,
the action of ${\rm Gal}_{*}^\mathcal{M}(\mathbb Z)$ 
on $\omega_*(\mathbb Q(1))\simeq\mathbb Q$
defines a surjection ${\rm Gal}^\mathcal{M}_*(\mathbb Z)\to\mathbb G_m$ 
and its kernel ${\rm Gal}_{*}^\mathcal{M}(\mathbb Z)_1$
is the unipotent radical of ${\rm Gal}_{*}^\mathcal{M}(\mathbb Z)$.
For $*={\rm DR}$,
there is a natural splitting $\tau:{\bf G}_m\to {\rm Gal}_{\rm DR}^\mathcal{M}(\mathbb Z)$
which gives a negative grading on its associated Lie algebra 
${\rm Lie}{\rm Gal}_{\rm DR}^\mathcal{M}(\mathbb Z)_1$.
By the axiom on the structure of the category $MT({\mathbb Z})$,
it is known that 
the Lie algebra is the graded Lie algebra
{\it freely} generated by one element in each degree $-3,-5,-7,\dots$.
(consult  \cite{De} \S 8 for the full story).
\end{rem}

The {\it motivic fundamental group}
$\pi_1^{\mathcal M}({\mathbb P}^1\backslash\{0,1,\infty\}
:\overrightarrow{01})$
constructed in \cite{DG} \S4
is a (pro-)object of $MT(\mathbb Z)$.
The KZ-associator (cf. Example \ref{two associators})
is essential in describing the Hodge realization of the motive
(cf. \cite{A, DG, F07}).
Since its Betti and de Rham realization is given by 
$F_2(\mathbb Q)$ and $\exp\frak f_2$,
the motivic Galois groups,
${\rm Gal}_{\rm Be}^\mathcal{M}(\mathbb Z)$ and
${\rm Gal}_{\rm DR}^\mathcal{M}(\mathbb Z)$,
acts there respectively.
The tannakian equivalence \eqref{tannakian equivalence}
induces a morphism of bitorsors
\begin{equation*}\label{map}
\Psi:{\rm Gal}_{\rm Be,DR}^\mathcal{M}(\mathbb Z)\to 
{\rm Isom}(F_2(\mathbb Q),\exp{\frak f}_2)
\end{equation*}
from the $({\rm Gal}^{\mathcal M}_{\rm DR}(\mathbb Z),{\rm Gal}^{\mathcal M}_{\rm Be}(\mathbb Z))$-bitorsor 
${\rm Gal}_{\rm Be,DR}^\mathcal{M}(\mathbb Z)$
to the $\left({\rm Aut}\exp{\frak f}_2,{\rm Aut}F_2(\mathbb Q)\right)$-bitorsor
${\rm Isom}(F_2(\mathbb Q),\exp{\frak f}_2)$.
%
%
The following has been conjectured (Deligne-Ihara conjecture) 
and finally proved by Brown
by using  Zagier's relation on MZV's. 

\begin{thm}[\cite{Br12}]\label{freeness}
The map $\Psi$ is injective.
\end{thm}

It is a proalgebraic group analogue of the so-called Bely\u{\i}'s theorem \cite{Be}
in the profinite group setting.
The theorem says that all unramified mixed Tate motives over 
${\rm Spec}\;{\mathbb Z}$ are
associated with MZV's.

\begin{rem}\label{rem-incl-Gal-M}
We recall that our
$\left(GRT(\mathbb Q),GT(\mathbb Q)\right)$-bitorsor $M(\mathbb Q)$
is naturally injected to the
$\left({\rm Aut}\exp{\frak f}_2,{\rm Aut}F_2(\mathbb Q)\right)$-bitorsor
${\rm Isom}(F_2(\mathbb Q),\exp{\frak f}_2)$:
$$
M(\mathbb Q)\hookrightarrow {\rm Isom}(F_2(\mathbb Q),\exp{\frak f}_2).
$$
As is explained in \cite{A, F07}, a certain geometric interpretations of 
the Grothendieck-Teichm\"{u}ller groups shows that
${\rm Im}\Psi$ is injected in $M(\mathbb Q)$ as bitorsors.
Thus by the above theorem, 
 $({\rm Gal}^{\mathcal M}_{\rm DR}(\mathbb Z),{\rm Gal}^{\mathcal M}_{\rm Be}(\mathbb Z))$-bitorsor 
${\rm Gal}_{\rm Be,DR}^\mathcal{M}(\mathbb Z)$
is mapped injectively to
$\left(GRT(\mathbb Q),GT(\mathbb Q)\right)$-bitorsor $M(\mathbb Q)$
as bitorsors:
\begin{equation}\label{incl-Gal-M}
{\rm Gal}_{\rm Be,DR}^\mathcal{M}(\mathbb Z)
\hookrightarrow M(\mathbb Q).
\end{equation}
The inclusion induces the one from 
${\rm LieGal}_{\rm DR}^\mathcal{M}(\mathbb Z)_1$ to $GRT(\mathbb Q)$.
By Remark \ref{free-remark} we get the claim in Remark \ref{free Lie subalgebra}.  
The $GT(\mathbb Q)$-action on $\widehat{{\mathbb Q}[B_n]}$
given in \eqref{GT to Aut B_n} induces
a ${\rm Gal}_{\rm Be}^\mathcal{M}(\mathbb Z)$-action there
and
$GRT(\mathbb Q)$-action on $\widehat{U\frak b_n}$
given in \eqref{GRT to Aut Ub_n} also induces
a ${\rm Gal}_{\rm DR}^\mathcal{M}(\mathbb Z)$-action there.
Hence by the equivalence \eqref{tannakian equivalence},
$\widehat{{\mathbb Q}[B_n]}$ is the Betti realization of 
a certain mixed Tate (pro-)motive over ${\rm Spec}\;{\mathbb Z}$,
while whose de Rham realization is given by $\widehat{U\frak b_n}$.
\end{rem}


\section{Proalgebraic tangles and chord diagrams}
\label{Proalgebraic tangles and chord diagrams}
We develop the actions of the Grothendieck-Teichm\"{u}ller groups
on proalgebraic braids and on infinitesimal braids 
explained in our previous section 
into the  ones on proalgebraic tangles and on chord diagrams
by following the method indicated in \cite{KT98}.
This section might be regarded as an extension of Bar-Natan's formalism
\cite{Bar98} on a relationship of the Grothendieck-Teichm\"{u}ller groups
with  proalgebraic braids into their relationship with proalgebraic
tangles.

\subsection{The $GT$-action}\label{sec: GT-action on proalgebraic tangles}
In this subsection we give  a review but with more detailed considerations
on the last appendix of both \cite{KT98} and \cite{KRT}
where an interesting $GT(\mathbb K)$-action on proalgebraic tangles and knots are briefly explained.
Proalgebraic tangles and proalgebraic knots are recalled in Definition \ref{definition of proalgebraic knots}.
They are shown in Proposition \ref{algebraic presentation of proalgebraic tangles}
to be described  by
proalgebraic pre-tangles (and knots)  introduced by our ABC-construction
in Definition \ref{definition of fundamental proalgebraic tangle}.
The $GT(\mathbb K)$-action on proalgebraic tangles is explained
in Definition \ref{GT-action on proalgebraic tangles}
and Proposition \ref{GT-action on KT}.
The induced $GT(\mathbb K)$-action on proalgebraic knots is discussed
in Proposition \ref{GT-action on KK}.
In Proposition \ref{proposition on GK to GKQl} 
we give a relationship of the $GT(\mathbb K)$-action on the proalgebraic knots
with the profinite $\widehat{GT}$-action on profinite knots which was
constructed in our previous paper \cite{F12}.

\begin{nota}\label{usual tangles}
Let $k,l\geqslant 0$.
Let $\epsilon=(\epsilon_1,\dots,\epsilon_k)$ and $\epsilon'=(\epsilon'_1,\dots,\epsilon'_l)$
be sequences (including the empty sequence $\emptyset$)
of  symbols $\uparrow$ and $\downarrow$.
An (oriented)
\footnote{We occasionally omit to mention it.
Throughout the paper all tangles are assumed to be oriented.}
 {\it tangle} of type $(\epsilon,\epsilon')$
means a smooth embedded compact oriented one-dimensional  real manifolds
in $[0,1]\times {\mathbb C}$
(hence it is a finite disjoint union of embedded one-dimensional intervals
and circles),
whose boundaries are $\{(1,1),\dots,(1,k),(0,1)\dots,(0,l)\}$
such that $\epsilon_i$ (resp. $\epsilon_j'$) is $\uparrow$ or $\downarrow$ 
if the tangle is oriented upwards or downwards at $(1,i)$  (resp. at $(0,j)$) respectively.
A {\it link} is a tangle of type $(\emptyset, \emptyset)$ , i.e. $k=l=0$,
and a {\it knot} means a link with a single connected component.
An {\it $n$-string link}
\footnote{
A string link is not a link in the sense of the previous sentence.
},
a string link with $n$-components, 
means a tangle with $\epsilon=\epsilon'$ and $k=l=n$
which consists of $n$-intervals connecting  $(0,i)$ with $(1,i)$
for each $1\leqslant i\leqslant n$ and no circles
(cf. Figure \ref{example of string link}).

We denote $\mathcal T$ to be the full set of isotopy classes of oriented tangles
and ${\mathcal T}_{\epsilon,\epsilon'}$ to be its subset 
consisting of tangles with type $(\epsilon,\epsilon')$.
Figure \ref{example of tangles} might help the readers 
to have a good understanding of the definition.
\begin{figure}[h]
\begin{tabular}{c}
  \begin{minipage}{0.45\hsize}
      \begin{center}
          \begin{tikzpicture}
                \draw[dotted] (0.5,0)--(3.5,0);

                 \draw[->] (1,0) ..controls(0.6,2.0) and (1.2,2.6)     ..(1.6,3);
                \draw[-] (1.6,3) ..controls(2.0,3.4) and (3,3)   ..(2.6,2.0);
                 \draw[color=white, line width=7pt](2.6,2.0) ..controls(2.2,1.2)  and (0.4,2.)      ..(1,4.0);
                \draw[->] (2.6,2.0) ..controls(2.2,1.2)  and (0.4,2.)      ..(1,4.0);

             \draw[color=white, line width=7pt] (2.0,3)-- (2,0.5);
                 \draw[->] (2.0,4.0)-- (2,3.3) ;
                 \draw[->]  (2,3)--(2,0);

                 \draw[color=white, line width=7pt](3,4) ..controls(3.5,2) and (3,0.5)  ..(1.1,0.5);
               \draw[->] (3,4) ..controls(3.5,2) and (3,0.5)  ..(1.1,0.5);
                 \draw[color=white, line width=7pt](0.7,0.5) ..controls(0.2,0.6) and (-0.1,1.8)  ..(1.6,1.4);
               \draw[-] (0.7,0.5) ..controls(0.2,0.6) and (-0.1,1.8)  ..(1.6,1.4);
                 \draw[color=white, line width=7pt](1.6,1.4) ..controls(2.4,1.2) and (2.8,0.5)  ..(3,0);
               \draw[->] (1.6,1.4) ..controls(2.4,1.2) and (2.8,0.5)  ..(3,0);

                \draw[dotted] (0.5,4)--(3.5,4.0);
        \end{tikzpicture}
              \caption{A string link}
              \label{example of string link}
     \end{center}
 \end{minipage}
 \begin{minipage}{0.55\hsize}
           \begin{tikzpicture}
                \draw[dotted] (0,0)--(4.5,0);

                 \draw[->] (2,0) ..controls(0.6,2.0) and (1.2,2.6)     ..(1.6,3);
                \draw[-] (1.6,3) ..controls(2.0,3.4) and (3,3)   ..(2.6,2.0);
                 \draw[color=white, line width=7pt](2.6,2.0) ..controls(2.2,1.2)  and (0.4,2.)      ..(1,4.0);
                \draw[->] (2.6,2.0) ..controls(2.2,1.2)  and (0.4,2.)      ..(1,4.0);

                 \draw[->] (2.0,4.0)--(1.82,3.34) ;
                 \draw[color=white, line width=7pt] (1.8,2.9)  arc (180:360:0.7);
                  \draw[->] (1.8,2.9)  arc (180:360:0.7);
                  \draw[->] (3.2,2.9)--(3,4) ;

               \draw[->] (1.4,0.6) ..controls(0.2,0.2) and (-0.1,0.5)  ..(0.2,0.8);
                 \draw[color=white, line width=7pt](0.2,0.8) ..controls(0.6,1.2) and (1.2,1.6)  ..(1.6,1.4);
                \draw[->](0.2,0.8) ..controls(0.6,1.2) and (1.2,1.6)  ..(1.6,1.4);
              \draw[-] (1.6,1.4) ..controls(2.0,1.0) and (1.8,0.9)  ..(1.7,0.8);

                 \draw[color=white, line width=7pt](1,0) ..controls(1.,0.8) and (3.8,-0.5)  ..(4,4);
                 \draw[<-] (1,0) ..controls(1.,0.8) and (3.8,-0.5)  ..(4,4);

                \draw[dotted] (0,4)--(4.5,4.0);
            \end{tikzpicture}
\caption{A tangle in ${\mathcal T}_{\uparrow\downarrow\uparrow\downarrow, \downarrow\uparrow}$}
\label{example of tangles}  
     \end{minipage}
\end{tabular}
\end{figure}
It is easily seen that there is a natural composition map
\begin{equation}\label{composition in discrete case}
\cdot:
{\mathcal T}_{\epsilon_1,\epsilon_2}\times
{\mathcal T}_{\epsilon_2,\epsilon_3}\to
{\mathcal T}_{\epsilon_1,\epsilon_3}
\end{equation}
for any sequences $\epsilon_1,\epsilon_2,\epsilon_3$.
The set $\mathcal{SL}_{\epsilon}$ denotes  the subspace of ${\mathcal T}_{\epsilon,\epsilon}$
consisting of string links.
By the above composition
$\mathcal{SL}_{\epsilon}$ forms a monoid for each $\epsilon$.
By putting on each $i$-th strand an orientation $\epsilon_i$,  
the pure braid group $P_n$ ($n>1$) may be regarded as a submonoid of $\mathcal{SL}_{\epsilon}$
with $\epsilon=(\epsilon_1,\dots,\epsilon_n)$.
By definition ${\mathcal T}_{\emptyset,\emptyset}$ is the set of isotopy classes of (oriented) links.
We denote  $\mathcal K$ to be its subspace 
consisting isotopy classes of (oriented) knots.
The set $\mathcal K$ forms a monoid by the connected sum (the knot sum)
\begin{equation}\label{connected sum in discrete case}
\sharp: {\mathcal K}\times{\mathcal K}\to {\mathcal K}.
\end{equation}
It  is a natural way to fuse
two oriented knots, with an appropriate position of orientation,
into one 
(an example is illustrated in Figure \ref{Example of connected sum}).
\begin{figure}[h]
\begin{center}
\begin{tikzpicture}
                 \draw[<-] (0.8,2.2) arc (0:-90:0.6);
                  \draw[<-](0.2,1.6) arc (-90:-180:0.6);
                 \draw [<-] (0,2.0)  arc (270:180:0.2);
                 \draw[color=white, line width=5pt] (0,2.4) ..controls(0.4,2.4) and (0.4,2.0)..(0.8,2.0);
                 \draw[<-] (0,2.4) ..controls(0.4,2.4) and (0.4,2.0)..(0.8,2.0);
                 \draw[color=white, line width=5pt]  (0,2.0) ..controls(0.4,2.0) and (0.4,2.4) ..(0.8, 2.4);
                \draw[-] (0,2.0) ..controls(0.4,2.0) and (0.4,2.4) ..(0.8, 2.4);
                 \draw (0.8,2.0)  arc (-90:90:0.2);
                 \draw[color=white, line width=5pt] (-0.2,2.2) arc (180:0:0.5);  
                 \draw (-0.2,2.2) arc (180:0:0.5);  
                \draw[color=white, line width=5pt]  (-0.4,2.2) ..controls (-0.4,2.4) and (-0.1,2.4) ..(0,2.4);
                  \draw[<-] (-0.4,2.2) ..controls (-0.4,2.4) and (-0.1,2.4) ..(0,2.4);
              \draw[color=white, line width=5pt] (-0.1,0.1).. controls  (-0.3,0.1)  and (-0.1,0.8)..(0.2,0.8);
                \draw [<-] (-0.1,0.1).. controls  (-0.3,0.1)  and (-0.1,0.8)..(0.2,0.8); 
              \draw[color=white, line width=5pt]  (0.5,0.1).. controls  (0.1,0.1)  and (-0.1,1.4)..(0.2,1.4);
                \draw[<-]  (0.5,0.1).. controls  (0.1,0.1)  and (-0.1,1.4)..(0.2,1.4);
              \draw[color=white, line width=5pt] (-0.1,0.1).. controls  (0.3,0.1)  and (0.5,1.4)..(0.2,1.4);
                \draw[->] (-0.1,0.1).. controls  (0.3,0.1)  and (0.5,1.4)..(0.2,1.4);
               \draw[color=white, line width=5pt](0.5,0.1).. controls  (0.7,0.1)  and (0.5,0.8)..(0.2,0.8);
                \draw  (0.5,0.1).. controls  (0.7,0.1)  and (0.5,0.8)..(0.2,0.8);
  \draw[dotted]   (0.2,1.5) circle (0.3);
                \draw (2.2,1.2) node{$\Longrightarrow$};

                 \draw[<-] (4.8,2.2) .. controls (4.8,1.9) and  (4.7,1.6) ..(4.5,1.6);
                 \draw[->] (3.6,2.2) .. controls (3.6,1.9) and  (3.7,1.6) ..(3.9,1.6);
                 \draw [<-] (4.0,2.0)  arc (270:180:0.2);
                 \draw[color=white, line width=5pt] (4.0,2.4) ..controls(4.4,2.4) and (4.4,2.0)..(4.8,2.0);
                 \draw[<-] (4.0,2.4) ..controls(4.4,2.4) and (4.4,2.0)..(4.8,2.0);
                 \draw[color=white, line width=5pt]  (4.0,2.0) ..controls(4.4,2.0) and (4.4,2.4) ..(4.8, 2.4);
                \draw[-] (4.0,2.0) ..controls(4.4,2.0) and (4.4,2.4) ..(4.8, 2.4);
                 \draw (4.8,2.0)  arc (-90:90:0.2);
                 \draw[color=white, line width=5pt] (3.8,2.2) arc (180:0:0.5);  
                 \draw (3.8,2.2) arc (180:0:0.5);  
                \draw[color=white, line width=5pt]  (4-0.4,2.2) ..controls (4-0.4,2.4) and (4-0.1,2.4) ..(4.0,2.4);
                  \draw[<-] (4-0.4,2.2) ..controls (4-0.4,2.4) and (4-0.1,2.4) ..(4.0,2.4);
  \draw[-] (4.5,1.6) .. controls (4.7,1.6) and  (4.3,1.6) ..(4.3,1.4);
  \draw[-] (4-0.1,1.6) .. controls (4-0.3,1.6) and  (4.1,1.6) ..(4.1,1.4);
              \draw[color=white, line width=5pt] (4-0.1,0.1).. controls  (4-0.3,0.1)  and (4-0.1,0.8)..(4.2,0.8);
                \draw [<-] (4-0.1,0.1).. controls  (4-0.3,0.1)  and (4-0.1,0.8)..(4.2,0.8); 
              \draw[color=white, line width=3pt] (4.5,0.1).. controls  (4.1,0.1)  and (4.1,1.4)..(4.1,1.4); 
                \draw[<-]  (4.5,0.1).. controls  (4.1,0.1)  and (4.1,1.4)..(4.1,1.4); 
              \draw[color=white, line width=3pt](4-0.1,0.1).. controls  (4.3,0.1)  and (4.3,1.4)..(4.3,1.4); 
                \draw[->] (4-0.1,0.1).. controls  (4.3,0.1)  and (4.3,1.4)..(4.3,1.4); 
               \draw[color=white, line width=5pt](4.5,0.1).. controls  (4.7,0.1)  and (4.5,0.8)..(4.2,0.8);
                \draw  (4.5,0.1).. controls  (4.7,0.1)  and (4.5,0.8)..(4.2,0.8); 
\end{tikzpicture}
\end{center}
\caption{Connected sum (knot sum)}
\label{Example of connected sum}
\end{figure}
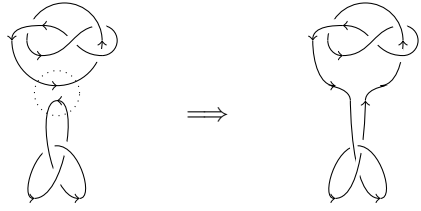
It can be done at any points.
Our short caution is that the connected sum $\sharp$ in \eqref{connected sum in discrete case}
is different from the composition  $\cdot$ in \eqref{composition in discrete case}.
(In fact a composition of two knots in not a knot but  a link).
\end{nota}

There is a  fundamental  identification between knots and long knots (string link with a single component).

\begin{prop}\label{identification between long knots and knots}
Let $\epsilon=\uparrow$ or $\downarrow$. Then 
the set $\mathcal{SL}_{\epsilon}$
of long knots with the composition $\cdot$
is identified with
the set  $\mathcal K$ with the connected sum $\sharp$
by closing the two endpoints of each.
Namely we have an identification of two monoids
$$\mathrm{cl}:(\mathcal{SL}_{\uparrow},\cdot)\simeq (\mathcal{K},\sharp).$$
\end{prop}

\begin{proof}
The identification 
is simply obtained by combining the ends of long knots.
Checking all compatibilities are easy to see.
\end{proof}

For more on tangles, consult the standard textbook, say, \cite{CDM}.

\begin{defn}[\cite{F12,KT98}]
\label{definition of proalgebraic knots}
(1)
Let ${\mathbb K}[{\mathcal T}_{\epsilon,\epsilon'}]$ be the free $\mathbb K$-module of finite formal sums
of elements of ${\mathcal T}_{\epsilon,\epsilon'}$.
A singular oriented tangle,
an \lq oriented tangle' 
allowed to have  a finite number of transversal double points (see \cite{KT98} for datail),
determines an element of ${\mathbb K}[{\mathcal T}_{\epsilon,\epsilon'}]$ 
by the desingularization of each double point by the following relation
$$
\diaProjected=\diaCrossP -\diaCrossN.
$$
Let ${\mathcal T}_n$ ($n\geqslant 0$) be the $\mathbb K$-submodule of 
$\mathbb K[{\mathcal T}_{\epsilon,\epsilon'}]$ generated by
all singular oriented tangles with type $(\epsilon,\epsilon')$
and with $n$ double points.
The descending filtration $\{{\mathcal T}_n\}_{n\geqslant 0}$
is called the {\it singular filtration}.
The topological $\mathbb K$-module 
$\widehat{{\mathbb K}[{\mathcal T}_{\epsilon,\epsilon'}]}$ 
of {\it proalgebraic tangles} of type $(\epsilon,\epsilon')$
means its completion
with respect to the singular filtration:
$$
\widehat{{\mathbb K}[{\mathcal T}_{\epsilon,\epsilon'}]}:=
\underset{N}{\varprojlim} \ {\mathbb K}[{\mathcal T}_{\epsilon,\epsilon'}]/{\mathcal T}_N.
$$
By abuse of notation, we denote the induced filtration on
$\widehat{{\mathbb K}[{\mathcal T}_{\epsilon,\epsilon'}]}$
by the same symbol $\{{\mathcal T}_n\}_{n\geqslant 0}$.
Note that there is a natural composition map
\begin{equation}\label{proalgebraic tangle composition}
\cdot:
\widehat{{\mathbb K}[{\mathcal T}_{\epsilon_1,\epsilon_2}]}\times
\widehat{{\mathbb K}[{\mathcal T}_{\epsilon_2,\epsilon_3}]}\to
\widehat{{\mathbb K}[{\mathcal T}_{\epsilon_1,\epsilon_3}]}
\end{equation}
for  any $\epsilon_1,\epsilon_2$ and $\epsilon_3$.
We denote  the collection of
$\widehat{{\mathbb K}[{\mathcal T}_{\epsilon,\epsilon'}]}$
for all  $\epsilon$ and $\epsilon'$
by $\widehat{{\mathcal T}_{\mathbb K}}$.

(2)
Let ${\mathbb K}[{\mathcal K}]$ be the  $\mathbb K$-submodule of ${\mathbb K}[{\mathcal T}_{\emptyset,\emptyset}]$ generated by ${\mathcal K}$.
By the product
\begin{equation}\label{connected sum in proalgebraic knots}
\sharp:\widehat{{\mathbb K}[{\mathcal K}]}\times
\widehat{{\mathbb K}[{\mathcal K}]}\to
\widehat{{\mathbb K}[{\mathcal K}]}
\end{equation}
induced by the {connected sum} $\sharp$ in \eqref{connected sum in discrete case}
and the coproduct map
$\Delta: {\mathbb K}[{\mathcal K}]\to {\mathbb K}[{\mathcal K}]\otimes_{\mathcal K}{\mathbb K}[{\mathcal K}]$
sending $k\mapsto k\otimes k$
and the augmentation map ${\mathbb K}[{\mathcal K}]\to {\mathcal K}$,
it carries a structure of co-commutative and commutative bi-algebra.
Put ${\mathcal K}_n:={\mathcal T}_n\cap{\mathbb K}[{\mathcal K}]$ ($n\geqslant 0$).
Then ${\mathcal K}_n$ forms an ideal of $\mathbb K[{\mathcal K}]$
and the descending filtration $\{{\mathcal K}_n\}_{n\geqslant 0}$
is called the {\it singular knot filtration} (cf. loc.cit.).
The topological commutative $\mathbb K$-algebra $\widehat{{\mathbb K}[{\mathcal K}]}$ of {\it proalgebraic knots} means its completion
with respect to the singular knot filtration:
$$
\widehat{{\mathbb K}[{\mathcal K}]}:=
\underset{N}{\varprojlim} \ {\mathbb K}[{\mathcal K}]/{\mathcal K}_N.
$$
It is a $\mathbb K$-linear subspace of $\widehat{{\mathbb K}[{\mathcal T}_{\emptyset\emptyset}]}$. 
Since each element $\gamma$ in $\mathcal K$ is congruent to the unit
$\orientedcircle$ modulo $\mathcal K_1$
due to the finiteness property of the unknotting number (cf. \cite{CDM}),
the inverse of $\gamma$ with respect to $\sharp$ always exists in
$\widehat{{\mathbb K}[{\mathcal K}]}$.
It defines the antipode map on $\widehat{{\mathbb K}[{\mathcal K}]}$,
which  yields a structure of co-commutative and commutative Hopf algebra
there.
Again by abuse of notation, we denote the induced filtration of
$\widehat{{\mathbb K}[{\mathcal K}]}$
by the same symbol $\{{\mathcal K}_n\}_{n\geqslant 0}$,
which is compatible with  a structure of filtered Hopf algebra on 
$\widehat{{\mathbb K}[{\mathcal K}]}$.

(3)
{\it Proalgebraic links} and {\it proalgebraic string links} 
can be defined in the same way.
The subset $\widehat{{\mathbb K}[{\mathcal{SL}}_{\epsilon}]}$
of  $\widehat{{\mathbb K}[{\mathcal{T}}_{\epsilon,\epsilon}]}$
consisting of proalgebraic string links forms a non-commutative
$\mathbb K$-algebra by the composition map \eqref{proalgebraic tangle composition}.
\end{defn}

Here is a fundamental identification in our proalgebraic setting.

\begin{lem}\label{identification between proalgebraic long knots and proalgebraic knots}
Let $\epsilon=\uparrow$ or $\downarrow$. 
There is an identification between two $\mathbb K$-algebras
\begin{equation}\label{cl on proalgebraic knots}
\mathrm{cl}:(\widehat{{\mathbb K}[{\mathcal{SL}}_{\epsilon}]},\cdot)\simeq 
(\widehat{{\mathbb K}[{\mathcal{K}}]},\sharp).
\end{equation}
which is compatible with their filtrations.
\end{lem}

\begin{proof}
It is an immediate  corollary of  Proposition \ref{identification between long knots and knots}.
\end{proof}

We give a piecewise construction of the above proalgebraic tangles
by using proalgebraic pre-tangles introduced below.

\begin{defn}\label{definition of fundamental proalgebraic tangle}
(1)
A {\it fundamental proalgebraic (oriented)
tangle}  means a vector belonging to an {\it ABC-space},
one of
the following $\mathbb K$-linear spaces
$A_{k,l}^\epsilon$, $\widehat{B^{\epsilon}_\tau}$ and $C_{k,l}^\epsilon$
for some $k,l,\epsilon, \tau$:
\begin{align*}
&A_{k,l}^\epsilon :={\mathbb K}\cdot a_{k,l}^\epsilon, \
\text{with} \ 
\epsilon=(\epsilon_i)_{i=1}^{k+l+1}\in
\{\uparrow, \downarrow\}^{k}
\times \{\annihilation, \opannihilation\}
\times\{\uparrow, \downarrow\}^{l} \ 
(k,l=0,1,2,\dots) ,
\\
&\widehat{B^{\epsilon}_\tau} :=
\widehat{{\mathbb K}[P_n]}\cdot \tau \
\text{with} \ 
\epsilon=(\epsilon_i)_{i=1}^{n}
\in \{\uparrow, \downarrow\}^{n} \
\text{and} \ 
\tau\in{\frak S}_n \
(n=1,2,3,4,\dots ),
\\
&C_{k,l}^\epsilon :={\mathbb K}\cdot c_{k,l}^\epsilon, \
\text{with} \
\epsilon=(\epsilon_i)_{i=1}^{k+l+1}\in
\{\uparrow, \downarrow\}^{k}
\times \{\creation,\opcreation\}
\times\{\uparrow, \downarrow\}^{l} \
(k,l=0,1,2,\dots ).
\end{align*}
Here 
$\widehat{{\mathbb K}[P_n]}\cdot \tau$ stands for the coset
of $\widehat{{\mathbb K}[B_n]}\bigm/ \widehat{{\mathbb K}[P_n]}$
corresponding to $\tau\in{\frak S}_n=B_n/P_n$,
the inverse image of ${\mathbb K}\cdot\sigma$ under the natural homomorphism
$\widehat{{\mathbb K}[B_n]}\to {\mathbb K}[{\frak S}_n]$.
To stress that an element $b$ belongs to $\widehat{B^{\epsilon}_\tau}$,
we occasionally denote $b^\epsilon$ or $(b,\epsilon)$.

For each ABC-space $V$, 
its {\it source} $s(V)$ and {\it target} $t(V)$,
which are sequences of $\uparrow$ and $\downarrow$, 
are  defined in a completely same way to \cite{F12};  
e.g. $s(\widehat{B^{\epsilon}_\tau})=\epsilon$, $t(\widehat{B^{\epsilon}_\tau})=\tau(\epsilon)$.

(2)
A {\it proalgebraic pre-tangle} $\Gamma$ of type $(\epsilon,\epsilon')$
is a vector belonging to a $\mathbb K$-linear space which is
a finite consistent (successively composable)
$\mathbb K$-linear tensor product of  ABC-spaces.
Namely it is an element belonging to a $\mathbb K$-linear space
$V_n\otimes\cdots V_2\otimes V_1$ for some $n$
such that $s(V_{i+1})=t(V_i)$ for all $i=1,2,\dots,n-1$
and $s(V_1)=\epsilon$ and $t(V_n)=\epsilon'$.
For our simplicity, we denote each of its element 
$\Gamma=\gamma_n\otimes\cdots\otimes\gamma_2\otimes\gamma_1$ with $\gamma_i\in V_i$
by
$\Gamma=\gamma_n\cdots\gamma_2\cdot\gamma_1$.
We also define $s(\Gamma):=s(V_1)$ and $t(\Gamma):=t(V_n)$.
A {\it proalgebraic pre-link} $\Gamma$ is a proalgebraic pre-tangle with $s(\Gamma)=t(\Gamma)=\emptyset$.
Two proalgebraic pre-tangles 
$\Gamma=\gamma_n\cdots\gamma_2\cdot\gamma_1$ and
$\Gamma'=\gamma'_m\cdots\gamma'_2\cdot\gamma'_1$
are called {\it composable} when $s(\Gamma)=t(\Gamma')$.
Their {\it composition} $\Gamma\cdot\Gamma'$ is defined by
$\gamma_n\cdots\gamma_2\cdot\gamma_1\cdot 
\gamma'_m\cdots\gamma'_2\cdot\gamma'_1$.

For each ABC-space $V$, 
its {\it skeleton} ${\mathbb S}(V)$ is defined in a completely same way to \cite{F12}.  
For a proalgebraic pre-tangle  $\Gamma=\gamma_n\cdots\gamma_2\cdot\gamma_1$ with $\gamma_i\in V_i$,
its {\it skeleton} ${\mathbb S}(\Gamma)$ stands for the graph of the compositions
${\mathbb S}(V_n) \cdots{\mathbb S}(V_2)\cdot{\mathbb S}(V_1)$ and
its {\it connected components}  mean the connected components of ${\mathbb S}(\Gamma)$
as graphs.
A {\it proalgebraic pre-knot} is a proalgebraic pre-link with a single connected component.
A {\it proalgebraic pre-string link} of type $\epsilon=(\epsilon_i)_{i=1}^n$ is 
a proalgebraic pre-tangle of type $(\epsilon,\epsilon)$
whose connected components consist of  $n$ intervals
connecting each $i$-th point on the bottom to the one on the top.

(3)
Two proalgebraic pre-tangles are called {\it isotopic} when they are related by a finite number of
the 6 moves replacing profinite tangles and profinite braids group $\widehat{B_n}$
by proalgebraic pre-tangles  and proalgebraic braid algebras 
$\widehat{{\mathbb K}[B_n]}$ in (T1)-(T6) in \cite{F12} 
and $c\in\widehat{\mathbb Z}$ by $c\in{\mathbb K}$ in (T6).
(N.B. We note that $\sigma_i^c$ for $c\in{\mathbb K}$ makes sense in
$\widehat{{\mathbb K}[B_n]}$
by the reason explained in \cite{F12} Proof of Proposition 2.29 (2).)

(4)
We denote ${\mathbb K}[\mathcal{T}_{\epsilon,\epsilon'}^{\rm pre}]$ 
to be the $\mathbb K$-linear space which is the quotient of 
the $\mathbb K$-span of proalgebraic pre-tangles with type $(\epsilon,\epsilon')$
divided by the equivalence linearly generated by the above
isotopy.
Note that there is a natural composition map
\begin{equation}\label{composition of Tpre}
\cdot:
{\mathbb K}[{\mathcal T}_{\epsilon_1,\epsilon_2}^{\rm pre}]\times
{\mathbb K}[{\mathcal T}_{\epsilon_2,\epsilon_3}^{\rm pre}]\to
{\mathbb K}[{\mathcal T}_{\epsilon_1,\epsilon_3}^{\rm pre}]
\end{equation}
for  any $\epsilon_1,\epsilon_2$ and $\epsilon_3$.
The subset ${\mathbb K}[{\mathcal{SL}}_{\epsilon}^{\rm pre}]$
of ${\mathbb K}[{\mathcal T}_{\epsilon,\epsilon}^{\rm pre}]$
consisting of proalgebraic pre-string  links of type $\epsilon$
forms a  non-commutative $\mathbb K$-algebra  by the composition map.

The symbol $\mathbb K[\mathcal{K}^{\rm pre}]$ stands for the  subspace of
$\mathbb K[\mathcal{T}_{\emptyset,\emptyset}^{\rm pre}]$ generated by proalgebraic pre-knots.
It can be proven in a same way to \cite{F12} that  $\mathbb K[\mathcal{K}^{\rm pre}]$
inherits a structure of  a commutative $\mathbb K$-algebra by the connected sum 
\begin{equation}\label{connected sum on KK-pre}
\sharp:\mathbb K[\mathcal{K}^{\rm pre}]\times \mathbb K[\mathcal{K}^{\rm pre}]\to
\mathbb K[\mathcal{K}^{\rm pre}].
\end{equation}
Here for  any two proalgebraic knots $K_1=\alpha_m\cdots\alpha_1$ and $K_2=\beta_n\cdots\beta_1$
with
$(\alpha_m,\alpha_1)=(\opannihilation,\creation)$
and
$(\beta_n,\beta_1)=(\opannihilation,\creation)$
(we may assume such presentations by (T6)),
their {connected sum} is defined by
\begin{equation*}
K_1\sharp K_2:=\alpha_{m}\cdots\alpha_2\cdot\beta_{n-1}\cdots\beta_1.
\end{equation*}
Again we note that $\sharp$ is different from the above composition \eqref{composition of Tpre}.

(5)
For a  $\mathbb K$-linear space $V$ with $V=\widehat{B^\epsilon_\tau}$,
we give its descending filtration $\{\mathcal T_N(V)\}_{N=0}^\infty$ of
$\mathbb K$-linear subspaces by $\mathcal T_N(V):=I^N$
and for a $\mathbb K$-linear space $V$ with 
$V=A^\epsilon_{k,l}$ or $C^\epsilon_{k,l}$
we give its filtration by ${\mathcal T}_0(V)=V$ and $\mathcal T_N(V):=\{0\}$ for $N>0$.
For a finite consistent tensor product $V=V_n\otimes\cdots\otimes V_1$ of ABC-spaces,
we give its descending filtration $\{\mathcal T_N(V)\}_{N=0}^\infty$ 
with 
$$
\mathcal T_N(V):=\sum_{i_n+\cdots+i_1=N}{\mathcal T}_{i_n}(V_n)\otimes\cdots\otimes{\mathcal T}_{i_1}(V_1),
$$
i.e.
the $\mathbb K$-linear subspaces generated by 
the subspaces ${\mathcal T}_{i_n}(V_n)\otimes\cdots\otimes{\mathcal T}_{i_1}(V_1)$
with $i_n+\cdots+i_1=N$. 
For any $\epsilon$ and $\epsilon'$,
the collection of the filtrations of  such consistent tensor product
$V$ with $s(V)=\epsilon'$ and $t(V)=\epsilon$
yields
a  filtration of  $\mathbb K$-submodules of
$\mathbb K[\mathcal{T}_{\epsilon,\epsilon'}^{\rm pre}]$,
which we denote by $\{\mathcal T^{\rm pre}_N\}_{N\geqslant 0}$.
They are compatible  with the composition map \eqref{composition of Tpre}.
The filtration on $\mathbb K[\mathcal{K}^{\rm pre}]$ induced by the filtration 
of  $\mathbb K[\mathcal{T}_{\emptyset,\emptyset}^{\rm pre}]$
is denoted by $\{\mathcal K^{\rm pre}_N\}_{N\geqslant 0}$.
The filtration is compatible with its algebra structure given by
\eqref{connected sum on KK-pre}.
\end{defn}

\begin{lem}
(1)
For any $\epsilon$ and $\epsilon'$ and any $N\geqslant 0$, there is an isomorphism
of $\mathbb K$-linear spaces
\begin{equation*}\label{two maps}
{\mathbb K}[{\mathcal T}_{\epsilon,\epsilon'}]/{\mathcal T}_N
\simeq \mathbb K[\mathcal{T}_{\epsilon,\epsilon'}^{\rm pre}]/{\mathcal T}^{\rm pre}_N.
\end{equation*}

(2)
For any $\epsilon$  and any $N\geqslant 0$, there is an isomorphism
of non-commutative $\mathbb K$-algebras
\begin{equation*}
{\mathbb K}[{\mathcal{SL}}_{\epsilon}]/{\mathcal T}_N
\simeq \mathbb K[\mathcal{SL}_{\epsilon}^{\rm pre}]/{\mathcal T}^{\rm pre}_N.
\end{equation*}

(3)
For any $N\geqslant 0$,
there is an isomorphism 
of commutative $\mathbb K$-algebras
\begin{equation*}
{\mathbb K}[{\mathcal K}]/{\mathcal K}_N
\simeq\mathbb K[\mathcal{K}^{\rm pre}]/{\mathcal K}_N^{\rm pre}.
\end{equation*}
\end{lem}

\begin{proof}
(1)
The map
is obtained because the set  $\mathcal T$ 
is described by the  sequence of elements in the discrete sets $A$, $B$ and $C$
modulo the discrete version of our moves (T1)-(T6)
(see \cite{F12}).
Showing that it is an isomorphism is attained by the isomorphism
$$
{\mathbb K}[B_n]/I^i\simeq\widehat{{\mathbb K}[B_n]}/I^i
$$
for $i=0,1,2,\dots.$

(2) and (3) It is a direct consequence of (1).
\end{proof}

Here are algebraic reformulations of proalgebraic tangles and knots.
\begin{prop}\label{algebraic presentation of proalgebraic tangles}
(1)
For each $\epsilon$ and $\epsilon'$, there is an identification of filtered $K$-linear spaces
$$
\varprojlim_n \ \mathbb K[\mathcal{T}_{\epsilon,\epsilon'}^{\rm pre}]/{\mathcal T}^{\rm pre}_n
\simeq \widehat{{\mathbb K}[{\mathcal T}_{\epsilon,\epsilon'}]}.
$$

(2)
For each $\epsilon$, there is an identification of filtered non-commutative $\mathbb K$-algebras
$$
\varprojlim_n \ \mathbb K[\mathcal{SL}_{\epsilon}^{\rm pre}]/{\mathcal T}^{\rm pre}_n
\simeq \widehat{{\mathbb K}[{\mathcal{SL}}_{\epsilon}]}.
$$

(3)
There is an identification of filtered commutative $\mathbb K$-algebras
$$
\varprojlim_n \ \mathbb K[\mathcal{K}^{\rm pre}]/{\mathcal K}^{\rm pre}_n
\simeq \widehat{{\mathbb K}[{\mathcal K}]}.
$$
\end{prop}

\begin{proof}
It is immediate by Definition \ref{definition of proalgebraic knots} and the above lemma.
\end{proof}

\begin{rem}\label{KP injected to KSL}
By \cite{Bar96} \S 4.2, we have a natural inclusion
$
\widehat{{\mathbb K}[{P_n}]}
\hookrightarrow\widehat{{\mathbb K}[{\mathcal{SL}}_{\uparrow^n}]}.
$
\end{rem}

Based on the identification,
the action of the  Grothendieck-Teichm\"{u}ller group $GT({\mathbb K})$ on
proalgebraic tangles
(in \cite{KRT, KT98} Appendix)
can be explained as follows:

\begin{defn}\label{GT-action on proalgebraic tangles}
Let ${\bar \Gamma}\in \widehat{{\mathbb K}[{\mathcal T}_{\epsilon,\epsilon'}]}/{\mathcal T}_N$
with any sequences $\epsilon,\epsilon'$ and $N\geqslant 0$.
Let $\Gamma$ be its representative in $\mathbb K[\mathcal{T}_{\epsilon,\epsilon'}^{\rm pre}]$
with a presentation $\Gamma=\gamma_{m}\cdots\gamma_{2}\cdot\gamma_{1}$
($\gamma_{j}$: fundamental proalgebraic tangle).
For $\sigma=(\lambda,f)\in {GT}({\mathbb K})$
with $\lambda\in{\mathbb K}^\times$ and $f\in {F}_2({\mathbb K})$,
we define 
\begin{equation}\label{GT-action on KT mod N}
\sigma(\bar \Gamma):=
\overline{\sigma(\gamma_{m})\cdots\sigma(\gamma_{2})\cdot\sigma(\gamma_{1})}
\in \widehat{{\mathbb K}[{\mathcal T}_{\epsilon,\epsilon'}]}/{\mathcal T}_N.   
\end{equation}
Here $\sigma(\gamma_j)$ is defined in a same way to \cite{F12} as follows.

(1)
If $\gamma_{j}\in A_{k,l}^\epsilon$, we define
$$\sigma(\gamma_{j}):=\gamma_{j}\cdot
(\nu_f)^{s(\gamma_{j})}_{k+2}
\cdot f_{1\cdots k,k+1,k+2}^{s(\gamma_{j})}.$$
Here the middle term stands for the proalgebraic tangle  whose source is $s(\gamma_{j})$
which is obtained by  putting the trivial braid with $k+1$-strands on the left of $\nu_f$ (see below) and
the trivial braid with $l$-strands on its right. 

(2)
If $\gamma_{j}=(b_n,\epsilon)\in \widehat{B_\tau^\epsilon}$ with 
$b_n\in  \widehat{{\mathbb K}[P_n]}\cdot\tau\subset \widehat{{\mathbb K}[B_n]}$, 
we define
$$\sigma(\gamma_{j}):=(\rho_n(\sigma)(b_n),\epsilon).$$
Here $\rho_n(\sigma)(b_n)$ is the image of $b_n$
by the action $\sigma\in GT(\mathbb K)$ on $\widehat{{\mathbb K}[B_n]}$
explained in \S \ref{GT-action on proalgebraic braids}.

(3)
If $\gamma_{j}\in C_{k,l}^\epsilon$, we define
$$\sigma(\gamma_{j}):=f_{1\cdots k,k+1,k+2}^{-1,t(\gamma_{j})}
\cdot\gamma_{j}.$$

The symbol $\nu_f^\epsilon$ ($\epsilon=\uparrow, \downarrow$) means
the proalgebraic  tangle in $\widehat{{\mathbb K}[{\mathcal T}_{\epsilon,\epsilon}]}$
which is a proalgebraic string link 
with a single strands  such that $s(\mu_f^\epsilon)=t(\mu_f^\epsilon)=\epsilon$
and which is given by the inverse of  $\Lambda_f^\epsilon$
with respect to the composition:
$$
\nu_f^\epsilon:=\{\Lambda_f^\epsilon\}^{-1}.
$$
Here $\Lambda_f^\downarrow$ in $\widehat{{\mathbb K}[{\mathcal T}_{\epsilon,\epsilon}]}$
represents the proalgebraic string link with a single strands given by
$$
\Lambda_f^\downarrow:=
a_{1,0}^{\downarrow\annihilation}
\cdot f^{\downarrow\uparrow\downarrow}\cdot 
c_{0,1}^{\creation\downarrow}.
$$
(cf. Figure \ref{proalgebraic error term})
and $\Lambda_f^\uparrow$ is the same one obtained by reversing its all arrows.
\end{defn}
\begin{figure}[h]
\begin{center}
         \begin{tikzpicture}
                    \draw (1.2,0.5) rectangle (2.4,0.9);
                    \draw (1.8,0.7) node{$f$};
                     \draw[<-] (1.4,0.2)--(1.4,0.5);
                    \draw[<-] (1.4,1.2)--(1.4,1.5) (1.4,0.9)--(1.4,1.2) ;
                      \draw (1.4,1.4) node{$\downarrow$};
                     \draw[->] (1.4,0.2)  arc (180:360:0.2);
                     \draw[->] (1.8,1.2)  arc (180:0:0.2);
                     \draw[->] (1.8,0.2)--(1.8,0.5);
                     \draw[->] (1.8,0.9)--(1.8,1.2);
                      \draw[<-] (2.2,-0.2)--(2.2,0.2) (2.2,0.2)--(2.2,0.5);
                      \draw[<-] (2.2,0.9)--(2.2,1.2);
                       \draw (2.2,0) node{$\downarrow$};
\draw[color=white, very thick] (1.3,-0.2)--(2.3,-0.2) (1.3,0.2)--(2.3,0.2) (1.3,1.2)--(2.3,1.2) (1.3,1.6)--(2.3,1.6) ;
         \end{tikzpicture}
\caption{$\Lambda_f^\downarrow$}
\label{proalgebraic error term}
\end{center}
\end{figure}
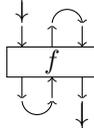

We note that the existence of the inverse of $\Lambda^\epsilon_f$ 
in $\widehat{{\mathbb K}[{\mathcal T}_{\epsilon,\epsilon}]}$ is immediate
because $\Lambda_f^\epsilon$ is congruent to the trivial braid with a string modulo
${{\mathcal T}_1}$
and $\widehat{{\mathbb K}[{\mathcal T}_{\epsilon,\epsilon}]}$ is completed by the filtration 
$\{\mathcal T_n\}_{n=0}^\infty$.

\begin{rem}
(1)
The inverse $\nu_f^\epsilon$ does not look exist in 
${\mathbb K}[{\mathcal T}_{\epsilon,\epsilon'}^{\rm pre}]$ generally. 
That is why we define 
$GT(\mathbb K)$-action on
$\widehat{{\mathbb K}[{\mathcal T}_{\epsilon,\epsilon'}]}$ below.

(2)
Our action \eqref{GT-action on KT mod N} looks slightly different from 
the one given in \cite{F12}.
One of the reasons is that we deal tangles
in Definition \ref{GT-action on proalgebraic tangles}
while we discuss knots in \cite{F12}.
\end{rem}


\begin{prop}[\cite{KRT, KT98}]\label{GT-action on KT}
The equation \eqref{GT-action on KT mod N} yields a well-defined $GT(\mathbb K)$-action on
$\widehat{{\mathbb K}[{\mathcal T}_{\epsilon,\epsilon'}]}/{\mathcal T}_N$
for any $\epsilon$, $\epsilon'$ and any $N\geqslant 0$,
and induces a well-defined $GT(\mathbb K)$-action on 
$\widehat{{\mathbb K}[{\mathcal T}_{\epsilon,\epsilon'}]}$ such that
the  equality 
$$\sigma(\Gamma\cdot\Gamma')=\sigma(\Gamma)\cdot\sigma( \Gamma')$$
holds in $\widehat{{\mathbb K}[{\mathcal T}_{\epsilon_1,\epsilon_3}]}$
for two proalgebraic tangles 
$\Gamma\in \widehat{{\mathbb K}[{\mathcal T}_{\epsilon_1,\epsilon_2}]}$ 
and $\Gamma'\in\widehat{{\mathbb K}[{\mathcal T}_{\epsilon_2,\epsilon_3}]}$
for any $\epsilon_1$, $\epsilon_2$ and $\epsilon_3$.
\end{prop}

\begin{proof}
Though its proof can be found  in loc.~cit~
where it is explained in terms of $\mathbb K$-linear braided monoidal categories,
we can prove in a same way to  the proof of \cite{F12} Theorem 2.38.
by direct calculations.
\end{proof}

We denote the above induced  action by
\begin{equation}\label{GT to Aut KT}
\rho_{\epsilon,\epsilon'}:GT(\mathbb K)\to \mathrm{Aut}  \ \widehat{{\mathbb K}[{\mathcal T}_{\epsilon,\epsilon'}]}.
\end{equation}
We may say that 
it is a generalization  of the map \eqref{GT to Aut B_n}
into the  proalgebraic tangle case.
As a consequence of the above proposition, we have

\begin{prop}\label{GT-action on KSL}
For each sequence  $\epsilon$, the subspace
$\widehat{{\mathbb K}[{\mathcal{SL}}_{\epsilon}]}$ of proalgebraic string links
is stable under the above $GT(\mathbb K)$-action.
The action is compatible with its non-commutative algebra structure
whose product is given by the composition.
\end{prop}

We denote the above action by
\begin{equation}\label{GT to Aut KSL}
\rho_{\epsilon}:GT(\mathbb K)\to \mathrm{Aut}  \ \widehat{{\mathbb K}[{\mathcal SL}_{\epsilon}]}.
\end{equation}
We will see in Theorem \ref{inner automorphism theorem on tangles}
that this action restricted into $GT_1(\mathbb K)$
is given by inner conjugation. 

Particularly by restricting our action \eqref{GT to Aut KT} into $\widehat{{\mathbb K}[{\mathcal T}_{\emptyset,\emptyset}]}$
we obtain the following.

\begin{prop}\label{GT-action on KK}
The subspace $\widehat{{\mathbb K}[{\mathcal K}]}$ 
of proalgebraic knots is stable under the above
$GT(\mathbb K)$-action.
The action there is not compatible with the connected sum $\sharp$ however
the equality
\begin{equation}\label{incompatibility}
\sigma(K_1\sharp K_2)\sharp\sigma(\orientedcircle)=\sigma(K_1)\sharp\sigma(K_2)
\end{equation}
holds for any $\sigma\in GT(\mathbb K)$ and any
$K_1,K_2\in \widehat{{\mathbb K}[{\mathcal K}]}$.
\end{prop}

\begin{proof}
Seeing that the subspace is stable can be verified by showing that connected components of
skeletons of proalgebraic tangles are unchanged.
The equation \eqref{incompatibility}
can be proven in a same way to the proof of the equation (2.21) in \cite{F12}.
\end{proof}

We denote the above action by
\begin{equation}\label{GT to Aut KK}
\rho_{0}:GT(\mathbb K)\to \mathrm{Aut}  \ \widehat{{\mathbb K}[{\mathcal K}_0]}.
\end{equation}

\begin{rem}
We will see in Proposition \ref{Gm-action on knots}
that this action is given by ${\mathbb G}_m$-action. 
We will explicitly determine in Theorem \ref{invariant subspace theorem}
the subspace of $\widehat{{\mathbb K}[{\mathcal K}]}$
which is invariant under the above $GT(\mathbb K)$-action.
\end{rem}

We note that in the proalgebraic knot setting 
our action above is not compatible
with the product structure \eqref{connected sum in proalgebraic knots}
due to \eqref{incompatibility},
while in the profinite knot setting
the $\widehat{GT}$-action on the group of profinite knots is compatible with the product structure
as shown in \cite{F12}  Theorem 2.38.(4).
In order to relate the $\widehat{GT}$-action on $\mathrm{Frac}{\widehat{\mathcal K}}$ constructed in \cite{F12}
with the above $GT(\mathbb K)$-action on $\widehat{{\mathbb K}[{\mathcal K}]}$,
we introduce  the proalgebraic group $\mathrm{Frac}{\mathcal K}(\mathbb K)$ below.

\begin{nota}
Put ${\mathcal K}(\mathbb K)$ to be the group-like part of $\widehat{{\mathbb K}[{\mathcal K}]}$,
which carries a structure of proalgebraic group.
It can be checked  directly that $GT(\mathbb K)$-action on $\widehat{{\mathbb K}[{\mathcal K}]}$
is compatible with its coproduct map and its antipode map.
Hence we have $GT(\mathbb K)$-action on ${\mathcal K}(\mathbb K)$
though it is not compatible with its product structure of ${\mathcal K}(\mathbb K)$
due to  \eqref{incompatibility}.
We denote its group of fraction by $\mathrm{Frac}{\mathcal K}(\mathbb K),$
which is the quotient space of ${\mathcal K}(\mathbb K)\times {\mathcal K}(\mathbb K)$
by the equivalent relations $(r,s)\approx (r',s')$ if
$r\sharp s'\sharp t=r'\sharp s\sharp t$
holds  in ${\mathcal K}(\mathbb K)$ for some $t\in {\mathcal K}(\mathbb K)$.
\end{nota}

\begin{lem}
The induced $GT(\mathbb K)$-action on $\mathrm{Frac}{\mathcal K}(\mathbb K)$
is compatible its group structure, i.e.
$$
\sigma(e)=e, \qquad
\sigma(x\sharp y)=\sigma(x)\sharp\sigma(y), \qquad
\sigma(\frac{1}{x})=\frac{1}{\sigma(x)}
$$
for any $\sigma\in GT({\mathbb K})$ and
$x,y\in \mathrm{Frac}{\mathcal K}(\mathbb K)$.
Here $e={\orientedcircle}/{\orientedcircle}$.
\end{lem}

\begin{proof}
Let $x=r_1/s_1$ and $y=r_2/s_2$. 
Then by  \eqref{incompatibility} it is easy to see
\begin{align*}
\sigma(x\sharp y)&=
\sigma(\frac{r_1\sharp r_2}{s_1\sharp s_2})=
\frac{\sigma({r_1\sharp r_2})}{\sigma({s_1\sharp s_2})} 
=\frac{\sigma(r_1\sharp r_2)\sharp\sigma(\orientedcircle)}
{\sigma(s_1\sharp s_2)\sharp\sigma(\orientedcircle)} 
=\frac{\sigma(r_1)\sharp\sigma(r_2)}{\sigma(s_1)\sharp\sigma(s_2)} \\
&=\frac{\sigma(r_1)}{\sigma(s_1)}\sharp\frac{\sigma(r_2)}{\sigma(s_2)}
=\sigma(\frac{r_1}{s_1})\sharp\sigma(\frac{r_2}{s_2}) 
=\sigma(x)\sharp\sigma(y),\\
\sigma(\frac{1}{x})\sharp\sigma(x)&=
\sigma(\frac{s_1}{r_1})\sharp\sigma(\frac{r_1}{s_1})
=\frac{\sigma(s_1)}{\sigma(r_1)}\sharp \frac{\sigma(r_1)}{\sigma(s_1)}
=\frac{\sigma(s_1)\sharp\sigma(r_1)}{\sigma(r_1)\sharp\sigma(s_1)}
=\frac{\sigma(r_1)\sharp\sigma(s_1)}{\sigma(r_1)\sharp\sigma(s_1)}
=\frac{\orientedcircle}{\orientedcircle}=e.
\end{align*}
\end{proof}

In \cite{F12}, the profinite  $\widehat{GT}$-action on $\mathrm{Frac}{\widehat{\mathcal K}}$ 
was constructed.
A relationship of the action 
with the above constructed $GT({\mathbb K})$-action on ${\mathcal K}({\mathbb K})$ 
is explicitly given below.

\begin{prop}\label{proposition on GK to GKQl}
For each prime $l$,
there is a  natural group homomorphism
\begin{equation}\label{GK to GKQl}
\mathrm{Frac}{\widehat{\mathcal K}}\to \mathrm{Frac}{\mathcal K}({\mathbb Q}_l).
\end{equation}
The $\widehat{GT}$-action on $\mathrm{Frac}{\widehat{\mathcal K}}$ constructed in \cite{F12}
is compatible with
the above $GT({\mathbb Q}_l)$-action on ${\mathcal K}({\mathbb Q}_l)$
under the maps \eqref{GT to GTQl} and \eqref{GK to GKQl}.
\end{prop}

\begin{proof}
The map \eqref{GK to GKQl} is naturally induced  by the map in \cite{F12} Proposition 2.30.
It follows from our construction that the map is  compatible with two actions.
\end{proof}

\subsection{The $GRT$-action}\label{sec: GRT-action on infinitesimal tangles}
In this subsection, we establish an \lq infinitesimal' counterpart of 
the $GT(\mathbb K)$-action on proalgebraic tangles given in the previous subsection.
Infinitesimal tangles (and knots) are introduced in Definition \ref{definition of infinitesimal tangle}
as an infinitesimal counterpart of proalgebraic tangles (and knots).
(It will be shown in next subsection that this notion  coincides with the notion of chord diagrams.)
A consistent $GRT(\mathbb K)$-action there 
is established in Proposition \ref{GRT-action on KIT}.

\begin{defn}\label{definition of fundamental infinitesimal tangle}
(1)
A {\it fundamental infinitesimal tangle}  means a vector belonging to 
one of the following  $\mathbb K$-linear spaces
(let us again call them ABC-spaces):
$A_{k,l}^\epsilon$, $\widehat{IB^{\epsilon}_\tau}$ and $C_{k,l}^\epsilon$
for some $k,l,\epsilon, \tau$:
\begin{align*}
&A_{k,l}^\epsilon :={\mathbb K}\cdot a_{k,l}^\epsilon, \
\text{with} \ 
\epsilon=(\epsilon_i)_{i=1}^{k+l+1}\in
\{\uparrow, \downarrow\}^{k}
\times \{\annihilation, \opannihilation\}
\times\{\uparrow, \downarrow\}^{l} \ 
(k,l=0,1,2,\dots) ,
\\
&\widehat{IB^{\epsilon}_\tau} :=
\widehat{U\frak p_n} 
\cdot\tau \
\text{with} \ 
\epsilon=(\epsilon_i)_{i=1}^{n}
\in \{\uparrow, \downarrow\}^{n} \
\text{and} \ 
\tau\in{\frak S}_n \
(n=1,2,3,4,\dots ),
\\
&C_{k,l}^\epsilon :={\mathbb K}\cdot c_{k,l}^\epsilon, \
\text{with} \
\epsilon=(\epsilon_i)_{i=1}^{k+l+1}\in
\{\uparrow, \downarrow\}^{k}
\times \{\creation,\opcreation\}
\times\{\uparrow, \downarrow\}^{l} \
(k,l=0,1,2,\dots ).
\end{align*}
Here 
$\widehat{U\frak p_n}\cdot \tau$ stands for the coset
of $\widehat{U\frak b_n}\bigm/ \widehat{U\frak p_n}$
corresponding to $\tau\in{\frak S}_n=B_n/P_n$.
To stress that an element $b$ belongs to $\widehat{IB^{\epsilon}_\tau}$,
we occasionally denote $b^\epsilon$ or $(b,\epsilon)$.

For each above space $V$, 
its {\it source} $s(V)$ and {\it target} $t(V)$,
are  defined in a completely same way to Definition \ref{definition of fundamental proalgebraic tangle}.  

(2)
An {\it infinitesimal pre-tangle} $D$ of type $(\epsilon,\epsilon')$
is defined in a same way to Definition \ref{definition of fundamental proalgebraic tangle}.
Namely it is $D=d_n\cdots d_2\cdot d_1$
where each $d_i$ is an infinitesimal fundamental tangle, 
a vector belonging to $V_i$, one of the above ABC-spaces, such that $s(V_{i+1})=t(V_i)$
for all $i=1,2,\dots,n-1$
and $s(V_1)=\epsilon$ and $t(V_n)=\epsilon'$.
We also define $s(D):=s(V_1)$ and $t(D):=t(V_n)$.
An {\it infinitesimal pre-link} $D$ is an infinitesimal pre-tangle with 
$s(D)=t(D)=\emptyset$.
Two infinitesimal pre-tangles 
$D=d_n\cdots d_2\cdot d_1$ and
$D'=d'_m\cdots d'_2\cdot d'_1$
are called {\it composable} when $s(D)=t(D')$
and their {\it composition} $D\cdot D'$ is defined by
$d_n\cdots d_2\cdot d_1\cdot 
 d'_m\cdots d'_2\cdot d'_1$.

For an infinitesimal pre-tangle  $D$
its {\it skeleton} ${\mathbb S}( D)$ 
and its {\it connected components}  can be defined in the same way to
Definition \ref{definition of fundamental proalgebraic tangle}.
An {\it infinitesimal pre-knot} is an infinitesimal pre-link with a single connected component.
An {\it infinitesimal pre-string link} of type $\epsilon=(\epsilon_i)_{i=1}^n$ is 
an infinitesimal pre-tangle of type $(\epsilon,\epsilon)$
consisting of $n$ connected components whose each $i$-th component
connect  $i$-th point on the bottom to the one on the top.

(3)
Two infinitesimal pre-tangles are called {\it isotopic} when they are related by a finite number of
the 6 moves (IT1)-(IT6).
Here (IT1)-(IT5)  are  the moves
replacing profinite tangles and profinite braids group $\widehat{B_n}$
by infinitesimal pre-tangles  and  infinitesimal braid algebras 
$\widehat{U\frak b_n}$ in (T1)-(T5) in \cite{F12} and 
(IT6) is an  \lq infinitesimal' variant of (T6), which is stated below.

(IT6): For $\alpha\in \mathbb K$, 
$c_{k,l}^{\epsilon}\in C_{k,l}^{\epsilon}$ 
and $t_{k+1,k+2}\cdot \tau\in \widehat{IB^{\epsilon'}_\tau}$
with $\tau=\tau_{k+1,k+2}\in \frak S_{k+l+2}$ (the switch of $k+1$ and $k+2$)
and  $t(C_{k,l}^{\epsilon})={\epsilon'}$
$$
\exp\{\alpha \ t_{k+1,k+2}\}\cdot\tau_{k+1,k+2}\cdot c_{k,l}^{\epsilon}=c_{k,l}^{\bar\epsilon}
$$
where $\bar\epsilon$ is the sequence obtained by
revering the $(k+1)$-st and $(k+2)$-nd arrows. 
And for $\alpha\in{\mathbb K}$,
$a_{k,l}^{\epsilon}\in A_{k,l}^{\epsilon}$ and
$\tau \cdot t_{k+1,k+2}\in \widehat{IB^{\epsilon'}_\tau}$
with $\tau=\tau_{k+1,k+2}\in \frak S_{k+l+2}$
and $s(A_{k,l}^{\epsilon})=\tau(\epsilon')$
$$
a_{k,l}^{\epsilon}\cdot \tau_{k+1,k+2} \cdot \exp\{\alpha \ t_{k+1,k+2}\}
=a_{k,l}^{\bar\epsilon}.
$$
where $\bar\epsilon$ is the sequence
obtained by
revering the $(k+1)$-st and $(k+2)$-nd arrows. 
Figure \ref{IT6} depicts the moves.
(N.B. We note that $\exp\{\alpha \ t_{k+1,k+2}\}$ for $\alpha\in{\mathbb K}$ makes sense in
$\widehat{U\frak p}_{k+l+2}$.)
\begin{figure}[h]
\begin{tabular}{cc}
\begin{minipage}{0.5\hsize}
\begin{center}
          \begin{tikzpicture}
                    \draw[-] (0,0.3) --(0,2.);
                    \draw[loosely dotted] (0.1,0.45) --(0.3, 0.45) (0.1,0.8)--(0.3,0.8) (0.1,1.5)--(0.3,1.5);
                    \draw[-] (0.4,0.3) --(0.4,2.);
                     \draw[decorate,decoration={brace,mirror}] (-0.1,0.3) -- (0.5,0.3) node[midway,below]{$k$};
\draw (1.1,0.6)  arc (180:360:0.2);
                   \draw[-] (1.1,1.0) --(1.5, .6);
                   \draw[-] (1.1,.6) --(1.5,1.0);
\draw  (0.5,1.2) rectangle (2.1,1.8);
                   \draw[-] (1.1,1.0)--(1.1,1.2)  (1.1,1.8)--(1.1,2.);
                   \draw[-]  (1.5,1.0)--(1.5,1.2) (1.5,1.8)--(1.5,2.);
\draw  (1.0, 1.45) node{{$\exp\{\alpha$}};
                    \draw[-]  (1.5,1.3)--(1.5,1.6) ;
                   \draw[densely dotted] (1.5,1.45)--(1.8,1.45);
                    \draw[-]  (1.8,1.3)--(1.8,1.6) ;
\draw  (2.0, 1.45) node{{$\}$}};
                    \draw[-] (2.2,0.3) --(2.2,2.);
                   \draw[loosely dotted] (2.3,0.45) --(2.5,0.45) (2.3,0.8) --(2.5,0.8) (2.3,1.5)--(2.5,1.5) ;
                    \draw[-] (2.6,0.3) --(2.6,2.);
                    \draw[decorate,decoration={brace,mirror}] (2.1,0.3) -- (2.7,0.3) node[midway,below]{$l$};
\draw[color=white, very thick] (-0.1,0.6)--(2.7,0.6) (-0.1,1.0)--(2.7,1.0)  ;
\draw  (3.0,1.0)node{$=$};
                    \draw[-] (3.4,0.3) --(3.4,2.0);
                    \draw[loosely dotted] (3.5,1.1) --(3.7, 1.1) ;
                    \draw[-] (3.8,0.3) --(3.8,2.0);
                     \draw[decorate,decoration={brace,mirror}] (3.3,0.3) -- (3.9,0.3) node[midway,below]{$k$};
\draw (3.9,2.0)  arc (180:360:0.2);
                    \draw[-] (4.4,0.3) --(4.4,2.0) ;
                    \draw[-] (4.8,0.3) --(4.8,2.0) ;
                   \draw[loosely dotted] (4.5,1.1) --(4.7, 1.1) ;
                    \draw[decorate,decoration={brace,mirror}] (4.3,0.3) -- (4.9,0.3) node[midway,below]{$l$};
\draw (5.1,0) node{,};
           \end{tikzpicture}
\end{center}
\end{minipage}
\begin{minipage}{0.5\hsize}
\begin{center}
          \begin{tikzpicture}
                    \draw[-] (0,0) --(0,1.7);
                    \draw[loosely dotted] (0.1,0.5) --(0.3, 0.5) (0.1,1.2)--(0.3,1.2)(0.1,1.55)--(0.3,1.55);
                    \draw[-] (0.4,0) --(0.4,1.7);
                     \draw[decorate,decoration={brace,mirror}] (-0.1,0) -- (0.5,0) node[midway,below]{$k$};
                    \draw[-] (1.1,0)--(1.1,.2) (1.1,0.8)--(1.1,1.0) ;
                    \draw[-] (1.5,0)--(1.5,.2) (1.5,0.8)--(1.5,1.0);
                    \draw[-] (1.1,1.0) --(1.5, 1.4);
                    \draw[-] (1.1,1.4) --(1.5, 1.0);
           \draw (1.1,1.4)  arc (180:0:0.2);

\draw  (0.5,0.2) rectangle (2.1,0.8) ;
\draw  (1.0, 0.45) node{{$\exp\{\alpha$}};
                    \draw[-]  (1.5,.3)--(1.5,.6) ;
                   \draw[densely dotted] (1.5,.45)--(1.8,.45);
                    \draw[-]  (1.8,.3)--(1.8,.6) ;
\draw  (2.0, .45) node{{$\}$}};

                    \draw[-] (2.2,0) --(2.2,1.7) ;
                    \draw[-] (2.6,0) --(2.6,1.7) ;
                   \draw[loosely dotted] (2.3,0.5) --(2.5, 0.5) (2.3,1.2) --(2.5,1.2)(2.3,1.55) --(2.5,1.55)  ;
                    \draw[decorate,decoration={brace,mirror}] (2.1,0) -- (2.7,0) node[midway,below]{$l$};
\draw[color=white, very thick]  (-0.1,1.0)--(2.7,1.0) (-0.1,1.4)--(2.7,1.4) ;

\draw  (3.0,1.0)node{$=$};
                    \draw[-] (3.4,0) --(3.4,1.7);
                    \draw[loosely dotted] (3.5,1) --(3.7, 1) ;
                    \draw[-] (3.8,0) --(3.8,1.7);
                     \draw[decorate,decoration={brace,mirror}] (3.3,0) -- (3.9,0) node[midway,below]{$k$};
\draw (3.9,0.1)  arc (180:0:0.2);
                    \draw[-] (4.4,0) --(4.4,1.7) ;
                    \draw[-] (4.8,0) --(4.8,1.7) ;
                   \draw[loosely dotted] (4.5,1) --(4.7, 1) ;
                    \draw[decorate,decoration={brace,mirror}] (4.3,0) -- (4.9,0) node[midway,below]{$l$};
           \end{tikzpicture}
\end{center}
\end{minipage}
\end{tabular}
\caption{(IT6)}
\label{IT6}
\end{figure}

(4)
We denote ${\mathbb K}[\mathcal{IT}_{\epsilon,\epsilon'}^{\rm pre}]$ 
to be the $\mathbb K$-linear space which is the quotient of 
the $\mathbb K$-span of infinitesimal pre-tangles with type $(\epsilon,\epsilon')$
divided by the equivalence linearly generated by the above
isotopy.
Note that there is a natural composition map
\begin{equation}\label{composition of ITpre}
\cdot:
{\mathbb K}[{\mathcal{IT}}_{\epsilon_1,\epsilon_2}^{\rm pre}]\times
{\mathbb K}[{\mathcal{IT}}_{\epsilon_2,\epsilon_3}^{\rm pre}]\to
{\mathbb K}[{\mathcal{IT}}_{\epsilon_1,\epsilon_3}^{\rm pre}]
\end{equation}
for  any $\epsilon_1,\epsilon_2$ and $\epsilon_3$.
The subspace  ${\mathbb K}[{\mathcal{ISL}}_{\epsilon}^{\rm pre}]$
of ${\mathbb K}[{\mathcal{IT}}_{\epsilon,\epsilon}^{\rm pre}]$
consisting of infinitesimal pre-string links 
forms a non-commutative $\mathbb K$-algebras
by the composition \eqref{composition of ITpre}. 

The symbol $\mathbb K[\mathcal{IK}^{\rm pre}]$ stands for the  subspace of
$\mathbb K[\mathcal{IT}_{\emptyset,\emptyset}^{\rm pre}]$ generated by infinitesimal pre-knots.
As in Definition \ref{definition of fundamental proalgebraic tangle},
it 
inherits a structure of  a commutative $\mathbb K$-algebra by the connected sum
(which can be defined in the same way to \cite{F12}) 
\begin{equation}\label{connected sum on KIK-pre}
\sharp:\mathbb K[\mathcal{IK}^{\rm pre}]\times \mathbb K[\mathcal{IK}^{\rm pre}]\to
\mathbb K[\mathcal{IK}^{\rm pre}].
\end{equation}
Again we note that $\sharp$ is different from the above composition \eqref{composition of ITpre}.

(5)
For a  $\mathbb K$-linear space $V$ with 
$V=A^\epsilon_{k,l}$ or $C^\epsilon_{k,l}$
we give its descending filtration of $\mathbb K$-linear subspaces
by ${\mathcal{IT}}_0(V)=V$ and $\mathcal{IT}_N(V):=\{0\}$ for $N>0$, 
and when $V=\widehat{IB^\epsilon_\tau}$,
we give its descending filtration $\{\mathcal{IT}_N(V)\}_{N=0}^\infty$ such that
$\mathcal{IT}_N(V)$ is the $\mathbb K$-linear subspace topologically generated by elements whose  degrees 
are greater than or equal to $N$.

By the method indicated in Definition \ref{definition of fundamental proalgebraic tangle},
for any sequences $\epsilon$ and $\epsilon'$,
$\mathbb K[\mathcal{IT}_{\epsilon,\epsilon'}^{\rm pre}]$
is inherited a filtration of its $\mathbb K$-submodules 
which we denote by $\{\mathcal{IT}^{\rm pre}_N\}_{N\geqslant 0}$.
They are compatible  with the composition map \eqref{composition of ITpre}.
The filtration on $\mathbb K[\mathcal{IK}^{\rm pre}]$ induced by the filtration 
of  $\mathbb K[\mathcal{IT}_{\emptyset,\emptyset}^{\rm pre}]$
is denoted by $\{\mathcal{IK}^{\rm pre}_N\}_{N\geqslant 0}$.
The filtration is compatible with its algebra structure given by
\eqref{connected sum on KIK-pre}.
\end{defn}

\begin{defn}\label{definition of infinitesimal tangle}
(1)
An {\it infinitesimal tangle} with type $(\epsilon,\epsilon')$ is an element of
$$
\widehat{{\mathbb K}[{\mathcal{IT}}_{\epsilon,\epsilon'}]}:=
\varprojlim_n \ \mathbb K[\mathcal{IT}_{\epsilon,\epsilon'}^{\rm pre}]/{\mathcal{IT}}^{\rm pre}_n.
$$

(2)
An {\it infinitesimal string link} with type $\epsilon$ is an element of
$$
\widehat{{\mathbb K}[{\mathcal{ISL}}_{\epsilon}]}:=
\varprojlim_n \ \mathbb K[\mathcal{ISL}_{\epsilon}^{\rm pre}]/{\mathcal{IT}}^{\rm pre}_n.
$$

(3)
An  {\it infinitesimal knot} is an element of the following completed $\mathbb K$-algebra
$$
\widehat{{\mathbb K}[{\mathcal{IK}}]}:=
\varprojlim_n \ \mathbb K[\mathcal{IK}^{\rm pre}]/{\mathcal{IK}}^{\rm pre}_n.
$$
\end{defn}

We note that 
$\widehat{{\mathbb K}[{\mathcal{IT}}_{\epsilon,\epsilon'}]}$ for $\epsilon,\epsilon'$
is inherited a composition product $\cdot$ by \eqref{composition of ITpre}
and  the set
$\widehat{{\mathbb K}[{\mathcal{ISL}}_{\epsilon}]}$ generally forms a non-commutative $\mathbb K$-algebra.
The set $\widehat{{\mathbb K}[{\mathcal{IK}}]}$ is
inherited a connected  $\sharp$ by \eqref{connected sum on KIK-pre}
and  
forms a commutative $\mathbb K$-algebra.

\begin{rem}
By \cite{Bar96} \S 4.2, we have a natural inclusion
$
\widehat{U\frak p_n}
\hookrightarrow\widehat{{\mathbb K}[{\mathcal{ISL}}_{\uparrow^n}]}.
$
Similarly for each $\epsilon=(\epsilon_i)_{i=1}^n$ with $\epsilon=\uparrow$, 
$\downarrow$,  there is an inclusion
$\widehat{U\frak p_n}
\hookrightarrow\widehat{{\mathbb K}[{\mathcal{ISL}}_{\epsilon}]}.$
We denote the image of each element $h\in \widehat{U\frak p_n}$ by $h^\epsilon$.
\end{rem}

As an analogue of Lemma \ref{identification between proalgebraic long knots and proalgebraic knots}, we have

\begin{lem}\label{identification between infinitesimal long knots and proalgebraic knots}
Let $\epsilon=\uparrow$ or $\downarrow$. 
There is an identification of two algebras
\begin{equation*}
\mathrm{cl}:
(\widehat{{\mathbb K}[{\mathcal{ISL}}_{\epsilon}]},\cdot)\simeq 
(\widehat{{\mathbb K}[{\mathcal{IK}}]},\sharp)
\end{equation*}
which  is compatible with their filtrations.
\end{lem}

Based on the identification,
the action of the graded Grothendieck-Teichm\"{u}ller group $GRT({\mathbb K})$ on
infinitesimal tangles is constructed as follows:

\begin{defn}\label{GRT-action on infinitesimal tangles}
Let ${\bar D}\in \widehat{{\mathbb K}[{\mathcal{IT}}_{\epsilon,\epsilon'}]}/{\mathcal{IT}}_N$
with any sequences $\epsilon,\epsilon'$ and $N\geqslant 0$.
Let $D$ be its representative in $\mathbb K[\mathcal{IT}_{\epsilon,\epsilon'}^{\rm pre}]$
with a presentation $D= d_{m}\cdots d_{2}\cdot d_{1}$
($ d_{j}$: fundamental  infinitesimal tangle).
For $\sigma=(c,g)\in {GRT}({\mathbb K})$,
hence $c\in{\mathbb K}^\times$ and $g\in\exp\frak{f}_2$,
we define 
\begin{equation}\label{GRT-action on KIT mod N}
\sigma(\bar D):=
\overline{\sigma( d_{m})\cdots\sigma( d_{2})\cdot\sigma( d_{1})}
\in \widehat{{\mathbb K}[{\mathcal{IT}}_{\epsilon,\epsilon'}]}/{\mathcal {IT}}_N.   
\end{equation}
with $\sigma( d_j)$ given  below. 

(1)
when $ d_{j}\in A_{k,l}^\epsilon$, we define
$$\sigma( d_{j}):= d_{j}\cdot
(\nu_g)^{s( d_{j})}_{k+2}
\cdot g_{1\cdots k,k+1,k+2}^{s( d_{j})}.$$
Here the middle term stands for the infinitesimal tangle  whose source is $s( d_{j})$
which is obtained by  putting the trivial infinitesimal braid with $k+1$-strands on the left of $\nu_g$ (see below) and
the trivial infinitesimal braid with $l$-strands on its right. 

(2)
when $ d_{j}=(b_n,\epsilon)\in \widehat{IB_\tau^\epsilon}$ with 
$b_n\in  \widehat{U\frak p_n} \cdot\tau\subset \widehat{U\frak b_n} $, 
we define
$$\sigma( d_{j}):=(\rho_n(\sigma)(b_n),\epsilon).$$
Here $\rho_n(\sigma)(b_n)$ is the image of $b_n$
by the action $\sigma\in GRT(\mathbb K)$ on $\widehat{U\frak b_n}$
explained in \S \ref{GRT-action on infinitesimal braids}.

(3)
when $ d_{j}\in C_{k,l}^\epsilon$, we define
$$\sigma( d_{j}):=g_{1\cdots k,k+1,k+2}^{-1,t( d_{j})}
\cdot d_{j}.$$

The symbol $\nu_g^\epsilon$ ($\epsilon=\uparrow, \downarrow$) means
the infinitesimal tangle in $\widehat{{\mathbb K}[{\mathcal{IT}}_{\epsilon,\epsilon}]}$
with a single connected component such that $s(\mu_g^\epsilon)=t(\mu_g^\epsilon)=\epsilon$
and which is given by the inverse of  $\Lambda_g^\epsilon$
with respect to the composition:
$$
\nu_g^\epsilon:=\{\Lambda_g^\epsilon\}^{-1}.
$$
Here $\Lambda_g^\downarrow$ in $\widehat{{\mathbb K}[{\mathcal{IT}}_{\epsilon,\epsilon}]}$
represents the infinitesimal string link with a single strands given by
$$
\Lambda_g^\downarrow:=
a_{1,0}^{\downarrow\annihilation}
\cdot g^{\downarrow\uparrow\downarrow}\cdot 
c_{0,1}^{\creation\downarrow}.
$$
(which can be depicted as the  picture in Figure \ref{proalgebraic error term}
replacing  $f$ by $g$)
and $\Lambda_g^\uparrow$ is the same one obtained by reversing its all arrows.
\end{defn}

We note that the existence of the inverse of $\Lambda^\epsilon_g$ 
in $\widehat{{\mathbb K}[{\mathcal{IT}}_{\epsilon,\epsilon}]}$ is immediate
because $\Lambda_g^\epsilon$ is congruent to the unit in 
 $\widehat{{\mathbb K}[{\mathcal{IT}}_{\epsilon,\epsilon}]}$ modulo
${{\mathcal T}_1}$
and $\widehat{{\mathbb K}[{\mathcal{IT}}_{\epsilon,\epsilon}]}$ is completed by the filtration 
$\{\mathcal T_n\}_{n=0}^\infty$.

\begin{prop}\label{GRT-action on KIT}
(1).
The equation \eqref{GRT-action on KIT mod N} yields a well-defined $GRT(\mathbb K)$-action on
$\widehat{{\mathbb K}[{\mathcal{IT}}_{\epsilon,\epsilon'}]}/{\mathcal{IT}}_N$
for any $\epsilon$, $\epsilon'$ and any $N\geqslant 0$,
and induces a well-defined $GRT(\mathbb K)$-action on 
$\widehat{{\mathbb K}[{\mathcal{IT}}_{\epsilon,\epsilon'}]}$ such that
the  equality 
$$\sigma( D\cdot D')=\sigma( D)\cdot\sigma( D')$$
holds in $\widehat{{\mathbb K}[{\mathcal{IT}}_{\epsilon_1,\epsilon_3}]}$
for two infinitesimal tangles 
$D\in \widehat{{\mathbb K}[{\mathcal{IT}}_{\epsilon_1,\epsilon_2}]}$ 
and $D'\in\widehat{{\mathbb K}[{\mathcal{IT}}_{\epsilon_2,\epsilon_3}]}$
for any $\epsilon_1$, $\epsilon_2$ and $\epsilon_3$.

(2)
For each $\epsilon$,
the subspace $\widehat{{\mathbb K}[{\mathcal{ISL}_{\epsilon}}]}$ 
of infinitesimal string links with type $\epsilon$
is stable under the above
$GRT(\mathbb K)$-action.
The action is compatible with its algebra structure whose product is given by the composition.

(3)
The subspace $\widehat{{\mathbb K}[{\mathcal{IK}}]}$ 
of infinitesimal knots
is stable under the above
$GRT(\mathbb K)$-action.
The action there is not compatible with the connected sum $\sharp$ however
the  equality
\begin{equation}\label{incompatibility in KIK}
\sigma(D_1\sharp D_2)\sharp\sigma(\orientedcircle)=\sigma(D_1)\sharp\sigma(D_2)
\end{equation}
holds for any $\sigma\in GRT(\mathbb K)$ and any
$D_1,D_2\in \widehat{{\mathbb K}[{\mathcal{IK}}]}$.
\end{prop}

\begin{proof}
Proof can be done in a completely same way to the proof of
Proposition \ref{GT-action on KT} and \ref{GT-action on KK}.
%
\end{proof}

We denote the above induced  actions respectively by
\begin{align}
\label{GRT to Aut KIT}
\rho_{\epsilon,\epsilon'}&:GRT(\mathbb K)\to \mathrm{Aut} \ \widehat{{\mathbb K}[{\mathcal {IT}}_{\epsilon,\epsilon'}]},\\
\label{GRT to Aut KISL}
\rho_{\epsilon}&:GRT(\mathbb K)\to \mathrm{Aut} \ \widehat{{\mathbb K}[{\mathcal {ISL}}_{\epsilon}]}, \\
\label{GRT to Aut KIK}
\rho_{0}&:GRT(\mathbb K)\to \mathrm{Aut} \ \widehat{{\mathbb K}[{\mathcal {IK}}]}.
\end{align}
We may say that the map \eqref{GRT to Aut KIT}
is a generalization  of the map \eqref{GRT to Aut Ub_n}
into the  infinitesimal tangle case.
We will see
that the action \eqref{GRT to Aut KISL} restricted into $GRT_1(\mathbb K)$
is given by inner conjugation in  Theorem \ref{inner automorphism theorem on chord diagrams}
and
that the action \eqref{GRT to Aut KIK} is given by ${\mathbb G}_m$-action
in Proposition \ref{Gm-action on knots}. 

\subsection{Associators}
In this subsection we reformulate
the isomorphism given in  \cite{KT98} Theorem 2.4
(also shown  in \cite{Bar, C, KT98, LM, P})
in our terminologies of proalgebraic  and infinitesimal tangles.
It is shown that each associator gives an  isomorphism between 
the system 
of proalgebraic tangles
and the system  
of infinitesimal tangles 
in Proposition \ref{KT to KIT}.
Proposition \ref{IT=CD} shows an equivalence of the notion of infinitesimal tangles
and the notion of chord diagrams.

Similarly to our previous subsections,
such an isomorphism is constructed piecewise.

\begin{defn}\label{definition: KT to KIT mod N}
Let ${\bar \Gamma}\in \widehat{{\mathbb K}[{\mathcal{T}}_{\epsilon,\epsilon'}]}/{\mathcal T}_N$
with any sequences $\epsilon,\epsilon'$ and $N\geqslant 0$.
Let $\Gamma$ be its representative in $\mathbb K[\mathcal{T}_{\epsilon,\epsilon'}^{\rm pre}]$
with a presentation $\Gamma=\gamma_{m}\cdots\gamma_{2}\cdot\gamma_{1}$
($\gamma_{j}$: fundamental proalgebraic tangle).
For $p=(\mu,\varphi)\in {M}({\mathbb K})$,
hence $\mu\in{\mathbb K}^\times$ and $\varphi\in\exp\frak{f}_2$,
we define 
\begin{equation}\label{KT to KIT mod N}
p(\bar \Gamma):=
\overline{p(\gamma_{m})\cdots p(\gamma_{2})\cdot p(\gamma_{1})}
\in \widehat{{\mathbb K}[{\mathcal{IT}}_{\epsilon,\epsilon'}]}/{\mathcal{IT}}_N.   
\end{equation}
with $ p(\gamma_j)$ given  below. 

(1)
when $\gamma_{j}\in A_{k,l}^\epsilon$, we define
$$ p(\gamma_{j}):=\gamma_{j}\cdot
(\nu_\varphi)^{s(\gamma_{j})}_{k+2}
\cdot \varphi_{1\cdots k,k+1,k+2}^{s(\gamma_{j})}.$$
Here $\nu_\varphi$ is the infinitesimal string link defined in
Definition \ref{GRT-action on infinitesimal tangles}.

(2)
when $\gamma_{j}=(b_n,\epsilon)\in \widehat{B_\tau^\epsilon}$ with 
$b_n\in  \widehat{{\mathbb K}[P_n]}\cdot\tau\subset \widehat{{\mathbb K}[B_n]}$, 
we define
$$ p(\gamma_{j}):=(\rho_n(p)(b_n),\epsilon).$$
Here $\rho_n(p)(b_n)\in \widehat{U\frak b_n}$ is the image of $b_n$
by the map given in Proposition \ref{M to Isom}.

(3)
when $\gamma_{j}\in C_{k,l}^\epsilon$, we define
$$ p(\gamma_{j}):=\varphi_{1\cdots k,k+1,k+2}^{-1,t(\gamma_{j})}
\cdot\gamma_{j}.$$
\end{defn}

As an analogue of Proposition \ref{GT-action on KT} and \ref{GRT-action on KIT}
in this subsection, we have the following: 

\begin{prop}\label{KT to KIT}
(1)
For each $p=(\mu,\varphi)\in {M}({\mathbb K})$,
the equation \eqref{KT to KIT mod N} yields a well-defined isomorphism
of $\mathbb K$-linear spaces
$$
\rho^N_{\epsilon,\epsilon'}(p):
\widehat{{\mathbb K}[{\mathcal{T}}_{\epsilon,\epsilon'}]}/{\mathcal T}_N
\simeq
\widehat{{\mathbb K}[{\mathcal{IT}}_{\epsilon,\epsilon'}]}/{\mathcal{IT}}_N
$$
for any $\epsilon$ and $\epsilon'$ and $N\geqslant 0$.

(2)
This induces an isomorphism 
of $\mathbb K$-linear spaces
\begin{equation}\label{isomorphism between KT and KIT}
\rho_{\epsilon,\epsilon'}(p):
\widehat{{\mathbb K}[{\mathcal{T}}_{\epsilon,\epsilon'}]}
\simeq
\widehat{{\mathbb K}[{\mathcal{IT}}_{\epsilon,\epsilon'}]}
\end{equation}
such that  the equality
$$
\rho_{\epsilon_1,\epsilon_3}(p)(\Gamma\cdot\Gamma')=\rho_{\epsilon_1,\epsilon_2}(p) (\Gamma)\cdot \rho_{\epsilon_2,\epsilon_3}(p)(\Gamma')
$$
holds in $\widehat{{\mathbb K}[{\mathcal{IT}}_{\epsilon_1,\epsilon_3}]}$
for two proalgebraic tangles 
$\Gamma\in \widehat{{\mathbb K}[{\mathcal T}_{\epsilon_1,\epsilon_2}]}$ 
and $\Gamma'\in\widehat{{\mathbb K}[{\mathcal T}_{\epsilon_2,\epsilon_3}]}$
for any $\epsilon_1$, $\epsilon_2$ and $\epsilon_3$.

(3)
Restriction of the map \eqref{isomorphism between KT and KIT} 
into proalgebraic string links of type $\epsilon$
yields an isomorphism
\begin{equation}\label{isomorphism between KSL and KISL}
\rho_\epsilon(p):\widehat{{\mathbb K}[{\mathcal{SL}}_{\epsilon}]}
\simeq
\widehat{{\mathbb K}[{\mathcal{ISL}}_{\epsilon}]}.
\end{equation}
It is compatible with non-commutative product structures of both algebra given by the compositions.

(4)
Restriction of the map \eqref{isomorphism between KT and KIT} 
into proalgebraic knots
yields an isomorphism
\begin{equation}\label{isomorphism between KK and KIK}
\rho_0(p):\widehat{{\mathbb K}[{\mathcal{K}}]}
\simeq
\widehat{{\mathbb K}[{\mathcal{IK}}]}.
\end{equation}
It is not compatible with the commutative product structure given by 
the connected sum $\sharp$
however the equality
\begin{equation}\label{incompatibility in KK and KIK}
\rho_0(p)(K_1\sharp K_2) \ \sharp  \ \rho_0(p)(\orientedcircle)
=\rho_0(p)(K_1) \ \sharp  \ \rho_0(p)(K_2)
\end{equation}
holds in $\widehat{{\mathbb K}[{\mathcal{IK}}]}$
for any $K_1, K_2\in \widehat{{\mathbb K}[{\mathcal{K}}]}$.
\end{prop}

\begin{proof}
(1)
Firstly we have to show that Definition \ref{definition: KT to KIT mod N} makes sense,
that is, the map $\rho^N_{\epsilon,\epsilon'}(p)$ is well-defined.
It is enough to  prove the equality 
$\rho^N_{\epsilon,\epsilon'}(p)(\bar{\Gamma_1})
=\rho^N_{\epsilon,\epsilon'}(p)(\bar{\Gamma_2})$
for
$\Gamma_1$ and $\Gamma_2 \in \widehat{{\mathbb K}[{\mathcal{T}}_{\epsilon,\epsilon'}]}$
when $\Gamma_1$ is obtained from $\Gamma_2$  by a single operation of one of the moves
(T1)-(T6).
This can be proved in a completely same way to the proof of \cite{F12} Theorem 2.38.(1).

Secondly we prove that the map $\rho^N_{\epsilon,\epsilon'}(p)$ is isomorphic.
This is achieved by considering the  $\mathbb K$-linear map
$$
S_a:{\mathcal{IT}}_a\bigm/{\mathcal{IT}}_{a+1}\to {\mathcal{T}}_a\bigm/{\mathcal{IT}}_{a+1}
$$
for $a=0,1,2,\dots, N-1$.
It is a map sending each $\Gamma=\gamma_m\cdots\gamma_1$ 
with each $\gamma_i$ belonging to  one $V_i$ of the ABC-spaces
to $S_a(\Gamma):= S_a(\gamma_m)\cdots S_a(\gamma_1)$
with $S_a(\gamma_i)=\gamma_i$ when $V_i=A_{k,l}^\epsilon$ or $C_{k,l}^\epsilon$
for some $k,l$, $\epsilon$,
and $S_a(\gamma_i)=\rho_n(p)^{-1}(\gamma_i)$
when $V_i=\widehat{IB^\tau_\epsilon}$
for some  $\tau\in\frak S_n$ (some $n\geqslant 1$) and $\epsilon$.
Here $\rho_n(p):
\widehat{{\mathbb K}[B_n]}
\to  
\widehat{U\frak b_n}
$
is the isomorphism given by \eqref{M to Isom KB UB}.
The well-definedness of $S_a$ can be checked directly. 
To see the compatibility for the move (IT6),
we need to use the congruence $\rho_n(p)^{-1}(t_{k+1,k+2})\equiv \sigma_{k+1}-\sigma_{k+1}^{-1} \pmod{I^1}$.
By the construction, $S_a$ is isomorphic.
By checking that $S_a$ gives an right inverse of  
$$
\rho^{a+1}_{\epsilon,\epsilon'}(p)\bigm|_{{\mathcal{T}}_a}:
{\mathcal{T}}_a\bigm/{\mathcal{IT}}_{a+1}
\to {\mathcal{IT}}_a\bigm/{\mathcal{IT}}_{a+1},
$$
inductively we get  that  $\rho^{a+1}_{\epsilon,\epsilon'}(p)$
is isomorphic.

(2)
Since both filtration $\{\mathcal{T}_n\}_{n=0}^\infty$ and $\{\mathcal{IT}_n\}_{n=0}^\infty$
are compatible with the isomorphism $\rho^{N}_{\epsilon,\epsilon'}(p)$,
the isomorphism \eqref{isomorphism between KT and KIT} is obtained.
Checking the equality of the composition is immediate to see.

(3)
It can be proved by the  same arguments of the proof of
Proposition \ref{GT-action on KSL}.

(4)
The statements follow from the same arguments given in the proof of
Proposition \ref{GT-action on KK}.
\end{proof}

The above isomorphism \eqref{isomorphism between KT and KIT} is compatible with 
both $GT(\mathbb K)$-action on $\widehat{{\mathbb K}[{\mathcal{T}}_{\epsilon,\epsilon'}]}$
and $GRT(\mathbb K)$-action on $\widehat{{\mathbb K}[{\mathcal{IT}}_{\epsilon,\epsilon'}]}$.

\begin{prop}\label{three torsor maps}
The induced maps 
\begin{align}
\label{M to Isom KT KIT}
\rho_{\epsilon,\epsilon'}&:M(\mathbb K)\to \mathrm{Isom}\left(\widehat{{\mathbb K}[{\mathcal{T}}_{\epsilon,\epsilon'}]},
\ \widehat{{\mathbb K}[{\mathcal{IT}}_{\epsilon,\epsilon'}]}\right),\\
\label{M to Isom KSL KISL}
\rho_{\epsilon}&:M(\mathbb K)\to \mathrm{Isom}\left(\widehat{{\mathbb K}[{\mathcal{SL}}_{\epsilon}]},
\ \widehat{{\mathbb K}[{\mathcal{ISL}}_{\epsilon}]}\right), \\
\label{M to Isom KK KIK}
\rho_{0}&:M(\mathbb K)\to \mathrm{Isom}\left(\widehat{{\mathbb K}[{\mathcal{K}}]},
\ \widehat{{\mathbb K}[{\mathcal{IK}}]}\right) 
\end{align}
are all  morphisms of bitorsors. 
\end{prop}

\begin{proof}
It is derived from Proposition \ref{M to Isom}.
\end{proof}

We may say that the map \eqref{M to Isom KT KIT} is a generalization
of the map \eqref{M to Isom} into proalgebraic tangles.

Next we will discuss a relation of infinitesimal tangles with chord diagrams.

\begin{nota}[\cite{KT98,KRT} etc]
Let  $\Gamma$ be a tangle in $\mathcal T_{\epsilon,\epsilon'}$.
A {\it chord diagram} on a curve $\Gamma$ is a finite (possibly empty) set of unordered pair
of points on $\Gamma\setminus \partial \Gamma$.
A homeomorphism of chord diagrams means 
a homeomorphism of the underlying curves preserving their orientations and 
fixing their endpoints
such that it preserves the distinguished pairs of points.
In our picture, we draw a dashed line , called a chord, between the two points of  a distinguished pair. 

We denote $\mathcal{CD}_{\epsilon,\epsilon'}^m$ to be the $\mathbb K$-linear space
generated by all homeomorphism classes of chord diagrams with $m$ chords 
($m\geqslant 0$)
on tangles of type $(\epsilon, \epsilon')$, 
subject to the 4T-relation and the FI-relation.
Here the  {\it  4T-relation} stands for the 4 terms relation
defined by 
$
D_1-D_2+D_3-D_4=0
$
where $D_j$ are chord diagrams with four chords
identical outside a ball in which they differ as illustrated in Figure \ref{4T-relation}
\begin{figure}[h]
\begin{tabular}{cccc}
     \begin{minipage}{0.25\hsize}
         \begin{center}
            \begin{tikzpicture}
                  \draw[->] (0,0)--(0,1.5) ;
                  \draw[->] (0.5,0)--(0.5,1.5) ;
                  \draw[->] (1,0)--(1,1.5) ;
                   \draw[densely dotted] (0,0.5)--(1,0.5);
                   \draw[densely dotted] (0,1)--(0.5,1);
            \end{tikzpicture}
                      \caption*{$D_1$}
         \end{center}
    \end{minipage}
     \begin{minipage}{0.25\hsize}
         \begin{center}
            \begin{tikzpicture}
                  \draw[->] (0,0)--(0,1.5) ;
                  \draw[->] (0.5,0)--(0.5,1.5) ;
                  \draw[->] (1,0)--(1,1.5) ;
                   \draw[densely dotted] (0,0.5)--(0.5,0.5);
                   \draw[densely dotted] (0,1)--(1,1) ;
            \end{tikzpicture}
                      \caption*{$D_2$}
         \end{center}
    \end{minipage}
     \begin{minipage}{0.25\hsize}
         \begin{center}
            \begin{tikzpicture}
                  \draw[->] (0,0)--(0,1.5) ;
                  \draw[->] (0.5,0)--(0.5,1.5) ;
                  \draw[->] (1,0)--(1,1.5) ;
                   \draw[densely dotted] (0.5,0.5)--(1,0.5);
                   \draw[densely dotted] (0,1)--(0.5,1);
            \end{tikzpicture}
                      \caption*{$D_3$}
         \end{center}
    \end{minipage}
     \begin{minipage}{0.25\hsize}
         \begin{center}
            \begin{tikzpicture}
                  \draw[->] (0,0)--(0,1.5) ;
                  \draw[->] (0.5,0)--(0.5,1.5) ;
                  \draw[->] (1,0)--(1,1.5) ;
                   \draw[densely dotted] (0,0.5)--(0.5,0.5);
                   \draw[densely dotted] (0.5,1)--(1,1);
            \end{tikzpicture}
                      \caption*{$D_4$}
         \end{center}
    \end{minipage}
\end{tabular}
    \caption{4T-relation}
    \label{4T-relation}
\end{figure} 
and the  {\it  FI-relation} stands for the frame independent relation
where we put 
$
D=0
$
for any chord diagrams $D$ with an isolated chord, 
a chord that does not intersect on any other one in  their diagrams.
We put $\widehat{\mathcal{CD}}_{\epsilon,\epsilon'}
:=\widehat{\oplus}_{m=0}^\infty\mathcal{CD}_{\epsilon,\epsilon'}^m$. 

It is known that it is well-behaved under the composition
\begin{equation}\label{composition of CD}
\cdot:\widehat{\mathcal{CD}}_{\epsilon_3,\epsilon_2}\times\widehat{\mathcal{CD}}_{\epsilon_2,\epsilon_1}
\to\widehat{\mathcal{CD}}_{\epsilon_3,\epsilon_1}.
\end{equation}
We denote the subspace of $\widehat{\mathcal{CD}}_{\epsilon,\epsilon'}$
consisting of chord diagrams whose underlying spaces are string links with type $\epsilon$
by $\widehat{\mathcal{CD}}({\epsilon})$.
 It forms a non-commutative $\mathbb K$-algebra by the composition map
\eqref{composition of CD}.

The subspace $\widehat{\mathcal{CD}}(\orientedcircle)$
$(\subset\widehat{\mathcal{CD}}_{\emptyset,\emptyset})$
of chord diagrams whose underlying spaces are homeomorphic to the oriented circles
forms a commutative algebra by the connected sum
\begin{equation}
\sharp:\widehat{\mathcal{CD}}(\orientedcircle)\times\widehat{\mathcal{CD}}(\orientedcircle)
\to\widehat{\mathcal{CD}}(\orientedcircle).
\end{equation}
We remind that the unit is given by 
the chordless chord diagram on the oriented circle $\orientedcircle$. 
\end{nota}

\begin{prop}\label{IT=CD}
(1) 
For each $\epsilon,\epsilon'$, there is a natural identifications
$$
\widehat{{\mathbb K}[{\mathcal{IT}}_{\epsilon,\epsilon'}]}
\simeq
\widehat{\mathcal{CD}}_{\epsilon,\epsilon'}
$$
which is compatible with the composition maps.

(2)
For each $\epsilon$, there is a natural identification of non-commutative graded $\mathbb K$-algebras:
$$
\widehat{{\mathbb K}[{\mathcal{ISL}}_{\epsilon}]}
\simeq
\widehat{\mathcal{CD}}({\epsilon}).
$$

(3)
There is a natural identification of commutative graded $\mathbb K$-algebras:
$$
\widehat{{\mathbb K}[{\mathcal{IK}}]}
\simeq
\widehat{\mathcal{CD}}(\orientedcircle).
$$
\end{prop}

\begin{proof}
(1)
By replacing each
$\diaProjected$-part  on the associated picture of each infinitesimal tangle 
$D$ by $\diaCrossP$ 
(actually we may replace it by $\diaCrossN$ because both are equivalent 
modulo homeomorphisms of underlying tangles)
and multiplying $(-1)^{D_\downarrow}$ to each $D$
(which is necessary to keep 4T-relation),
we obtain a well-defined $\mathbb K$-linear map
\begin{equation}\label{isom-KIT-CD}
\widehat{{\mathbb K}[{\mathcal{IT}}_{\epsilon,\epsilon'}]}\to
\widehat{\mathcal{CD}}_{\epsilon,\epsilon'}.
\end{equation}
Here $D_{\downarrow}$ is the the set of ends of chord on $D$
which hit downward lines.
Its composition with the map 
$\rho_{\epsilon,\epsilon'}(p)$ in \eqref{isomorphism between KT and KIT}
is  the isomorphism
$$
\widehat{{\mathbb K}[{\mathcal T}_{\epsilon,\epsilon'}]}\to
\widehat{\mathcal{CD}}_{\epsilon,\epsilon'}.
$$
given as the non-framed version of \cite{KT98} Theorem 2.4.
Since $\rho_{\epsilon,\epsilon'}(p)$ is isomorphic,
the morphism in our claim should be isomorphic.
Checking the compatibility with the composition map is immediate.

(2)
It immediately follows from (1).

(3)
It is obtained by a restriction of the above claim  into the case 
$(\epsilon,\epsilon')=(\emptyset,\emptyset)$. 
It can be checked directly that  the map is compatible with the connected sum.
\end{proof}

Hereafter we identify
$\widehat{{\mathbb K}[{\mathcal IT}_{\epsilon,\epsilon'}]}$ with
$\widehat{\mathcal{CD}}_{\epsilon,\epsilon'}$,
$\widehat{{\mathbb K}[{\mathcal{ISL}}_{\epsilon}]}$ with
$\widehat{\mathcal{CD}}({\epsilon})$, and
$\widehat{{\mathbb K}[{\mathcal IK}]}$ with
$\widehat{\mathcal{CD}}(\orientedcircle)$
by the map \eqref{isom-KIT-CD}.
Thus the identification given in 
Lemma \ref{identification between infinitesimal long knots and proalgebraic knots}
is reformulated as the identification below
\begin{equation}\label{CD(O)=CD(1)}
(\widehat{\mathcal{CD}}({\epsilon}),\cdot)\simeq 
(\widehat{\mathcal{CD}}(\orientedcircle),\sharp)
\end{equation}
for $\epsilon=\uparrow$ or $\downarrow$.

\begin{rem}
Kontsevich's isomorphism (\cite{K})
\begin{equation}\label{Kontsevich's isomorphism}
I:\widehat{{\mathbb C}[{\mathcal K}]}\simeq
\widehat{\mathcal{CD}}(\orientedcircle)
\end{equation}
is given  by specifying $p$ of
$\rho_0(p)$ of \eqref{isomorphism between KK and KIK}
to $p_\mathrm{KZ}=(1,\varphi_\mathrm{KZ})\in M(\mathbb C)$ with
$\varphi_\mathrm{KZ}=\varPhi_\mathrm{KZ}\left(\frac{1}{2\pi\sqrt{-1}}A,\frac{1}{2\pi\sqrt{-1}}B\right).$
We note that it is independent of the choice of $\varphi$
(cf. \cite{LM}).
For each oriented knot $K$, the image $I(K)$ is called the {\it Kontsevich invariant} of $K$.
\end{rem}

\begin{rem}\label{MTM-structure}
By the same arguments to Remark \ref{rem-incl-Gal-M},
the proalgebraic tangles $\widehat{{\mathbb Q}[{\mathcal{T}}_{\epsilon,\epsilon'}]}$,
the proalgebraic string links $\widehat{{\mathbb Q}[{\mathcal{SL}}_{\epsilon}]}$
and 
the proalgebraic knots $\widehat{{\mathbb K}[{\mathcal{K}}]}$
are regarded as Betti realizations of mixed Tate (pro-)motives over
${\rm Spec}\; {\mathbb Z}$.
And their corresponding de Rham realizations are given by
the spaces 
$\widehat{\mathcal{CD}}_{\epsilon,\epsilon'}$,
$\widehat{\mathcal{CD}}({\epsilon})$ and
$\widehat{\mathcal{CD}}(\orientedcircle)$
of chord diagrams located there
respectively.
In Remark \ref{TM-structure}, we will see that
the proalgebraic knots $\widehat{{\mathbb K}[{\mathcal{K}}]}$
carries a structure of Tate (pro-)motives.
\end{rem}


\section{Main results}\label{Main results}
We discuss and derive  distinguished properties of
the action  of the Grothendieck-Teichm\"{u}ller groups
on proalgebraic tangles (constructed in \S\ref{sec: GT-action on proalgebraic tangles})
which can not be observed in  the action on proalgebraic braids
 (discussed in \S\ref{GT-action on proalgebraic braids}).
By exploiting the properties,
we explicitly determine the proalgebraic knot whose Kontsevich invariant
is the unit, the trivial chord diagram (Theorem \ref{inverse image theorem}).

\subsection{Proalgebraic string links}
We restrict the previously constructed action of the Grothendieck-Teichm\"{u}ller group
$GT(\mathbb K)$ on proalgebraic tangles
into  the action of its unipotent part $GT_1(\mathbb K)$
on proalgebraic string links and show that 
it is simply described by an  inner conjugation
(Proposition \ref{inner automorphism theorem on chord diagrams} and
Theorem \ref{inner automorphism theorem on tangles}). 
The proofs are based on Twistor Lemmas (Lemma \ref{twistor lemma} and \ref{twistor lemma on GT}).

\begin{nota}
(1)
For $n>1$, 
$\epsilon_i:\widehat{{\mathbb K}[{\mathcal{ISL}}_{\uparrow^n}]}
\to\widehat{{\mathbb K}[{\mathcal{ISL}}_{\uparrow^{n-1}}]}$
($i=1,2\dots n$) means the map sending an infinitesimal string link $D$ to $0$
if at least  one chord of $D$ has an endpoint on the $i$-th strand;
otherwise  $\epsilon_i(D)$ is obtained by removing the $i$-th strand.

(2)
For $D\in  \widehat{{\mathbb K}[{\mathcal{ISL}}_{\uparrow^n}]}$,
we denote  $D_{1,\dots,n}$ (resp. $D_{2,\dots,n+1}$)  to be the element 
in $\widehat{{\mathbb K}[{\mathcal{ISL}}_{\uparrow^{n+1}}]}$
obtained by putting a chordless straight line  on the right (resp. the left) of $D$
and  $D_{1,\dots, i-1 ,i\ i+1,i+2, \dots, n+1}$ ($i=1,2,\dots,n$)
to be also the element  in $\widehat{{\mathbb K}[{\mathcal{ISL}}_{\uparrow^{n+1}}]}$
obtained by doubling the $i$-th strand and taking the sum over all possible lifts of
the chord endpoints of $D$ from the $i$-th strand
to one of the new two strands.

(3)
For $D\in \widehat{{\mathbb K}[{\mathcal{ISL}}_{\uparrow^n}]}$ and $\tau\in \frak S_n$,
$D_{\tau(1),\dots,\tau(n)}$ denote the element in $\tau^{-1}\cdot D\cdot \tau$
where the product is taken as a product of infinitesimal tangles
(recall the identification given in Proposition \ref{IT=CD}).
 \end{nota}

Hereafter we regard $\exp\frak f_2$ and $U\frak p_5$ to be the subspaces
of $ \widehat{{\mathbb K}[{\mathcal{ISL}}_{\uparrow\uparrow\uparrow}]}$ and
$\widehat{{\mathbb K}[{\mathcal{ISL}}_{\uparrow\uparrow\uparrow\uparrow}]}$ respectively.
The following lemma which is shown for $GRT(\mathbb K)$
might be called as a reformulation of 
\cite{LM} Theorem 8 which is shown for a \lq chord diagrammatic' analogue of
$M_1(\mathbb K)$.

\begin{lem}[Twistor Lemma]\label{twistor lemma}
Let $\sigma=(c,g)\in GRT(\mathbb K)$, thus $c\in {\mathbb K}^\times$ and $g\in \exp \frak f_2\subset 
\widehat{{\mathbb K}[{\mathcal{ISL}}_{\uparrow\uparrow\uparrow}]}$.
Then $g$ is gauge equivalent to $1$,
namely, there exists $ \varDelta(\sigma)\in \widehat{{\mathbb K}[{\mathcal{ISL}}_{\uparrow\uparrow}]}^\times$
satisfying
\begin{equation}
\epsilon_1( \varDelta(\sigma))=\epsilon_2( \varDelta(\sigma))=\uparrow
\end{equation}
and   the symmetric condition
\begin{equation}\label{symmetric condition}
 \varDelta(\sigma)= \varDelta(\sigma)_{2,1}
\end{equation}
such that
\begin{equation}\label{twisting formula}
g= \varDelta(\sigma)_{2,3}\cdot  \varDelta(\sigma)_{1,23}\cdot  \varDelta(\sigma)^{-1}_{12,3}\cdot  \varDelta(\sigma)^{-1}_{1,2}
\end{equation}
holds in $\widehat{{\mathbb K}[{\mathcal{ISL}}_{\uparrow\uparrow\uparrow}]}^\times$. 
\end{lem}

\begin{proof}
The proof can be done recursively in the same way to the proof of \cite{LM} Theorem 8,
or rather, we may say that actually it is easier:
Because we  have $(1,g)$ and $(1,1)\in GRT_1(\mathbb K)$,
both $g$ and $1$ satisfy the same relations \eqref{GRT-2-cycle}--\eqref{GRT-pentagon equation}.
Our proof is obtained just by replacing  $\Phi$ by $g$ and $\Phi'$  by $1$ in their proof.
\end{proof}

\begin{rem}
Other variants of twistor lemma 
can be found in several literatures such as
\cite{LS} Theorem 2.1 with twistors in $\mathrm{Aut} \widehat{F_2}$ for $\widehat{GT}_1$,
\cite{AT12} Theorem 7.5 with twistors in $\mathrm{TAut} \frak f_2$ for $KRV^0_3$, and
\cite{AET} Theorem 2 with twistors in $\mathrm{TAut} \frak f_2$ for $M_1(\mathbb K)$.
Actually all of them are attributed to \cite{Dr} Theorem ${\mathrm A}^\prime$.
\end{rem}

The above $ \varDelta(\sigma)$ may not be uniquely chosen  but it can be chosen independently
from $c$ by the construction.
When we make such a choice,  we  occasionally denote $ \varDelta(g)$ 
instead of $ \varDelta(\sigma)$ by abuse of notations.
We note that the first two terms on the right hand side of \eqref{twisting formula}
commute each other, so do the last two terms.

\begin{defn}
We call such $ \varDelta(\sigma)$ in $\widehat{{\mathbb K}[{\mathcal{ISL}}_{\uparrow\uparrow}]}^\times$
a {\it twistor}
\footnote{
We call this element twistor
because it is related to Drinfeld's notion of twisting in \cite{Dr}.
}
of $\sigma =(c,g)\in GRT(\mathbb K)$.
For  a twisor $ \varDelta(\sigma)$,
we put 
$ \varDelta(\sigma,\uparrow):=\uparrow\in\widehat{{\mathbb K}[{\mathcal{ISL}}_{\uparrow}]}^\times$ and 
\begin{equation}
 \varDelta(\sigma,\uparrow^n):=   \varDelta(\sigma)_{12\cdots n-1,n}\cdots  \varDelta(\sigma)_{12,3}\cdot  \varDelta(\sigma)_{1,2}
\in \widehat{{\mathbb K}[{\mathcal{ISL}}_{\uparrow^n}]}^\times
\end{equation}
for $n\geqslant 2$.
Here $ \varDelta(\sigma)_{1\cdots k,k+1}$ means the element in $\widehat{{\mathbb K}[{\mathcal{ISL}}_{\uparrow^n}]}$
obtained  multi-doubling of the first strand of $ \varDelta(\sigma)\in\widehat{{\mathbb K}[{\mathcal{ISL}}_{\uparrow\uparrow}]}$
by $k$ strands, summing up all possible lifts of chords
and putting  $n-k-1$ chordless line  $\uparrow^{n-k-1}$ on its right.
For any sequence  $\epsilon=(\epsilon_i)_{i=1}^n$,
we determine  the element $ \varDelta(\sigma,\epsilon)$ in $\widehat{{\mathbb K}[{\mathcal{ISL}}_{\epsilon}]}$
by reversing corresponding all arrows of $ \varDelta(\sigma,\uparrow^n)$.
We note that $ \varDelta(\sigma,\uparrow\uparrow)= \varDelta(\sigma)$ and 
$g= \varDelta(\sigma,\uparrow\uparrow\uparrow)_{3,2,1}\cdot
 \varDelta(\sigma,\uparrow\uparrow\uparrow)^{-1}$
by \eqref{symmetric condition}.

An automorphism $\theta_{\epsilon,\epsilon'}^{ \varDelta(\sigma)}$
of 
$\widehat{{\mathbb K}[{\mathcal{IT}}_{\epsilon,\epsilon'}]}$ 
associated to a twistor $ \varDelta(\sigma)$ can be constructed piecewise as follows.

(1)
when $D\in A_{k,l}^\epsilon$, we define
$$
\theta_{\epsilon,\epsilon'}^{ \varDelta(\sigma)}(D):= 
D\cdot (\varDelta(\sigma)^{-1}_{k+1,k+2})^{s(D)}\cdot
(\nu_g)^{s(D)}_{k+2}
\cdot g_{1\cdots k,k+1,k+2}^{s( D)}. 
$$
Here the term $(\varDelta(\sigma)^{-1}_{k+1,k+2})^{s(D)}$ means the element in 
$\widehat{{\mathbb K}[{\mathcal{ISL}}_{s(D)}]}$
obtained by putting 
the trivial chordless diagram with $k$-strands  on the left of
$\varDelta(\sigma)^{-1}$ and
and the trivial chordless diagram with $l$-strands  on its right.

(2)
when $ D=(b_n,\epsilon)\in \widehat{IB_\tau^\epsilon}$ ,
we define
$$
\theta_{\epsilon,\epsilon'}^{ \varDelta(\sigma)}(D)
:= 
(\rho_n(\sigma)(b_n),\epsilon). 
$$

(3)
when $ D\in C_{k,l}^\epsilon$, we define
$$
\theta_{\epsilon,\epsilon'}^{ \varDelta(\sigma)}(D)
:=
g_{1\cdots k,k+1,k+2}^{-1,t( D)} 
\cdot (\varDelta(\sigma)_{k+1,k+2})^{s(D)}\cdot D.
$$

Here $\mu_g^\epsilon$ is the one defined in  Definition \ref{GRT-action on infinitesimal tangles}.
\end{defn}

\begin{lem}
For  any sequences $\epsilon, \epsilon'$, any $\sigma\in GRT(\mathbb K)$ and any twistor
$\varDelta(\sigma)$,
the above construction determines a well-defined automorphism
$\theta_{\epsilon,\epsilon'}^{ \varDelta(\sigma)}$ of 
$\widehat{{\mathbb K}[{\mathcal{IT}}_{\epsilon,\epsilon'}]}$
which is compatible with the composition map \eqref{composition of CD}.
\end{lem}

\begin{proof}
This can be verified by an almost same way to the proof of Proposition \ref{GRT-action on KIT}
except for  compatibilities of (IT5) and (IT6).
\begin{itemize}
\item 
To check the compatibility of the first equality of (IT5), it is enough to show the equality
illustrated in Figure \ref{Compatibility with the first equality of (IT5)}.
By the identification of 
$\widehat{{\mathbb K}[{\mathcal{ISL}}_{\uparrow}]}$
with 
$\widehat{{\mathbb K}[{\mathcal{IK}}]}$
given in Lemma \ref{identification between infinitesimal long knots and proalgebraic knots} (cf.\eqref{CD(O)=CD(1)}),
showing the validity
is deduced to showing the equality illustrated in Figure \ref{proof of compatibility with  the first equality of (IT5)}.
It is immediate to see because we have (IT6) and
$\varDelta(\sigma)_{2,1}=\varDelta(\sigma)$ by \eqref{symmetric condition}.
\begin{figure}[h]
\begin{tabular}{c}
  \begin{minipage}{0.5\hsize}
      \begin{center}
         \begin{tikzpicture}
                      \draw (1.,0.5) rectangle (2.2,1.0);
                     \draw (1.6,0.7) node{$\varDelta(\sigma)$};
                     \draw[-] (1.3,0.2)--(1.3,0.5);
                    \draw[-] (1.3,1.0)--(1.3,3.1);
                     \draw[-] (1.3,0.2)  arc (180:360:0.3);
                     \draw[-] (1.9,0.2)--(1.9,0.5);
                     \draw[-] (1.9,1.0)--(1.9,1.9);
                     \draw[-] (1.9,2.4)--(1.9,2.7);
                     \draw[-] (2.5,-0.2)--(2.5,1.9);
                     \draw[-] (2.5,2.4)--(2.5,2.7);
                      \draw (1.6,1.9) rectangle (2.8,2.4);
                     \draw (2.2,2.1) node{$\varDelta(\sigma)^{-1}$};
                     \draw[-] (1.9,2.7)  arc (180:0:0.3);
\draw  (3.2,1.3)node{$=$};
                     \draw[-] (4,-0.2)--(4,3.1);
\draw[color=white, very thick] (0.9,0.2)--(2.6,0.2) (0.9,1.45)--(2.6,1.45)  (0.9,2.7)--(2.6,2.7);
         \end{tikzpicture}
\caption{}
\label{Compatibility with  the first equality of (IT5)}
     \end{center}
 \end{minipage}

 \begin{minipage}{0.5\hsize}
     \begin{center}
         \begin{tikzpicture}
                      \draw (1.,0.5) rectangle (2.2,1.0);
                     \draw (1.6,0.7) node{$\varDelta(\sigma)$};
                     \draw[-] (1.3,0.2)--(1.3,0.5);
                    \draw[-] (1.3,1.0)--(1.3,1.2);
                    \draw[-] (1.3,1.7)--(1.3,1.9);
                    \draw[-] (1.3,2.4)--(1.3,2.7);
                     \draw[-] (1.3,0.2)  arc (180:360:0.3);
                     \draw[-] (1.9,0.2)--(1.9,0.5);
                     \draw[-] (1.9,1.0)--(1.9,1.2);
                     \draw[-] (1.9,1.7)--(1.9,1.9);
                     \draw[-] (1.9,2.4)--(1.9,2.7);

                    \draw[-] (1.3,1.2)--(1.9,1.7);
                    \draw[-] (1.3,1.7)--(1.9,1.2);

                      \draw (1.,1.9) rectangle (2.2,2.4);
                     \draw (1.6,2.1) node{$\varDelta(\sigma)^{-1}$};
                     \draw[-] (1.9,2.7)  arc (0:180:0.3);
\draw  (2.5,1.4)node{$=$};
                     \draw[-] (3.8,1.45)  arc (0:360:0.4);
\draw[color=white, very thick] (0.9,0.2)--(2.6,0.2) (0.9,1.2)--(2.6,1.2) (0.9,1.7)--(2.6,1.7)  (0.9,2.7)--(2.6,2.7) (2.9,1.45)--(3.9,1.45);
         \end{tikzpicture}
\caption{}
\label{proof of compatibility with  the first equality of (IT5)}
      \end{center}
  \end{minipage}
\end{tabular}
\end{figure}

The compatibility of the second equality of (IT5) can be done in the same way.
\item
To check  the compatibility of the first equality of (IT6), it is enough to show the equality
illustrated in Figure \ref{Compatibility with the first equality of (IT6)}.
Actually it is a consequence of (IT3), (IT5) and (IT6).
The compatibility of the second equality of (IT6) can be done in the same way.
\begin{figure}[h]
\begin{center}
         \begin{tikzpicture}
                      \draw (1.,0.5) rectangle (2.2,1.0);
                     \draw (1.6,0.7) node{$\varDelta(\sigma)$};
                     \draw[-] (1.3,0.2)--(1.3,0.5);
                    \draw[-] (1.3,1.0)--(1.3,1.2);
                    \draw[-] (1.3,1.7)--(1.3,1.9);
                    \draw[-] (1.3,2.4)--(1.3,2.7);
                     \draw[-] (1.3,0.2)  arc (180:360:0.3);
                     \draw[-] (1.9,0.2)--(1.9,0.5);
                     \draw[-] (1.9,1.0)--(1.9,1.2);
                     \draw[-] (1.9,1.7)--(1.9,1.9);
                     \draw[-] (1.9,2.4)--(1.9,2.7);

                    \draw[-] (1.3,1.2)--(1.9,1.7);
                    \draw[-] (1.3,1.7)--(1.9,1.2);

                      \draw (0.7,1.9) rectangle (2.5,2.4);
  \draw (1.3,2.15) node{$\exp\{c\alpha$};
                    \draw[-]  (1.95,2)--(1.95,2.3) ;
                   \draw[densely dotted] (1.95,2.15)--(2.25,2.15);
                    \draw[-]  (2.25,2)--(2.25,2.3) ;
\draw  (2.4, 2.15) node{{$\}$}};
\draw[color=white, very thick] (1.2,0.2)--(2.,0.2) (1.2,1.2)--(2,1.2) (1.2,1.7)--(2,1.7)  ;

\draw  (2.8,1.3)node{$=$};

                      \draw (3.2,1.1) rectangle (4.4,1.6);
                     \draw (3.8,1.3) node{$\varDelta(\sigma)$};
                     \draw[-] (3.5,0.8)--(3.5,1.1) (3.5,1.6)--(3.5, 2.);
                     \draw[-] (3.5,0.8)  arc (180:360:0.3);
                     \draw[-] (4.1,0.8)--(4.1,1.1)  (4.1,1.6)--(4.1, 2.);
\draw[color=white, very thick] (3.4,0.8)--(4.2,0.8)  ;
         \end{tikzpicture}
\caption{}
\label{Compatibility with the first equality of (IT6)}
\end{center}
\end{figure}
\end{itemize}
\end{proof}

It is easily shown that the restriction of $\theta_{\epsilon,\epsilon}^{ \varDelta(\sigma)}$
into the algebra $\widehat{{\mathbb K}[{\mathcal{ISL}}_{\epsilon}]}$
of infinitesimal string links
induces its automorphism, 
denoted by $\theta_{\epsilon}^{ \varDelta(\sigma)}$.
The following explains its relationship with our previous automorphism 
$\rho_{\epsilon}(\sigma)$ in \eqref{GRT to Aut KIT}.

\begin{prop}\label{prop:twisting=our action}
For any $\sigma\in GRT(\mathbb K)$ and any sequence $\epsilon$, 
\begin{equation}\label{twisting=our action}
    \rho_{\epsilon}(\sigma)(D)= \theta_{\epsilon}^{ \varDelta(\sigma)}(D)
\end{equation}
holds for all  
$D\in \widehat{{\mathbb K}[{\mathcal{ISL}}_{\epsilon}]}$.
\end{prop}

\begin{proof}
Differences between two actions
$\theta_{\epsilon}^{ \varDelta(\sigma)}$ and $\rho_{\epsilon}(\sigma)$
are observed only on the action on $A_{k,l}^\epsilon$ and  $C_{k,l}^\epsilon$.
It is enough  to show that equality of Figure \ref{Compatibility with  the first equality of (IT5)},
which were proved in the above lemma.
\end{proof}

When we restrict $\sigma$ into the unipotent part $GRT_1(\mathbb K)$,
we obtain a relationship of $\theta_{\epsilon, \epsilon'}^{ \varDelta(\sigma)}$
with the identity map of 
$\widehat{{\mathbb K}[{\mathcal{IT}}_{\epsilon,\epsilon'}]}$
as follows.

\begin{prop}\label{prop:twisting=inner conjugation}
For any $\sigma=(1,g)\in GRT_1(\mathbb K)$ and  any sequences $\epsilon$ and $\epsilon'$,
\begin{equation}\label{twisting=inner conjugation}
\theta_{\epsilon,\epsilon'}^{ \varDelta(\sigma)}(D)=
 \varDelta(\sigma,\epsilon)\cdot D 
\cdot  \varDelta(\sigma,\epsilon')^{-1}.
\end{equation}
holds for any 
$D\in
\widehat{{\mathbb K}[{\mathcal{IT}}_{\epsilon,\epsilon'}]}$.
\end{prop}

\begin{proof}
The equation can be checked piecewise.
\begin{itemize}
\item When $D\in A_{k,l}^\epsilon$,
we may assume that $D=a_{k,l}^\epsilon$
with $s(D)=\epsilon_1$ and $t(D)=\epsilon_2$. Then
\begin{align*}
\theta_{\epsilon_2,\epsilon_1}^{ \varDelta(\sigma)}(a_{k,l}^\epsilon)= &
a_{k,l}^\epsilon\cdot\varDelta(\sigma)_{k+1,k+2}^{-1}\cdot (\mu_g)_{k+2}\cdot
g_{1\cdots k, k+1,k+2} \\
=& a_{k,l}^\epsilon\cdot\varDelta(\sigma)_{k+1,k+2}^{-1}\cdot (\mu_g)_{k+2} \\
&\cdot
\varDelta(\sigma)_{k+1,k+2}\cdot \varDelta(\sigma)_{1\cdots k, k+1 \ k+2}\cdot
\varDelta(\sigma)_{1\cdots k+1,k+2}^{-1}\cdot \varDelta(\sigma)_{1\cdots k,k+1}^{-1} \\
=& a_{k,l}^\epsilon\cdot  (\mu_g)_{k+2} \cdot
\varDelta(\sigma)_{1\cdots k+1,k+2}^{-1}\cdot \varDelta(\sigma)_{1\cdots k,k+1}^{-1}. \\
\intertext{By Lemma \ref{mu=1},}
=& a_{k,l}^\epsilon\cdot  
\varDelta(\sigma)_{1\cdots k+1,k+2}^{-1}\cdot \varDelta(\sigma)_{1\cdots k,k+1}^{-1} \\
=& \varDelta(\sigma)_{1,2}\cdots  \varDelta(\sigma)_{1\cdots k+l-1,k+l}\cdot
 a_{k,l}^\epsilon \cdot
\varDelta(\sigma)_{1\cdots k+l+1,k+l+2}^{-1}\cdots \varDelta(\sigma)_{1,2}^{-1} \\
=&\varDelta(\sigma,\epsilon_2)\cdot a_{k,l}^\epsilon \cdot
\varDelta(\sigma,\epsilon_1)^{-1}.
\end{align*} 
\item When $D\in \widehat{IB}$,
it is enough to to show the case when
$D=\tau_{i,i+1}$ or $t_{i,i+1} \in U{\frak b}_n$  ($1\leqslant i \leqslant n-1$).
If $D=\tau_{i,i+1}$  with $s(D)=\epsilon_1$ and $t(D)=\epsilon_2$,
by Proposition \ref{proalg-GRT-action theorem on braids},
\begin{align*}
\theta_{\epsilon_2,\epsilon_1}^{ \varDelta(\sigma)}(\tau_{i,i+1} )= &
g_{1\cdots i-1,i,i+1}^{-1}\cdot{\tau_{i,i+1}}\cdot g_{1\cdots i-1,i,i+1} \\
=&\varDelta(\sigma)_{1\cdots i-1,i}\cdot  \varDelta(\sigma)_{1\cdots i,i+1}\cdot
 \varDelta(\sigma)_{1\cdots i-1, i \ i+1}^{-1}\cdot \varDelta(\sigma)_{i,i+1}^{-1} \\
&\cdot \tau_{i,i+1}\cdot
\varDelta(\sigma)_{i,i+1}\cdot \varDelta(\sigma)_{1\cdots i-1, i \ i+1}\cdot
\varDelta(\sigma)_{1\cdots i,i+1}^{-1}\cdot \varDelta(\sigma)_{1\cdots i-1,i}^{-1}. \\
\intertext{By \eqref{symmetric condition},}
=&\varDelta(\sigma)_{1\cdots i-1,i}\cdot  \varDelta(\sigma)_{1\cdots i,i+1}\cdot\tau_{i,i+1}\cdot
\varDelta(\sigma)_{1\cdots i,i+1}^{-1}\cdot \varDelta(\sigma)_{1\cdots i-1,i}^{-1} \\
=&\varDelta(\sigma)_{1,2}\cdots  \varDelta(\sigma)_{1\cdots n-1,n}\cdot\tau_{i,i+1}\cdot
\varDelta(\sigma)_{1\cdots n-1,n}^{-1}\cdots\varDelta(\sigma)_{1,2}^{-1}\\
=&\varDelta(\sigma,\epsilon_2)\cdot \tau_{i,i+1}\cdot
\varDelta(\sigma,\epsilon_1)^{-1}.
\end{align*}
If $D=t_{i,i+1}$  with $s(D)=\epsilon_1$ and $t(D)=\epsilon_2$,
again by Proposition \ref{proalg-GRT-action theorem on braids},
\begin{align*}
\theta_{\epsilon_2,\epsilon_1}^{ \varDelta(\sigma)}(t_{i,i+1} )= &
g_{1\cdots i-1,i,i+1}^{-1}\cdot{t_{i,i+1}}\cdot g_{1\cdots i-1,i,i+1} \\
=&\varDelta(\sigma)_{1\cdots i-1,i}\cdot  \varDelta(\sigma)_{1\cdots i,i+1}\cdot
 \varDelta(\sigma)_{1\cdots i-1, i \ i+1}^{-1}\cdot \varDelta(\sigma)_{i,i+1}^{-1} \\
&\cdot t_{i,i+1}\cdot
\varDelta(\sigma)_{i,i+1}\cdot \varDelta(\sigma)_{1\cdots i-1, i \ i+1}\cdot
\varDelta(\sigma)_{1\cdots i,i+1}^{-1}\cdot \varDelta(\sigma)_{1\cdots i-1,i}^{-1} \\
\intertext{By Lemma \ref{center lemma},}
=&\varDelta(\sigma)_{1\cdots i-1,i}\cdot  \varDelta(\sigma)_{1\cdots i,i+1}\cdot t_{i,i+1}\cdot
\varDelta(\sigma)_{1\cdots i,i+1}^{-1}\cdot \varDelta(\sigma)_{1\cdots i-1,i}^{-1} \\
=&\varDelta(\sigma)_{1,2}\cdots  \varDelta(\sigma)_{1\cdots n-1,n}\cdot t_{i,i+1}\cdot
\varDelta(\sigma)_{1\cdots n-1,n}^{-1}\cdots\varDelta(\sigma)_{1,2}^{-1}\\
=&\varDelta(\sigma,\epsilon_2)\cdot t_{i,i+1}\cdot
\varDelta(\sigma,\epsilon_1)^{-1}.
\end{align*}
We note that we use $c=1$  for the second equality.
\item When $D\in C_{k,l}^\epsilon$,
we may assume that $D=c_{k,l}^\epsilon$  with $s(D)=\epsilon_1$ and $t(D)=\epsilon_2$. Then 
\begin{align*}
\theta_{\epsilon_2,\epsilon_1}^{ \varDelta(\sigma)}(c_{k,l}^\epsilon)= &
g_{1\cdots k, k+1,k+2}^{-1}\cdot \varDelta(\sigma)_{k+1,k+2}\cdot c_{k,l}^\epsilon \\
=& \varDelta(\sigma)_{1\cdots k,k+1}\cdot  \varDelta(\sigma)_{1\cdots k+1,k+2}\cdot
 \varDelta(\sigma)_{1\cdots k, k+1 \ k+2}^{-1}\cdot
 \varDelta(\sigma)_{k+1,k+2}^{-1} \\
& \cdot  \varDelta(\sigma)_{k+1,k+2}\cdot c_{k,l}^\epsilon \\
=& \varDelta(\sigma)_{1\cdots k,k+1}\cdot  \varDelta(\sigma)_{1\cdots k+1,k+2}\cdot
 c_{k,l}^\epsilon \\
=& \varDelta(\sigma)_{1,2}\cdots  \varDelta(\sigma)_{1\cdots k+l+1,k+l+2}\cdot
 c_{k,l}^\epsilon \cdot
\varDelta(\sigma)_{1\cdots k+l,k+l}^{-1}\cdots \varDelta(\sigma)_{1,2}^{-1} \\
=&\varDelta(\sigma,\epsilon_2)\cdot c_{k,l}^\epsilon \cdot
\varDelta(\sigma,\epsilon_1)^{-1}.
\end{align*}
\end{itemize}
Hence we get the equality \eqref{twisting=inner conjugation} for any
$D\in\widehat{{\mathbb K}[{\mathcal{IT}}_{\epsilon,\epsilon'}]}$
\end{proof}

The followings are required to prove the previous proposition.

\begin{lem}\label{mu=1}
For $\sigma=(c,g)\in GRT(\mathbb K)$, the infinitesimal long knot $\mu_g\in 
\widehat{{\mathbb K}[{\mathcal{ISL}}_{\uparrow}]}$
is actually equal to the trivial 
chordless chord diagram on $\uparrow$.
\end{lem}

\begin{proof}
It is enough to show that $\Lambda_g$ 
is the  trivial diagram $\downarrow$.
Since
$g= \varDelta(\sigma)_{2,3}\cdot  \varDelta(\sigma)_{1,23}\cdot  \varDelta(\sigma)^{-1}_{12,3}\cdot  \varDelta(\sigma)^{-1}_{1,2}$,
it is enough to show the equality depicted in Figure \ref{Proof of Lemma mu=1}.
\begin{figure}[h]
\begin{center}
         \begin{tikzpicture}
                      \draw (1.,0.5) rectangle (2.2,1.0);
                     \draw (1.6,0.7) node{$\varDelta(\sigma)^{-1}$};
                     \draw[-] (1.3,0.2)--(1.3,0.5);
                    \draw[-] (1.3,1.0)--(1.3,3.1);
                     \draw[-] (1.3,0.2)  arc (180:360:0.3);
                     \draw[-] (1.9,0.2)--(1.9,0.5);
                     \draw[-] (1.9,1.0)--(1.9,1.9);
                     \draw[-] (1.9,2.4)--(1.9,2.7);
                     \draw[-] (2.5,-0.2)--(2.5,1.9);
                     \draw[-] (2.5,2.4)--(2.5,2.7);
                      \draw (1.6,1.9) rectangle (2.8,2.4);
                     \draw (2.2,2.1) node{$\varDelta(\sigma)$};
                     \draw[-] (1.9,2.7)  arc (180:0:0.3);
\draw  (3.2,1.3)node{$=$};
                     \draw[-] (4,-0.2)--(4,3.1);
\draw[color=white, very thick] (0.9,0.2)--(2.6,0.2) (0.9,1.45)--(2.6,1.45)  (0.9,2.7)--(2.6,2.7);
         \end{tikzpicture}
\caption{Proof of Lemma \ref{mu=1}}
\label{Proof of Lemma mu=1}
\end{center}
\end{figure}
It can be proved in a same way to Figure \ref{proof of compatibility with  the first equality of (IT5)}.
\end{proof}

The following  is proved in a topological way.

\begin{lem}\label{center lemma}
The element $t_{12}$ lies in the center of 
$\widehat{{\mathbb K}[{\mathcal{ISL}}_{\uparrow\uparrow}]}$.
Namely 
\begin{equation}\label{center equation}
t_{12}\cdot D=D\cdot t_{12}
\end{equation}
holds for any $D\in 
\widehat{{\mathbb K}[{\mathcal{ISL}}_{\uparrow\uparrow}]}$.
\end{lem}

\begin{proof}
First, for a discrete case, we have 
$T\cdot (\sigma_{1,2})^{2\alpha}=(\sigma_{1,2})^{2\alpha}\cdot T$
for any $T\in\mathcal{SL}(\uparrow\uparrow)$ and $\alpha\in\mathbb Z$.
Whence we will have the same equality 
$T\cdot (\sigma_{1,2})^{2\alpha}=(\sigma_{1,2})^{2\alpha}\cdot T$
for any $T\in\widehat{{\mathbb K}[{\mathcal{SL}}_{\uparrow\uparrow}]}$ and $\alpha\in\mathbb K$.
Then by the isomorphism \eqref{isomorphism between KSL and KISL},
we have 
$D\cdot \exp\{\alpha t_{1,2}\}=\exp\{\alpha t_{1,2}\}\cdot D$
for any $D\in \widehat{{\mathbb K}[{\mathcal{ISL}}_{\uparrow\uparrow}]}$
and $\alpha\in\mathbb K$.
Since the space $\widehat{{\mathbb K}[{\mathcal{ISL}}_{\uparrow\uparrow}]}$ is completed,
we have the equality \eqref{center equation}.
\end{proof}

The proposition below states that our $GRT_1(\mathbb K)$-action on
 $\widehat{{\mathbb K}[{\mathcal{ISL}}_{\epsilon}]}$ is described by
an inner conjugation action of twistor.

\begin{prop}\label{inner automorphism theorem on chord diagrams}
Let $\epsilon$ 
be any sequence.
The action $\rho_{\epsilon}$ of \eqref{GRT to Aut KISL}
restricted into the unipotent part  $GRT_1(\mathbb K)$
on the algebra  $\widehat{{\mathbb K}[{\mathcal{ISL}}_{\epsilon}]}$  of infinitesimal string links
is  given by the inner conjugation of the twistor $\varDelta(\sigma,\epsilon)$, i.e.
\begin{equation}\label{inner action on chords of string links}
\rho_{\epsilon}(\sigma)(D)= \varDelta(\sigma,\epsilon)\cdot D\cdot  \varDelta(\sigma,\epsilon)^{-1}
\end{equation}
holds for $\sigma=(1,g)\in GRT_1(\mathbb K)$ and 
$D\in\widehat{{\mathbb K}[{\mathcal{ISL}}_{\epsilon}]}$. 
\end{prop}

\begin{proof}
The formula is obtained by combining 
\eqref{prop:twisting=our action} and \eqref{prop:twisting=inner conjugation}.
\end{proof}

Next we discuss  $GT_1(\mathbb K)$-action on
the algebra  $\widehat{{\mathbb K}[{\mathcal{SL}}_{\epsilon}]}$ of proalgebraic string links.

\begin{nota}
(1)
For $n>1$, 
$\epsilon_i:\widehat{{\mathbb K}[{\mathcal{SL}}_{\uparrow^n}]}\to
\widehat{{\mathbb K}[{\mathcal{SL}}_{\uparrow^{n-1}}]}$
($i=1,2\dots n$) means the map  removing the $i$-th strand on
each proalgebraic string links.

(2)
For each $\Gamma\in \widehat{{\mathbb K}[{\mathcal{SL}}_{\uparrow^n}]}$,
we denote  $\Gamma_{1,\dots,n}$ (resp. $\Gamma_{2,\dots,n+1}$)  to be the element 
in $\widehat{{\mathbb K}[{\mathcal{SL}}_{\uparrow^{n+1}}]}$
obtained by putting a straight line  on the right (resp. the left) of $\Gamma$
and  $\Gamma_{1,\dots, i-1 ,i\ i+1,i+2, \dots, n+1}$ ($i=1,2,\dots,n$)
to be also the element  in $\widehat{{\mathbb K}[{\mathcal{SL}}_{\uparrow^{n+1}}]}$
obtained by {\it doubling} the $i$-th strand.

(3)
Particularly for $\Gamma\in \widehat{{\mathbb K}[{\mathcal{SL}}_{\uparrow\uparrow}]}$,
the symbol $\Gamma_{2,1}$ denotes
the element $\sigma_{1,2}^{-1}\cdot \Gamma\cdot \sigma_{1,2}$
in $\Gamma\in \widehat{{\mathbb K}[{\mathcal{SL}}_{\uparrow\uparrow}]}$.
\end{nota}

Hereafter we regard $F_2(\mathbb K)$ and $P_5(\mathbb K)$ to be the subspaces
of $\widehat{{\mathbb K}[{\mathcal{SL}}_{\uparrow\uparrow\uparrow}]}$ and
$\widehat{{\mathbb K}[{\mathcal{SL}}_{\uparrow\uparrow\uparrow\uparrow}]}$ respectively.
The following lemma is a $GT(\mathbb K)$-analogue of 
Twistor Lemma \ref{twistor lemma}.

\begin{lem}[Twistor Lemma]\label{twistor lemma on GT}
Let $\sigma=(\lambda,f)\in GT(\mathbb K)$, thus $\lambda\in {\mathbb K}^\times$ and $f\in F_2(\mathbb K)\subset 
\widehat{{\mathbb K}[{\mathcal{SL}}_{\uparrow\uparrow\uparrow}]}$.
Then $f$ is gauge equivalent to $1$,
namely, there exists $ \varpi(\sigma)\in \widehat{{\mathbb K}[{\mathcal{SL}}_{\uparrow\uparrow}]}^\times$
satisfying
\begin{equation}\label{normalization condition on Pi}
\epsilon_1( \varpi(\sigma))=\epsilon_2( \varpi(\sigma))=\uparrow
\end{equation}
and   the symmetric condition
\begin{equation}\label{symmetric condition on Pi}
 \varpi(\sigma)= \varpi(\sigma)_{2,1}
\end{equation}
such that
\begin{equation}\label{twisting formula of Pi}
f= \varpi(\sigma)_{2,3}\cdot  \varpi(\sigma)_{1,23}\cdot  \varpi(\sigma)^{-1}_{12,3}\cdot  \varpi(\sigma)^{-1}_{1,2}
\end{equation}
holds in $\widehat{{\mathbb K}[{\mathcal{SL}}_{\uparrow\uparrow\uparrow}]}^\times$. 
\end{lem}

\begin{proof}
Fix an element $p=(1,\varphi)\in M_1(\mathbb K)$.
Then by Proposition \ref{GRT-GT-bitorsor}
it yields an isomorphism 
$$r_p:GT_1(\mathbb K) \simeq GRT_1(\mathbb K).$$
By Proposition \ref{KT to KIT}.(3), we have an isomorphism
$$\rho_{\uparrow\uparrow\uparrow}(p):\widehat{{\mathbb K}[{\mathcal{SL}}_{\uparrow\uparrow\uparrow}]}\simeq
\widehat{{\mathbb K}[{\mathcal{ISL}}_{\uparrow\uparrow\uparrow}]}.$$
Put
$\varpi(\sigma):=\rho_{\uparrow\uparrow\uparrow}(p)^{-1}\left( \varDelta\left(r_p(\sigma)\right)\right)$.
Then using Proposition \ref{three torsor maps} and Twistor Lemma \ref{twistor lemma},
we can check the validities of 
\eqref{normalization condition on Pi}--\eqref{twisting formula of Pi} for $\varpi(\sigma)$
by direct calculations.
\end{proof}

The above $ \varpi(\sigma)$ may not be uniquely chosen  and
may depend on $\lambda\in{\mathbb K}^\times$ unlike the case for $GRT(\mathbb K)$.
Here we note again that the first two terms on the right hand side of \eqref{twisting formula}
commute each other, so do the last two terms.

\begin{defn}
We call such $ \varpi(\sigma)$ in $\widehat{{\mathbb K}[{\mathcal{SL}}_{\uparrow\uparrow}]}^\times$
a {\it twistor} of $\sigma =(\lambda,f)\in GT(\mathbb K)$.
For  a twisor $ \varpi(\sigma)$, we put
$ \varpi(\sigma,\uparrow):=\uparrow\in\widehat{{\mathbb K}[{\mathcal{SL}}_{\uparrow}]}^\times$ and 
\begin{equation}
 \varpi(\sigma,\uparrow^n):=   \varpi(\sigma)_{12\cdots n-1,n}\cdots  \varpi(\sigma)_{12,3}\cdot  \varpi(\sigma)_{1,2}
\in \widehat{{\mathbb K}[{\mathcal{SL}}_{\uparrow^n}]}^\times
\end{equation}
for $n\geqslant 2$.
Here $ \varpi(\sigma)_{1\cdots k,k+1}$ means the element in 
$\widehat{{\mathbb K}[{\mathcal{SL}}_{\uparrow^n}]}$
obtained  multi-doubling of the first strand of $ \varpi(\sigma)\in
\widehat{{\mathbb K}[{\mathcal{SL}}_{\uparrow\uparrow}]}$
by $k$ strands, 
and putting  $n-k-1$ straight lines  $\uparrow^{n-k-1}$ on its right.
For any sequence  $\epsilon=(\epsilon_i)_{i=1}^n$,
we determine  the element $ \varpi(\sigma,\epsilon)$ in $
\widehat{{\mathbb K}[{\mathcal{SL}}_{\epsilon}]}^\times$
by reversing all corresponding arrows of $ \varpi(\sigma,\uparrow^n)$.
We note that $ \varpi(\sigma,\uparrow\uparrow)= \varpi(\sigma)$ and 
$f= \varpi(\sigma,\uparrow\uparrow\uparrow)_{3,2,1}\cdot
 \varpi(\sigma,\uparrow\uparrow\uparrow)^{-1}$
by \eqref{symmetric condition on Pi}.
\end{defn}

Our theorem in this subsection is to state that
our $GT_1(\mathbb K)$-action on $\widehat{{\mathbb K}[{\mathcal{SL}}_{\epsilon}]}$
is described by the inner conjugation action by the twistor.

\begin{thm}\label{inner automorphism theorem on tangles}
Let $\epsilon$ be any sequence.
The action
$$
\rho_{\epsilon}: GT_1(\mathbb K)\to  \mathrm{Aut}\ \widehat{{\mathbb K}[{\mathcal{SL}}_{\epsilon}]}
$$
induced by \eqref{GT to Aut KSL},
of the unipotent part  $GT_1(\mathbb K)$
on the algebra  $\widehat{{\mathbb K}[{\mathcal{SL}}_{\epsilon}]}$ of proalgebraic string links
is simply given by an inner conjugation of twistor  $\varpi(\sigma,\epsilon)\in
\widehat{{\mathbb K}[{\mathcal{SL}}_{\epsilon}]}^\times$.
That is,
\begin{equation}\label{inner action on string links}
\rho_{\epsilon}(\sigma)(\Gamma)= \varpi(\sigma,\epsilon)\cdot \Gamma\cdot  \varpi(\sigma,\epsilon)^{-1}
\end{equation}
holds for $\sigma=(1,f)\in GT_1(\mathbb K)$ and 
$\Gamma\in \widehat{{\mathbb K}[{\mathcal{SL}}_{\epsilon}]}$.
\end{thm}

\begin{proof}
Let $p=(1,\varphi)\in M_1(\mathbb K)$ be an element taken in the proof of 
Twistor Lemma \ref{twistor lemma on GT}.
By Proposition \ref{GRT-GT-bitorsor}
it yields an isomorphism  $r_p:GT_1(\mathbb K) \simeq GRT_1(\mathbb K)$.
By Proposition \ref{KT to KIT}.(3), we have an isomorphism
$$\rho_\epsilon(p):\widehat{{\mathbb K}[{\mathcal{SL}}_{\uparrow^n}]}\simeq
\widehat{{\mathbb K}[{\mathcal{ISL}}_{\uparrow^n}]}.$$
Our claim follows from 
Proposition \ref{three torsor maps} and \ref{inner automorphism theorem on chord diagrams}.
We may take $\varpi(\sigma,\epsilon)$
in  the above \eqref{inner action on string links}.
by
$\rho_\epsilon(p)^{-1}\left( \varDelta(r_p(\sigma),\epsilon)\right)$.
\end{proof}

Here is a corollary of  Theorem \ref{inner automorphism theorem on tangles}.

\begin{cor}\label{inner action on braids}
The restricted action  of  \eqref{GT to Aut B_n}  into the unipotent part $GT_1(\mathbb K)$
on proalgebraic pure braids, denoted by
$$
\rho_n:GT_1(\mathbb K)\to \mathrm{Aut}\ \widehat{{\mathbb K}[P_n]},
$$
is simply given by an inner conjugation by the element 
$\varpi(\sigma,\uparrow^n)\in \widehat{{\mathbb K}[{\mathcal{SL}}_{\uparrow^n}]}$.
That is, 
\begin{equation}\label{eq:inner action on braids}
\rho_n(\sigma)(x_{ij})=\varpi(\sigma,\uparrow^n)\cdot x_{ij}\cdot\varpi(\sigma,\uparrow^n)^{-1}
\end{equation}
holds for $1\leqslant i,j\leqslant n$ and  $\sigma\in GT_1(\mathbb K)$.
\end{cor}

We note that 
though $\varpi(\sigma,\uparrow^n)$ belongs to $\widehat{{\mathbb K}[{\mathcal{SL}}_{\uparrow^n}]}$
the right hand side of \eqref{eq:inner action on braids} belongs to 
its subspace $\widehat{{\mathbb K}[P_n]}$. 

\begin{rem}\label{faithful-trivial}
(1)
Since  the $GT(\mathbb K)$-action on $\widehat{{\mathbb K}[P_n]}$ is faithful
for $n\geqslant 3$,
the action 
$\rho_{\epsilon}: GT_1(\mathbb K)\to  \mathrm{Aut}\ \widehat{{\mathbb K}[{\mathcal{SL}}_{\epsilon}]}$
is faithful for $n\geqslant 3$,
where $n$ is the number of strings, i.e. the cardinality of $\epsilon$,
by Remark \ref{KP injected to KSL}.

(2)
When $n=1$,  the action is far from faithful.
Actually the kernel of the action
$\rho_{\epsilon}: GT_1(\mathbb K)\to  \mathrm{Aut}\ \widehat{{\mathbb K}[{\mathcal{SL}}_{\epsilon}]}$
(with $\epsilon=\uparrow$ or $\downarrow$)
on proalgebraic long knots
is the unipotent part $GT_1(\mathbb K)$
because $\varpi(\sigma,\epsilon)$ in Theorem \ref{inner automorphism theorem on tangles}
is trivial.
\end{rem}

The rest case $n=2$ looks unclear.

\begin{prob}\label{problem n=2}
Is the action 
$$\rho_{\epsilon}: GT(\mathbb K)\to  \mathrm{Aut}\ \widehat{{\mathbb K}[{\mathcal{SL}}_{\epsilon}]}$$
(with $\epsilon=(\epsilon_1,\epsilon_2)\in\{\uparrow, \downarrow\}^2$)
on proalgebraic $2$-string links 
faithful?
\end{prob}

\subsection{Proalgebraic knots}
We discuss a non-trivial grading on proalgebraic knots 
induced by $GT(\mathbb K)$-action.
In \cite{BLT} the image of the trivial knot, the unknot, 
under the Kontsevich isomorphism \eqref{Kontsevich's isomorphism} is explicitly determined.
In this subsection, we  work in an opposite direction.
That is, we explicitly calculate the inverse image $\gamma_0$ of the unit,
the trivial (chordless) chord diagram,  under the Kontsevich isomorphism.
It is explicitly described as combinations of two-bridge  knots (Theorem \ref{inverse image theorem}).
We also show that the invariant space of  proalgebraic knots under the  $GT(\mathbb K)$-action is one-dimensional generated by the element $\gamma_0$
(Theorem \ref{invariant subspace theorem}).

The following lemma is a knot analogue 
of  Proposition \ref{prop:twisting=our action}.

\begin{lem}\label{twisting is our action on knots}
For any $\sigma\in GRT(\mathbb K)$ 
\begin{equation}\label{twisting=our action on knots}
    \rho_{0}(\sigma)(D)= \theta_{0}^{ \varDelta(\sigma)}(D)
\end{equation}
holds for all 
$D\in \widehat{{\mathbb K}[{\mathcal{IK}}]}$. 
\end{lem}

\begin{proof}
It can be proven in a completely same way
to the proof of Proposition \ref{prop:twisting=our action}.
\end{proof} 

We may say that the following is an analogue of 
Theorem \ref{inner automorphism theorem on chord diagrams} and
\ref{inner automorphism theorem on tangles}
for proalgebraic knots. 

\begin{prop}\label{Gm-action on knots}
(1)
The $GRT(\mathbb K)$-action $\rho_0$ constructed in \eqref{GRT to Aut KIK}
on the algebra 
$\widehat{{\mathbb K}[{\mathcal{IK}}]}$
of infinitesimal knots
actually factors through $\mathbb{G}_m(\mathbb K)$($=\mathbb{K}^\times$)-action.
Namely the kernel of the action is its unipotent part $GRT_1(\mathbb K)$.

(2)
The $GT(\mathbb K)$-action $\rho_0$ constructed in \eqref{GT to Aut KK}
on the algebra $\widehat{{\mathbb K}[{\mathcal K}]}$ of proalgebraic knots
actually factors through $\mathbb{G}_m(\mathbb K)$($=\mathbb{K}^\times$)-action.
Namely the kernel of the action is its unipotent part $GT_1(\mathbb K)$.
\end{prop}

\begin{proof}
(1)
It is obtained from a combination of \eqref{twisting=our action on knots}
and \eqref{twisting=inner conjugation} 
because we have $s(\orientedcircle)=t(\orientedcircle)=\emptyset$.

(2)
The proof is almost same to the proof of  Theorem \ref{inner automorphism theorem on tangles}.
It can be derived from the above claim in this proposition and
Proposition \ref{three torsor maps}.
\end{proof}

As a corollary of the above proposition, we obtain a non-trivial decomposition of knots below.

\begin{cor}\label{decomposition of knots}
Each oriented knot $K$ admits a canonical decomposition
\begin{equation}\label{grading decomposition of knots}
K =K_0+K_1+K_2+\cdots
\end{equation}
in $\widehat{{\mathbb K}[{\mathcal K}]}$
such that
$$\sigma(K_m)=\lambda^m\cdot K_m $$
holds  for $\sigma=(\lambda,f)\in GT(\mathbb K)$ and $m\geqslant 0$.
\end{cor}

\begin{proof}
It is a direct corollary of Proposition \ref{Gm-action on knots}.(2).
By the representation theory of ${\mathbb G}_m$,
the completed vector space
$\widehat{{\mathbb K}[{\mathcal K}]}$ is decomposed into the product of eigenspaces
$V_a$ ($a\in\mathbb Z$) where $\mathbb G_m$ acts as a multiplication of  $a$-th power:
\begin{equation}\label{V=KK}
\widehat{{\mathbb K}[{\mathcal K}]}=\prod_{i\geqslant 0}V_i.
\end{equation}
By our construction,
the decomposition is compatible its filtration $\{\mathcal K\}_{N=0}^\infty$ , i.e.
$
\prod_{i\geqslant N}V_i=\mathcal K_N.
$
We also have  $\dim V_a<\infty$ for all $a\in\mathbb Z$ and actually
$\dim V_a=0$ for $a<0$.
Thus we get the claim.\qed   \\
{\it Another Proof.}
Again by the representation theory of ${\mathbb G}_m$,
Proposition \ref{Gm-action on knots}.(1) implies that 
the completed vector space 
$\widehat{{\mathbb K}[{\mathcal{IK}}]}$, hence
$\widehat{\mathcal{CD}}({\orientedcircle})$,
is decomposed into the product of eigenspaces $W_a$ ($a\in\mathbb Z$)
where $\mathbb G_m$ acts as a multiplication of  $a$-th power:
\begin{equation}
\widehat{\mathcal{CD}}({\orientedcircle})=\prod_{i\geqslant 0}W_i.
\end{equation}
Since  the $\mathbb G_m$-action is compatible with the grading
$\widehat{\mathcal{CD}}({\orientedcircle})=\widehat{\oplus}_{m=0}^\infty{\mathcal{CD}^m}({\orientedcircle})$
where ${\mathcal{CD}}^m({\orientedcircle})$ is the $\mathbb K$-linear space
spanned by the chord diagrams with $m$-chords,
we have 
\begin{equation}\label{W=CD}
W_{-m}={\mathcal{CD}}^m({\orientedcircle})
\end{equation}
for $m\geqslant 0$.
Thus $\dim W_a<\infty$ for all $a\in\mathbb Z$ and actually $\dim W_{a}=0$ for $a>0$.
Let $p=(1,\varphi)\in M_1(\mathbb K)$.
Then we have 
an isomorphism
$$r_p:GT(\mathbb K) \simeq GRT(\mathbb K)$$
by Proposition \ref{GRT-GT-bitorsor} and
an isomorphism
$$\rho_0(p):\widehat{{\mathbb K}[{\mathcal{K}}]}\simeq
\widehat{\mathcal{CD}}(\orientedcircle)$$
by Proposition \ref{KT to KIT}.(4).
By using Proposition \ref{three torsor maps},
we can show that
\begin{equation}\label{V=W}
V_a:=\rho_0(p)^{-1}(W_{-a})
\end{equation}
is the eigenspace where 
$\mathbb G_m$ acts as a multiplication of $a$-th power and
actually the space is invariant under any choice of $p=(1,\varphi)\in M_1(\mathbb K)$.
\end{proof}

We note that in the decomposition \eqref{V=KK}, we have
$
\dim V_m=\dim {\mathcal{CD}}^m({\orientedcircle}).
$
for all $m\geqslant 0$.

\begin{rem}
Let $\iota_0=(-1,1)\in GT(\mathbb K)$.
Since it is an involution, the space $\widehat{{\mathbb K}[{\mathcal K}]}$
is divided into two eigenspaces $V_+$ and $V_-$
where the action of $\iota_0$ is given by the multiplication by $1$ and $-1$
respectively.
So an each oriented knot $K$ is decomposed as 
$$K=K_++K_-\in \widehat{{\mathbb K}[{\mathcal K}]}
\qquad \text{with}\quad K_\pm\in V_\pm.$$
Since $\iota_0(K)$ is nothing but the mirror image $\bar K$ of $K$, we have
$K_+=\frac{1}{2}(K+\bar K)$ and 
$K_-=\frac{1}{2}(K-\bar K).$
We note that 
in terms of the decomposition \eqref{grading decomposition of knots}
they are expressed as
$$K_{+}=\sum_{i\geqslant 0} K_{2i}  \quad \text{and} \quad
K_{-}=\sum_{i\geqslant 0} K_{2i+1}.$$
\end{rem}

Our first theorem in this subsection is an explicit presentation 
of  the proalgebraic knot 
whose Kontsevich invariant is trivial.
That is, we explicitly calculate the inverse image of
the unit, 
the trivial (chordless) chord diagram, under the Kontsevich isomorphism 
$I:\widehat{{\mathbb C}[{\mathcal K}]}\simeq\widehat{\mathcal{CD}}(\orientedcircle)$
given in \eqref{Kontsevich's isomorphism}.

\begin{thm}\label{inverse image theorem}
The inverse image $I^{-1}(e)$ of the unit  $e\in \widehat{\mathcal{CD}}(\orientedcircle)$
($\simeq\widehat{{\mathbb K}[{\mathcal{IK}}]}$), 
under Kontsevich's isomorphism $I$ 
is explicitly  given by 
\begin{equation}\label{gamma0}
\gamma_0:=\orientedcircle-{c_0}+c_0\sharp c_0-c_0\sharp c_0\sharp c_0
+c_0\sharp c_0\sharp c_0\sharp c_0-\cdots  
\in\widehat{{\mathbb C}[{\mathcal K}]}
\end{equation}
where $\sharp$ is the connected sum and 
$c_0\in \widehat{{\mathbb C}[{\mathcal K}]}$ is given below (see also Figure \ref{c0-intro}):
\begin{align}\label{c0}
c_0:=\underset{ k_m>1}{\sum_{m,k_1,\dots,k_m\in {\mathbb N}}}&
(-1)^m\frac{\zeta^\mathrm{inv}(k_1,\dots,k_m)}{(2\pi\sqrt{-1})^{k_1+\cdots+k_m}}\cdot 
\Bigl\{
a_{0,0}^{\opannihilation}\cdot a_{2,0}^{\downarrow\uparrow\opannihilation} \cdot
(\log\sigma_2^2)^{k_m-1}\cdot (\log\sigma_3^2)\cdot  \\
&
(\log\sigma_2^2)^{k_{m-1}-1}\cdot (\log\sigma_3^2) \cdots 
\cdots (\log\sigma_2^2)^{k_1-1}\cdot (\log\sigma_3^2)
\cdot c_{1,1}^{\downarrow\opcreation\uparrow}\cdot c_{0,0}^{\creation}\Bigr\}. \notag
\end{align}
\begin{figure}[h]
\begin{center}
         \begin{tikzpicture}

\draw  (-2,2) node{{\huge $c_0=\underset{ k_m>1}{\sum\limits_{m,k_1,\dots,k_m\in {\mathbb N}}}
\frac{(-1)^m\zeta^\mathrm{inv}(k_1,\dots,k_m)}{(2\pi\sqrt{-1})^{k_1+\cdots+k_m}}\cdot$}};

                     \draw[->] (2.3,-3.2)  arc (180:360:0.9);
                   \draw[<-] (2.9,-2.8)  arc (180:360:0.3);

                     \draw[->] (2.9,-2.8)--(2.9,-1.9);
                     \draw[-] (3.5,-2.8)--(3.5,-2.6);
                    \draw[->]  (4.1,-3.2)--(4.1,-2.6);

                      \draw (3.2,-2.6) rectangle (4.4,-2.1);
                     \draw (3.8,-2.4) node{$log \ \sigma_3^2$};

                     \draw[-]  (2.9,-2.1)--(2.9,-1.9) (3.5,-2.1)--(3.5,-1.9) ;
                \draw[->] (4.1,-2.1)--(4.1,0.2) ;

                      \draw (2.6,-1.9) rectangle (3.8,-1.4);
                     \draw (3.2,-1.7) node{$log \ \sigma_2^2$};
                     \draw[-]  (2.9,-1.4)--(2.9,-1.2) (3.5,-1.4)--(3.5,-1.2) ;

                     \draw[dotted] (2.9,-1.2)--(2.9,-0.8) (3.5,-1.2)--(3.5, -0.8)  ;

                     \draw[-]  (2.9,-0.8)--(2.9,-0.6) (3.5,-0.8)--(3.5,-0.6) ;
                      \draw (2.6,-0.6) rectangle (3.8,-0.1);
                     \draw (3.2,-0.4) node{$log \ \sigma_2^2$};

    \draw[decorate,decoration={brace,mirror}] (4.5,-1.8) -- (4.5,-0.2) node[right, midway]{$k_1-1$};

                     \draw[-]  (3.5,-0.1)--(3.5,0.2)  ;
                      \draw (3.2,0.2) rectangle (4.4,0.7);
                     \draw (3.8,0.4) node{$log \ \sigma_3^2$};

                     \draw[->]  (2.9,-0.1)--(2.9,0.9);
                     \draw[<-]   (3.5,0.7)--(3.5,0.9);
                     \draw[->]   (4.1,0.7)--(4.1,0.9);

                     \draw[dotted] (2.9,0.9)--(2.9,2.4) (3.5,0.9)--(3.5, 2.4)  (4.1,0.9)--(4.1, 2.4);

                     \draw[->] (2.9,2.4)--(2.9,3.3) ;
                     \draw[<-] (3.5,2.4)--(3.5,2.6)  ;
                     \draw[->] (4.1,2.4)--(4.1,2.6) ;

                      \draw (3.2,2.6) rectangle (4.4,3.1);
                     \draw (3.8,2.8) node{$log \ \sigma_3^2$};
                     \draw[-]  (3.5,3.1)--(3.5,3.3) ;

                   \draw[->] (4.1,3.1)--(4.1,5.3);

                      \draw (2.6,3.3) rectangle (3.8,3.8);
                     \draw (3.2,3.5) node{$log \ \sigma_2^2$};
                     \draw[-]  (2.9,3.8)--(2.9,4)  (3.5,3.8)--(3.5,4) ;

                     \draw[dotted]  (2.9,4)--(2.9,4.4)  (3.5,4)--(3.5,4.4) ;

                     \draw[-]  (2.9,4.4)--(2.9,4.6)  (3.5,4.4)--(3.5,4.6) ;
                      \draw (2.6,4.6) rectangle (3.8,5.1);
                     \draw (3.2,4.8) node{$log \ \sigma_2^2$};
                     \draw[-]   (3.5,5.1)--(3.5,5.4) ;

    \draw[decorate,decoration={brace,mirror}] (4.5,3.4) -- (4.5,5) node[right, midway]{$k_{m-1}-1$};

                     \draw[->]  (2.9,5.1)--(2.9,6.1);

                      \draw (3.2,5.4) rectangle (4.4,5.9);
                     \draw (3.8,5.6) node{$log \ \sigma_3^2$};
                     \draw[-]  (3.5,5.9)--(3.5,6.1)  ;

                   \draw[->] (4.1,5.9)--(4.1,8.1);

                      \draw (2.6,6.1) rectangle (3.8,6.6);
                     \draw (3.2,6.3) node{$log \ \sigma_2^2$};
                     \draw[-]  (2.9,6.6)--(2.9,6.8)  (3.5,6.6)--(3.5,6.8) ;

                     \draw[dotted]  (2.9,6.8)--(2.9,7.2)  (3.5,6.8)--(3.5,7.2) ;

                     \draw[-]  (2.9,7.2)--(2.9,7.4)  (3.5,7.2)--(3.5,7.4) ;

                      \draw (2.6,7.4) rectangle (3.8,7.9);
                     \draw (3.2,7.6) node{$log \ \sigma_2^2$};
                     \draw[-] (2.9,7.9)--(2.9,8.5) (3.5,7.9)--(3.5,8.1) ;

    \draw[decorate,decoration={brace,mirror}] (4.5,6.2) -- (4.5,7.8) node[right, midway]{$k_m-1$};

                     \draw[<-] (3.5,8.1)  arc (180:0:0.3);
                     \draw[<-] (2.3,8.5)  arc (180:0:0.3);

                     \draw[<-] (2.3,-3.2)--(2.3,8.5) ;

\draw[color=white, very thick]  (2.3,-3.2)--(4.3,-3.2) (2.3,-2.8)--(4.3,-2.8) (2.3,8.1)--(4.3,8.1) (2.3,8.5)--(4.3,8.5)  ;
       \end{tikzpicture}
\caption{$c_0$}
\label{c0-intro}
\end{center}
\end{figure}
\end{thm}

Here we define 
\begin{equation}\label{log-sigma}
\log\sigma_i^2=-\sum_{k=1}^\infty \frac{1}{k}\{1-\sigma_i^2\}^k\in\widehat{{\mathbb C}[P_4]}
\end{equation}
for $i=2,3$.
And  we define the {\it inversed MZV} $\zeta^{\mathrm{inv}}(k_1,\dots,k_m)$ to be the coefficient 
of $A^{k_m-1}B\cdots A^{k_1-1}B$ multiplied by $(-1)^m$
in the {\it inversed KZ-associator} $\varPhi_\mathrm{KZ}^{\mathrm{inv}}(A,B)\in \mathbb{R}\langle\langle A,B\rangle\rangle$,
which is the inverse of the KZ-associator
$\varPhi_\mathrm{KZ}(A,B)$ in \eqref{LM-formula} with respect to the multiplication \eqref{product of proalg-GRT}. Namely
\begin{align}\label{inv-KZ}
\varPhi_\mathrm{KZ}^{\mathrm{inv}}(A,B)=:1+
\underset{k_m>1}
{\underset{m,k_1,\dots,k_m\in{\mathbb N}}{\sum}}
(-1)^m & \zeta^{\mathrm{inv}}(k_1,\cdots,k_m)
A^{k_m-1}B\cdots A^{k_1-1}B \\ \notag
&+\text{(other terms)}
\end{align}
where $\varPhi_\mathrm{KZ}^{\mathrm{inv}}(A,B)$
is the series uniquely defined by
\begin{equation}\label{inv-KZ2}
\varPhi_\mathrm{KZ}^{\mathrm{inv}}\left(\varPhi_\mathrm{KZ}(A,B)\cdot
A\cdot\varPhi_\mathrm{KZ}(A,B)^{-1},B\right)
=\varPhi_\mathrm{KZ}(A,B)^{-1}.
\end{equation}

The inversed MZV's with small depth are calculated in
Example \ref{examples of zeta-inv}
and the first two terms of $\gamma_0$ is calculated in 
Example \ref{2terms}.

\begin{proof}
By our construction,
the degree $0$-part of the image $I(\orientedcircle)$ of the trivial knot (unknot) $\orientedcircle$ 
under the Kontsevich isomorphism $I$ given in \eqref{Kontsevich's isomorphism}
is
the trivial (chordless) chord diagram $e$.
Therefore on the decomposition $K=K_0+K_1+\cdots$ 
in \eqref{grading decomposition of knots}  for $K=\orientedcircle$,
we have
\begin{equation}\label{K=I-1(e)}
K_0=I^{-1}(e)
\end{equation}
by \eqref{V=W}.
To calculate $K_0$, we take ${\mathbb K}$ to be the polynomial algebra ${\mathbb C}[T^\pm]$ 
generated by  $T$ and $T^{-1}$ and
take $p\in M_1(\mathbb C)$ to be a specific element 
$p_\mathrm{KZ}=(1,\varphi_\mathrm{KZ})\in M(\mathbb C)$
with $\varphi_\mathrm{KZ}=\varPhi_\mathrm{KZ}\left(\frac{1}{2\pi\sqrt{-1}}A,\frac{1}{2\pi\sqrt{-1}}B\right)$.
We have $(T^{-1},1)\in GRT(\mathbb K)$.
By Proposition \ref{GRT-GT-bitorsor}, we obtain a unique element 
$(T,f_T)\in GT(\mathbb K(T))$
satisfying
\begin{equation*}
(T^{-1},1)\circ (1,\varphi_\mathrm{KZ})=(1,\varphi_\mathrm{KZ})\circ (T,f_T).
\end{equation*}
It can be read as
\begin{equation*}
\varphi_\mathrm{KZ}(TA,TB)=
f_T(\varphi_\mathrm{KZ}\cdot e^A\cdot\varphi_\mathrm{KZ}^{-1},e^B)\cdot\varphi_\mathrm{KZ}.
\end{equation*}
We get that $f_T$ belongs to $F_2({\mathbb C}[T])$. So
\begin{equation*}
f_T\left(\exp\{\varphi_\mathrm{KZ}\cdot A\cdot\varphi_\mathrm{KZ}^{-1}\},\exp\{B\}\right)
\equiv \varphi_\mathrm{KZ}^{-1} \pmod T.
\end{equation*}
By replacing $A$ and $B$ by $2\pi\sqrt{-1} A$ and $2\pi\sqrt{-1} B$ respectively,
we obtain 
\begin{equation*}
f_T(e^{2\pi\sqrt{-1}A},e^{2\pi\sqrt{-1}B})\equiv \varPhi_\mathrm{KZ}^{\mathrm{inv}} \pmod T.
\end{equation*}
It says that
\begin{equation*}\label{fT mod T}
f_T(\sigma_1^2,\sigma_2^2)\equiv
\varPhi_\mathrm{KZ}^{\mathrm{inv}} \left(\frac{\log\sigma_1^2}{2\pi\sqrt{-1}} \ ,\frac{\log\sigma_2^2}{2\pi\sqrt{-1}}\right) \pmod T,
\end{equation*}
where precisely it means that
\begin{equation*}
\mathrm{red}\left(f_T(\sigma_1^2,\sigma_2^2)\right)=
\mathrm{red}\left(\varPhi_\mathrm{KZ}^{\mathrm{inv}} \left(\frac{\log\sigma_1^2}{2\pi\sqrt{-1}} \ , \frac{\log\sigma_2^2}{2\pi\sqrt{-1}}\right)\right)
\end{equation*}
with the reduction map $\mathrm{red}:F_2(\mathbb C[T])\to F_2(\mathbb C)$ 
caused by putting $T=0$.
Therefore by Definition \ref{GT-action on proalgebraic tangles}
\begin{align*}
\Lambda_{f_T}^\uparrow=&a_{1,0}^{\uparrow\opannihilation} \cdot
f_T(\sigma_1^2,\sigma_2^2)
\cdot c_{0,1}^{\opcreation\uparrow} \\
\equiv& a_{1,0}^{\uparrow\opannihilation} \cdot\varPhi_\mathrm{KZ}^{\mathrm{inv}} \left(\frac{\log\sigma_1^2}{2\pi\sqrt{-1}} \ , \frac{\log\sigma_2^2}{2\pi\sqrt{-1}}\right)
\cdot c_{0,1}^{\opcreation\uparrow}   \pmod T\\
=&a_{1,0}^{\uparrow\opannihilation}\cdot c_{0,1}^{\opcreation\uparrow}+
\underset{ k_m>1}{\sum_{m,k_1,\dots,k_m\in {\mathbb N}}}
(-1)^m\frac{\zeta^\mathrm{inv}(k_1,\dots,k_m)}{(2\pi\sqrt{-1})^{k_1+\cdots+k_m}}\cdot 
\Bigl\{
a_{1,0}^{\uparrow\opannihilation} 
\cdot
(\log\sigma_1^2)^{k_m-1}\cdot (\log\sigma_2^2) \\
&\qquad\qquad\qquad\qquad
\cdot(\log\sigma_1^2)^{k_{m-1}-1}\cdot (\log\sigma_2^2) \cdots 
\cdots (\log\sigma_1^2)^{k_1-1}\cdot (\log\sigma_2^2)\cdot
c_{0,1}^{\opcreation\uparrow}
\Bigr\} .
\end{align*}
The third equality follows from \eqref{inv-KZ}.
Here we note that the non-admissible terms, 
the \lq other terms' in \eqref{inv-KZ}, all vanish by (T6).
By the isomorphism \eqref{cl on proalgebraic knots}, we have
$$
\mathrm{cl}(\Lambda_{f_T}^\uparrow)
=a_{0,0}^{\opannihilation}\cdot (\downarrow\otimes \Lambda_{f_T}^\uparrow)\cdot
c_{0,0}^{\creation}
\equiv \orientedcircle+c_0   \pmod T
$$
(for $c_0$, see Figure \ref{c0-intro}).
Whence, by taking inverse of both, we have
\begin{equation}\label{cl=gamma}
\mathrm{cl}(\nu_{f_T}^\uparrow)\equiv \gamma_0   \pmod T
\end{equation}
(for $\gamma_0$, see \eqref{gamma0-intro}).
We note that the infinite summation on the right hand side of \eqref{gamma0}
converges because $c_0^{\sharp n}\in {\mathcal K}_n$ for  $n\geqslant 1$.
By Corollary \ref{decomposition of knots},
$K_0$ is obtained  by evaluating $(T,f_T)(\orientedcircle)$
at $T=0$, i.e.
\begin{equation}
(T,f_T)(\orientedcircle)\equiv K_0 \pmod T.
\end{equation}
Since $\orientedcircle=a^{\opannihilation}_{0,0}\cdot c_{0,0}^{\creation}$,
the action on $\orientedcircle$ is calculated to be
\begin{equation}\label{Tf=cl}
(T,f_T)(\orientedcircle)=a^{\opannihilation}_{0,0}\cdot (\downarrow\otimes\nu_{f_T}^\uparrow)\cdot c_{0,0}^{\creation}=\mathrm{cl}( \nu_{f_T}^\uparrow).
\end{equation}
By \eqref{cl=gamma}--\eqref{Tf=cl}, we have
\begin{equation}
\gamma_0= K_0
\end{equation}
because both $\gamma_0$ and $K_0$ are independent of $T$. 
Finally we obtain
$$\gamma_0=I^{-1}(e)$$
by \eqref{K=I-1(e)}.
\end{proof}

It might be worthy to mention that our $c_0$ is given by a linear combination of two bridge knots
and our $\gamma_0$ is given by a linear combination of connected sums of
two bridge knots.

\begin{rem}
We remind that the image $I(\orientedcircle)$ of the unit $\orientedcircle$,
the trivial knot, 
under the Kontsevich isomorphism $I$ 
was calculated in \cite{BLT},
which may be be regarded as a calculation 
in an opposite direction to our theorem.
\end{rem}

By the definition of inversed MZV's, \eqref{inv-KZ} and \eqref{inv-KZ2},
they are given by polynomial combinations
of MZV's.

\begin{eg}\label{examples of zeta-inv}
(1)
For $n>1$,
$$
\zeta^\mathrm{inv}(n)=-\zeta(n).
$$

(2)
For $a>0$ and $b>1$,
\begin{align*}
\zeta^\mathrm{inv}(a,b)=&\zeta(a)\zeta(b)-\zeta(a,b)
+\sum_{i=0}^{a-2}(-1)^i\binom{i+b-1}{i}\zeta(b+i)\zeta(a-i) \\
&+(-1)^a\sum_{j=0}^{b-2}\binom{j+a-1}{j}\zeta(b-j)\zeta(a+j).
\end{align*}
\end{eg}

Here is a brief computation of our $\gamma_0$.

\begin{eg}\label{2terms}
\begin{equation*}
\gamma_0\equiv
\orientedcircle-\frac{1}{24}\{\orientedcircle-3_1^l \} \pmod{\mathcal K_4}.
\end{equation*}
Here $3_1^l$ stands for the left trefoil knot.
\end{eg}

\begin{prop}
$\gamma_0\in\widehat{{\mathbb Q}[{\mathcal K}]}$.
\end{prop}

\begin{proof}
In the proof of above theorem, there is no specific reason to choose
$(1,\varphi_\mathrm{KZ})\in M_1(\mathbb C)$.
Take $(\mu, \varphi)\in M(\mathbb K)$ to be any associator with an expansion
\begin{align*}
\varphi(A,B)=1+
\underset{k_m>1}
{\underset{m,k_1,\dots,k_m\in{\mathbb N}}{\sum}}
(-1)^m & c(k_1,\cdots,k_m)
A^{k_m-1}B\cdots A^{k_1-1}B  
+\text{(regularized terms)}.\notag
\end{align*}
Then $(1,\varphi(A/\mu,B/\mu))\in M_1(\mathbb K)$.
By using $(1,\varphi(A/\mu,B/\mu))\in M_1(\mathbb K)$ 
instead of $(1,\varphi_\mathrm{KZ})\in M_1(\mathbb C)$
in the proof of Theorem \ref{inverse image theorem},
we obtain an explicit formula of $c_0$,
which is nothing but the replacement of 
$2\pi\sqrt{-1}$ by $\mu\in {\mathbb K}^\times$ and $\zeta(k_1,\dots,k_m)$ by
$c(k_1,\dots,k_m)$  in \eqref{c0}. 
Thus our claim is obtained by
particularly taking a rational associator, an element of $M(\mathbb Q)$,
whose existence is guaranteed in \cite{Dr}.
\end{proof}

Our second theorem in this subsection is on the invariant subspace of $\widehat{{\mathbb K}[{\mathcal K}]}$
under our $GT(\mathbb K)$-action  (cf. \eqref{GT to Aut KK}).

\begin{thm}\label{invariant subspace theorem}
(1)
The 
invariant subspace $V_0$ of $\widehat{{\mathbb K}[{\mathcal K}]}$
under our $GT(\mathbb K)$-action 
is $1$-dimensional and actually generated by $\gamma_0$.

(2)
On the decomposition $K=K_0+K_1+\cdots$ given in \eqref{grading decomposition of knots},
$$K_0=\gamma_0$$
holds for any oriented knot $K$.
\end{thm}

\begin{proof}
(1)
Since ${\mathcal{CD}}^0({\orientedcircle})$ is $1$-dimensional generated by the 
trivial (chordless) chord diagram $e$,
the $\mathbb K$-linear space $V_0$ in \eqref{V=KK} is $1$-dimensional
by \eqref{W=CD} and \eqref{V=W}.
Thus it is enough to show that 
$\rho_0(p)^{-1}(e)=\gamma_0$
for a particular choice of $p=(1,\varphi)\in M_1(\mathbb K)$.
It is obtained by Theorem \ref{inverse image theorem}
because  $\rho_0(p)=I$ for $p=p_\mathrm{KZ}$ (actually it holds for any $p\in M_1(\mathbb K)$).

(2)
By our construction, the degree $0$-part of the image $I(K)$ is always equal to 
$e \in{\mathcal{CD}}^0({\orientedcircle})$.
So  $I(K_0)=e$, by  \eqref{V=W}.
Then by the  arguments given above,
we get $K_0=\gamma_0$.
\end{proof}

Though our knot $\gamma_0$ is not  equal to the trivial knot, 
the unknot $\orientedcircle$ ,
we may say that it is  the \lq trivial' knot
in a particular sense.

\begin{rem}\label{TM-structure}
We saw in Remark \ref{MTM-structure}  particularly that
the space  $\widehat{{\mathbb K}[{\mathcal{K}}]}$ of proalgebraic knots
carries a structure of a mixed Tate (pro-)motive over
${\rm Spec}\; {\mathbb Z}$.
But Proposition \ref{Gm-action on knots} tells that
it falls to infinite direct sum of Tate motives ${\mathbb Q}(n)$
for $n\geqslant 0$.
Our $\gamma_0$ is the generator of the degree $0$-part $V_0$.
Finding a basis of the degree $n$-part $V_n$ for another $n$, 
such as the  basis which is dual to the Vassiliev invariants listed
in \cite{CDM} Table 3.2,
is worthy to calculate.
\end{rem}

\thanks{
{\it Acknowledgements}.
Part of the paper was written at Max Planck Institute for Mathematics.
The author also thanks the institute for its hospitality.
The referees' efforts to make this paper better is gratefully acknowledged.
This work was supported by 
Grant-in-Aid for Young Scientists (A) 24684001.
}


\end{document}